\documentclass[12pt]{article}
\usepackage{amsmath,amssymb,amsthm,amsfonts,mathrsfs}
\usepackage{appendix}
\usepackage{array,enumitem,graphics,xcolor,yfonts,colonequals}
\usepackage{stmaryrd}
\usepackage[alphabetic]{amsrefs}
\usepackage[all,cmtip]{xy}
\usepackage{tikz-cd}
\usepackage[colorlinks,anchorcolor=blue,citecolor=blue,linkcolor=blue,urlcolor=blue,bookmarksopen=true]{hyperref}
\usepackage{comment}
\usepackage{mathtools}
\usepackage[margin=1in]{geometry}

\usepackage{titlesec}
\titleformat{\section}
  {\normalfont\large\bf}{\thesection}{1em}{}
\titleformat{\subsection}[runin]
  {\normalfont\normalsize\bf}{\thesubsection}{1em}{}

\usepackage{fancyhdr}
\fancypagestyle{titlepage}
{
	\fancyhf{}

	\fancyfoot[l]{
	\href{https://mathscinet.ams.org/mathscinet/msc/msc2020.html}
	}
}

\AtBeginDocument{%
	\def\MR#1{}
}

\newcommand{\bZ}{\mathbb{Z}}    
\newcommand{\bQ}{\mathbb{Q}}    
\newcommand{\bR}{\mathbb{R}}    
\newcommand{\bC}{\mathbb{C}}    
\newcommand{\bH}{\mathbb{H}}    

\newcommand{\GL}{\mathrm{GL}}           
\newcommand{\SL}{\mathrm{SL}}           
\newcommand{\SO}{\mathrm{SO}}           
\newcommand{\PSL}{\mathrm{PSL}}         
\newcommand{\uO}{\mathrm{O}}            
\newcommand{\Aut}{\operatorname{Aut}}   
\newcommand{\Ext}{\operatorname{Ext}}
\newcommand{\Hom}{\operatorname{Hom}}   

\newcommand{\bP}{\mathbb{P}}    
\newcommand{\cH}{\mathcal{H}}   
\newcommand{\cF}{\mathcal{F}}   
\newcommand{\cC}{\mathcal{C}}   

\newcommand{\cO}{\mathcal{O}}   
\newcommand{\cE}{\mathcal{E}}   
\newcommand{\Coh}{\mathrm{Coh}} 
\newcommand{\Pic}{\operatorname{Pic}}	
\newcommand{\NS}{\operatorname{NS}}     
\newcommand{\Amp}{\mathrm{Amp}}         
\newcommand{\ch}{\mathrm{ch}}           
\newcommand{\cl}{\mathrm{cl}}           

\newcommand{\Db}{\mathrm{D^b}}  
\newcommand{\sD}{\mathscr{D}}   
\newcommand{\sP}{\mathscr{P}}   
\newcommand{\sA}{\mathscr{A}}   
\newcommand{\sT}{\mathscr{T}}   
\newcommand{\sF}{\mathscr{F}}   
\newcommand{\Stab}{\mathrm{Stab}} 
\newcommand{\std}{\mathrm{std}} 

\newcommand{\cA}{\mathcal{A}}   
\newcommand{\mukaiH}{\widetilde{H}} 
\newcommand{\cN}{\mathcal{N}}   
\newcommand{\cP}{\mathcal{P}}   
\newcommand{\cQ}{\mathcal{Q}}   

\newcommand{\cI}{\mathcal{I}}   

\newcommand{\cat}{\mathrm{cat}}     
\newcommand{\poly}{\mathrm{poly}}   
\newcommand{\tr}{\mathrm{tr}}       

\newcommand{\git}{\mathbin{
  \mathchoice{\backslash\mkern-6mu\backslash}
    {\backslash\mkern-6mu\backslash}
    {\backslash\mkern-5mu\backslash}
    {\backslash\mkern-5mu\backslash}}}

\newcommand{\frn}{w_n}
\newcommand{\red}{\mathrm{red}}

\newtheorem*{thm*}{Theorem}
\newtheorem*{prop*}{Proposition}
\newtheorem*{cor*}{Corollary}
\newtheorem{thm}{Theorem}[section]
\newtheorem{prop}[thm]{Proposition}
\newtheorem{propdef}[thm]{Proposition-Definition}
\newtheorem{cor}[thm]{Corollary}
\newtheorem{lemma}[thm]{Lemma}

\numberwithin{equation}{section}

\theoremstyle{definition}
\newtheorem{defn}[thm]{Definition}
\newtheorem{eg}[thm]{Example}
\newtheorem{ques}[thm]{Question}
\newtheorem{rmk}[thm]{Remark}

\begin{document}


\title{Categorical Nielsen realization problems for generic K3~surfaces}
\author{Yu-Wei Fan
    \and Kuan-Wen Lai}
\date{}

\newcommand{\ContactInfo}{{
\bigskip\footnotesize

\bigskip
\noindent Y.-W.~Fan,
\textsc{Center for Mathematics and Interdisciplinary Sciences, Fudan University, \\ 
Shanghai 200433, China}
\par\nopagebreak
\smallskip
\noindent\textsc{Shanghai Institute for Mathematics and Interdisciplinary Sciences (SIMIS), \\  
Shanghai 200433, China}
\par\nopagebreak
\noindent\textsc{Email:} \texttt{yuweifanx@gmail.com}

\bigskip
\noindent K.-W.~Lai,
\textsc{Department of Smart Computing and Applied Mathematics \\
Tunghai University \\
No.~1727, Sec.~4, Taiwan~Blvd., Xitun~Dist., Taichung~City 407224, Taiwan}
\par\nopagebreak
\smallskip
\noindent\textsc{National Center for Theoretical Science \\
No.~1, Sec.~4, Roosevelt~Rd., Taipei~City 106319, Taiwan}
\par\nopagebreak
\noindent\textsc{Email:} \texttt{kwlai@thu.edu.tw}
}}

\maketitle
\thispagestyle{titlepage}

\begin{abstract}
For a complex projective K3~surface of Picard number one, we prove that every nontrivial finite subgroup of autoequivalences has order two and compute the number of conjugacy classes of such subgroups. We obtain similar formulas for the subgroups which are finite up to shifts, and deduce that such a K3~surface has an associated cubic fourfold if and only if it has an autoequivalence of order three modulo shifts. These results are proved by showing that every such subgroup fixes a Bridgeland stability condition up to the $\mathbb{C}$-action. We also establish similar existence results for curves, twisted abelian surfaces, generic twisted K3~surfaces, and standard autoequivalences of surfaces.
\end{abstract}

\setcounter{tocdepth}{3}
\tableofcontents


\addtocontents{toc}{\protect\setcounter{tocdepth}{1}}
\section{Introduction}
\label{sec:intro}

Let $X$ be a complex K3~surface whose N\'eron--Severi group is generated by an ample class of self-intersection $2n$. It is well known that such a surface has essentially no nontrivial automorphisms. More precisely,
$$
    \Aut(X) = \begin{cases}
    \{\mathrm{id}\} & \text{if}\quad n\geq2 \\
    \left<\iota\right>\cong\bZ_2 & \text{if}\quad n=1
\end{cases}
$$
where, in the latter case, $\iota$ is the covering involution of the double cover $X\to\bP^2$ \cite{Nik81}*{Corollary~10.1.3}; see also \cite{Huy16_K3Lect}*{Corollary~15.2.12}. Now consider $\Db(X)$, the bounded derived category of coherent sheaves on $X$. The group $\Aut(X)$ embeds naturally into the group $\Aut(\Db(X))$ of $\bC$-linear exact autoequivalences, and one may ask whether the latter contains any other nontrivial elements of finite order. For quartic K3~surfaces, namely when $n=2$, one can consider
$$
    \Theta\colonequals(-\otimes\cO_X(1))\circ T_{\cO_X}
$$
where $T_{\cO_X}$ is the spherical twist along $\cO_X$. It is known that $\Theta^4 = [2]$ (\cite{CK08},\cite{BFK}*{Proposition~5.8}), so $\Theta^2[-1]$ is an involution. In general, constructing finite-order autoequivalences is a nontrivial task.

Throughout the paper, all varieties are assumed to be smooth, complex, and projective, unless otherwise specified. For a K3~surface of Picard number one, we define its degree to be the self-intersection of the ample generator of its N\'eron--Severi group. One of our main results is the following classification of finite subgroups of autoequivalences.

\begin{thm}
\label{mainthm:finite-subgp}
Let $X$ be a K3~surface of Picard number one and degree $2n$. Then every nontrivial finite subgroup of $\Aut(\Db(X))$ has order~$2$ and is generated by an anti-symplectic involution. Moreover, the number of conjugacy classes of these subgroups equals
$$\begin{cases}
    1 & \text{if}
    \quad n = 1,2, \\
    0 & \text{if}
    \quad n \text{ is divisible by } 4 \text{ or a prime } p\equiv 3\pmod{4}, \\
    2^{a-1} & \text{if}
    \quad n = 2^{\ell}p_1^{k_1}\cdots p_a^{k_a},
    \quad \ell\in\{0,1\},
    \quad a, k_i\geq1,
    \quad p_i\equiv 1\pmod{4}\text{ are primes}.
\end{cases}$$
\end{thm}

It is also natural to classify subgroups of autoequivalences that are finite modulo shift functors. Suppose, for instance, that a K3~surface $X$ has an associated cubic fourfold $Y\subseteq\bP^5$, so that $\Db(X)$ is equivalent to the Kuznetsov component of $\Db(Y)$. Then $\Aut(\Db(X))/\bZ[2]$ contains an element of order~$3$ by \cite{Kuz04}*{Lemma~4.2}, where $\bZ[2]\subseteq\Aut(\Db(X))$ denotes the infinite cyclic subgroup generated by the double shift functor $[2]$.
In \cite{Huy25}*{Remark~7.1}, Huybrechts asks whether the converse holds. We completely classify finite subgroups of $\Aut(\Db(X))/\bZ[2]$, derive formulas for the numbers of their conjugacy classes, and thereby answer this question affirmatively for K3~surfaces of Picard number one.

\begin{thm}
\label{mainthm:finite-subgp-mod-2}
Let $X$ be a K3~surface of Picard number one and degree $2n$. Then every maximal finite subgroup $G\subseteq\Aut(\Db(X))/\bZ[2]$ is isomorphic to $G_s\times\bZ_2[1]$, where $G_s\subseteq G$ is the symplectic subgroup and
$$
    G_s 
    \cong\begin{cases}
        \bZ_6 & \text{if}\quad n=1, \\
        \bZ_4 & \text{if}\quad n=2, \\
        0,\; \bZ_2, \;\text{or }\; \bZ_3 & \text{if}\quad n\geq 3.
    \end{cases}
$$
Here $\bZ_2[1]\subseteq\Aut(\Db(X))/\bZ[2]$ is the subgroup of order~$2$ generated by the class of $[1]$.

For $n=1,2$, there is a unique conjugacy class of maximal finite subgroups. For $n\geq 3$, the numbers of conjugacy classes of subgroups isomorphic to $\bZ_2\times\bZ_2[1]$ and $\bZ_3\times\bZ_2[1]$ are given, respectively, by Theorem~\ref{mainthm:finite-subgp} and by
$$\begin{cases}
    0 & \text{if}
    \quad n = 3, \\
    0 & \text{if}
    \quad n \text{ is divisible by } 9 \text{ or a prime } p\equiv 2\pmod{3}, \\
    2^{a-1} & \text{if}
    \quad n = 3^{\ell}p_1^{k_1}\cdots p_a^{k_a},
    \quad \ell\in\{0,1\},
    \quad a, k_i\geq1,
    \quad p_i\equiv 1\pmod{3}\text{ are primes}.
\end{cases}$$
\end{thm}

For K3~surfaces of Picard number one and degree at least $4$, we obtain the following criterion for the existence of an associated cubic fourfold. A complete list of equivalent criteria is given in Proposition~\ref{prop:cubic}, while the degree~$2$ case is discussed in Remark~\ref{rmk:cubic_deg-2}.

\begin{prop}
\label{prop:cubic_partial}
A K3~surface $X$ of Picard number one and degree at least $4$ has an associated cubic fourfold if and only if $\Aut(\Db(X))/\bZ[2]$ contains an element of order~$3$.
\end{prop}

\subsection{Categorical Nielsen realization problems}

Theorems~\ref{mainthm:finite-subgp} and \ref{mainthm:finite-subgp-mod-2} arise from our study of categorical analogues of the \emph{Nielsen realization problem}, which is the starting point of this paper. In the classical version of the realization problem, one considers an oriented surface and asks whether every finite subgroup of its mapping class group fixes a point of its Teichm\"uller space. This problem was posed by Nielsen \cite{Nie32} and answered affirmatively by Kerckhoff \cite{Ker83}. In \cite{FL24}, Farb and Looijenga study various versions of the realization problem for K3~manifolds.

In light of the various connections between Teichm\"uller theory and the theory of stability conditions on triangulated categories \cites{GMN13,DHKK14,BS15,HKK17}, we formulate the following questions. Let $\sD$ be a triangulated category and $\Stab(\sD)$ be the space of stability conditions on $\sD$ introduced by Bridgeland \cite{Bri07}.
\begin{center}
    \it Does every finite subgroup $G\subseteq\Aut(\sD)$ fix a point of $\Stab(\sD)$?
\end{center}
The space $\Stab(\sD)$ carries a free $\bC$-action, and the double shift functor $[2]$ acts trivially on the quotient manifold $\Stab(\sD)/\bC$. Since shift functors have no classical counterpart in Teichm\"uller theory, it is also natural to consider the following modified question:
\begin{center}
    \it Does every finite subgroup $G\subseteq\Aut(\sD)/\bZ[2]$ fix a point of $\Stab(\sD)/\bC$?
\end{center}
In fact, an affirmative answer to the second question implies an affirmative answer to the first; see Lemma~\ref{lemma:q2impliesq1}.

In this paper, we give affirmative answers to both questions when $\sD$ is the bounded derived category of coherent sheaves on a curve, a twisted abelian surface, or a generic twisted K3~surface; see Propositions~\ref{prop:nielsen_curve} and \ref{prop:nielsen_twisted-ab-k3}. For arbitrary surfaces, we get affirmative answers when restricting to standard autoequivalences; see Proposition~\ref{prop:nielsen_stdauto-surf}. For K3~surfaces of Picard number one, we prove the following sharper statement:

\begin{thm}
\label{mainthm:nielsen_K3}
Let $X$ be a K3~surface, $\Stab^\dag(X)\subseteq\Stab(X)$ be the connected component containing geometric stability conditions, and $G\subseteq\Aut(\Db(X))/\bZ[2]$ be a subgroup. If $X$ has Picard number one, then
$$
    G\text{ is finite}
    \quad\Longleftrightarrow\quad
    G\text{ fixes a point of }\Stab^\dag(X)/\bC.
$$
For a K3~surface of arbitrary Picard number, the implication ``$\Longleftarrow$'' still holds.
\end{thm}

The implication ``$\Longleftarrow$'' is proved in Proposition~\ref{prop:gepner_auteq-isom} for arbitrary K3~surfaces. When~$X$ has Picard number one, the work of Bayer--Bridgeland \cite{BB17} implies that $\Stab^\dag(X)$ is simply connected and preserved by autoequivalences. Moreover, the submanifold
$$
    \Stab^\dag_\red(X)/\bC
    \subseteq
    \Stab^\dag(X)/\bC,
$$
consisting of $\bC$-orbits of \emph{reduced stability conditions}, is the universal cover of the hyperbolic upper half-plane with a discrete subset of points removed. It is therefore biholomorphic to the upper half-plane, and autoequivalences act on it through elements of its isometry group $\PSL(2,\bR)$. The implication ``$\Longrightarrow$'' is then established using the fact that every finite subgroup of $\PSL(2,\bR)$ fixes a point. We refer to Theorem~\ref{thm:nielsen_k3-pic1} for details.

We also obtain a finer result in Theorem~\ref{thm:dist-stabred}, which identifies each maximal subgroup of autoequivalences that is finite up to shifts with the isotropy group of a unique point on $\Stab^\dag_\red(X)/\bC$, and precisely locates the points whose isotropy groups arise this way. This phenomenon naturally parallels the correspondence between finite subgroups of the modular group $\PSL(2,\bZ)$ and its elliptic points in the upper half-plane. The key to applying this result to the classification of finite subgroups of autoequivalences is Lemma~\ref{lemma:only-Z2}, which shows that the subgroup of autoequivalences fixing a given stability condition is either trivial or isomorphic to $\bZ_2$.

\subsection{Gepner-type autoequivalences and beyond}

A stability condition representing a point of $\Stab(\sD)/\bC$ fixed by the action of a non-shift functor is called a \emph{Gepner-type stability condition}. This notion was introduced by Toda in the study of Gepner points of stringy K\"ahler moduli spaces of Calabi--Yau manifolds \cite{Tod14}. Motivated by this notion and the realization problems, we introduce \emph{Gepner-type autoequivalences} as follows.

\begin{defn}
\label{defn:gepner}
Let $\sD$ be a triangulated category. We say an autoequivalence $\Phi\in\Aut(\sD)$ is of \emph{Gepner type} if its action on $\Stab(\sD)/\bC$ has a fixed point, or equivalently, if there exists a stability condition $\sigma$ such that $\Phi(\sigma) = \sigma\cdot\lambda$ for some $\lambda\in\bC$. We refer to \eqref{eqn:C-action} for the definition of the $\bC$-action.

More generally, we say that a subgroup
$
    G\subseteq\Aut(\sD)/\bZ[m],
$
where $m\in\bZ$, is of Gepner type if there exists a point in $\Stab(\sD)/\bC$ fixed by all elements of $G$. For a stability condition~$\sigma$, the autoequivalences of Gepner type with respect to $\sigma$ form the group
$$
    \Aut(\sD,\,\sigma\cdot\bC) = \{
        \Phi\in\Aut(\sD)
        \mid
        \Phi(\sigma)\in\sigma\cdot\bC
    \}.
$$
We denote its subgroup consisting of elements fixing $\sigma$ by
$$
    \Aut(\sD,\sigma) = \{
        \Phi\in\Aut(\sD)
        \mid
        \Phi(\sigma) = \sigma
    \}.
$$
\end{defn}

Our study fits into the broader study of autoequivalences through their actions on the space of stability conditions. This paper focuses primarily on finite-order autoequivalences. In general, autoequivalences may be studied from a dynamical perspective, for instance, by considering whether they act as \emph{reducible} or \emph{pseudo-Anosov} transformations on a suitable compactification of the space of stability conditions (cf.~\cites{DHKK14,FFHKL21}). Such a dynamical classification has been established for elliptic curves by Kikuta--Koseki--Ouchi \cite{KKO22}. For K3~surfaces of Picard number one, we classify autoequivalences into four types according to their actions on the period domain of stability conditions: \emph{finite order}, \emph{$(-2)$-reducible}, \emph{$0$-reducible}, and \emph{pseudo-Anosov}. We also formulate a question about reducible autoequivalences that could potentially lead to a characterization of this classification in terms of categorical entropy; see Question~\ref{ques:entropy} and Proposition~\ref{conj:NTclassificationDyn}.

\subsection*{Organization of the paper}
In Section~\ref{sec:curve-std}, we answer the realization problems for curves and for standard autoequivalences on surfaces. In Section~\ref{sec:k3_gepner}, we study Gepner-type autoequivalences of K3~surfaces of arbitrary Picard number, and establish technical results to be used later. We also answer the realization problems for twisted abelian surfaces and generic twisted K3~surfaces in this section. In Section~\ref{sec:auteqMod2}, we analyze the autoequivalence group of a K3~surface of Picard number one via its action on the space of central charges. This reveals the structure of the group modulo even shifts and leads to the proof of Theorem~\ref{mainthm:finite-subgp-mod-2}. We conclude the section with criteria for the existence of associated cubic fourfolds. In Section~\ref{sec:catNielsen}, we answer the realization problems for K3~surfaces of Picard number one, prove a refined version of Theorem~\ref{mainthm:nielsen_K3}, and then apply it to prove Theorem~\ref{mainthm:finite-subgp}. In Section~\ref{sec:polychotomy}, we discuss a classification of autoequivalences and its connection to categorical entropy.

\subsection*{Acknowledgements}
This project was initiated in the fall of 2021 following a conversation between the second author and Daniel~Huybrechts, in which Huybrechts proposed a categorical analogue of the Nielsen realization problem for K3~surfaces. We are grateful to him for suggesting this direction and for many valuable discussions throughout the preparation of the paper. We thank Gebhard~Martin for his help in understanding the reference \cite{Fri28}. We also thank Arend~Bayer, Alexander~Kuznetsov, and Evgeny~Shinder for inspiring discussions and valuable comments on the paper. Finally, we thank the anonymous referees for their careful reading of the paper and for their many helpful comments and suggestions.

The initial draft of this paper was completed while the second author was supported by the ERC Synergy Grant HyperK (ID: 854361). The first author is currently supported by the National Natural Science Foundation of China (Grant No.~12401053). The second author is currently supported by the NSTC Research Grant 113-2115-M-029-003-MY3 and partially supported by the Academic Summit Program 114-2639-M-002-009-ASP.

\addtocontents{toc}{\protect\setcounter{tocdepth}{2}}
\section{Curves and standard autoequivalences on surfaces}
\label{sec:curve-std}

In this section, we recall the definition of Bridgeland stability conditions on triangulated categories and certain natural group actions on the space of stability conditions. As a warm-up for later sections, we solve the realization problems for bounded derived categories of coherent sheaves on curves, and for groups of standard autoequivalences on surfaces. We also provide a classification of Gepner-type autoequivalences in the curve case.

\subsection{Preliminaries on stability conditions}
\label{subsec:prelim_stab}

Let us review the definition of Bridgeland stability conditions following \cite{Bri07}. Let $\sD$ be a triangulated category, $K_0(\sD)$ be the Grothendieck group of $\sD$, and
$
    \cl: K_0(\sD)\rightarrow\Gamma
$
be a group homomorphism to a finite-rank abelian group $\Gamma$. Then a \emph{Bridgeland stability condition} $\sigma=(Z,\sP)$ on $\sD$ consists of
\begin{itemize}
    \item a group homomorphism $Z\colon\Gamma\rightarrow\bC$ called the \emph{central charge}, and
    \item a collection of full additive subcategories $\sP=\{\sP(\phi)\subseteq\sD\}_{\phi\in\bR}$ called the \emph{slicing}
\end{itemize}
such that
\begin{enumerate}[label=(\arabic*)]
	\item for every nonzero object $E\in\sP(\phi)$, we have $Z(E)\colonequals Z(\cl([E]))\in\bR_{>0}\cdot e^{i\pi\phi}$,
	\item $\sP(\phi+1)=\sP(\phi)[1]$,
	\item if $\phi_1>\phi_2$ and $A_i\in \sP(\phi_i)$ for $i=1,2$, then $\Hom(A_1, A_2)=0$,
	\item for any nonzero object $E\in\sD$, there exists a (unique) collection of exact triangles
    $$\xymatrix@C=.5em{
        & 0 \ar[rrrr] &&&& E_1 \ar[rrrr] \ar[dll] &&&& E_2
        \ar[rr] \ar[dll] && \cdots \ar[rr] && E_{n-1}
        \ar[rrrr] &&&& E \ar[dll]  &  \\
        &&& A_1 \ar@{-->}[ull] &&&& A_2 \ar@{-->}[ull] &&&&&&&& A_n \ar@{-->}[ull] 
    }$$
	where $A_i\in \sP(\phi_i)$ and $\phi_1>\phi_2>\cdots>\phi_n$,
	\item there exists a constant $C>0$ and a norm $||\cdot||$ on $\Gamma\otimes_\bZ\bR$ such that
    $$
        ||\cl([E])|| \leq C|Z(E)|
        \qquad\text{for all}\qquad
        E\in \sP(\phi).
    $$
\end{enumerate}

Due to \cite{Bri07}*{Proposition~5.3}, giving a stability condition $\sigma = (Z,\sP)$ is equivalent to giving a pair $(Z,\sA)$, where $\sA$ is the heart of a bounded t-structure on $\sD$, such that the following additional conditions are satisfied:
\begin{enumerate}[label=(\alph*)]
    \item for any nonzero object $E\in\sA$, we have $Z(E)\in\cH\cup\bR_{<0}$ where
    $$
        \cH\colonequals\{
            z\in\bC \mid \mathrm{Im}(z)>0
        \}.
    $$
    \item for any nonzero object $E\in\sA$, there exists a (unique) filtration
    $$
        0=E_0\subseteq E_1\subseteq\cdots\subseteq E_n=E
        \qquad\text{where}\qquad
        E_i\in\sA
    $$
    such that $A_i\colonequals E_i/E_{i-1}$ is $Z$-semistable for all $i$ and $\arg(A_1)>\cdots>\arg(A_n)$. Here $A\in\sA$ is called \emph{$Z$-semistable} if $\arg(A)\geq\arg(B)$ for any subobject $B$ of $A$ in $\sA$.
    \item there exists a constant $C>0$ and a norm $||\cdot||$ on $\Gamma\otimes_\bZ\bR$ such that
	$$
        ||\cl([E])|| \leq C|Z(E)|
    $$
    for any $Z$-semistable $E\in\sA$.
\end{enumerate}

We will mainly consider the case where $\sD=\Db(X)$ is the bounded derived category of coherent sheaves on a variety $X$. In this setting, we take $\Gamma$ to be the numerical Grothendieck group
$$
    \cN(\sD)\colonequals
    K_0(\sD)/K_0(\sD)^{\perp\chi},
$$
namely the quotient of the Grothendieck group by the kernel of the Euler pairing
$$
    \chi(E_1,E_2)
    = \sum_i(-1)^i\dim\Hom(E_1,E_2[i]).
$$
The homomorphism $\cl\colon K_0(\sD)\rightarrow\cN(\sD)$ is then the natural quotient.

\begin{eg}
\label{eg:curvestab}
Consider $\sD=\Db(X)$, where $X$ is a curve. Then $\cN(\sD)$ can be identified with $\bZ\oplus\bZ$, under which the map $\cl\colon K_0(\sD)\rightarrow\cN(\sD)$ sends $[E]$ to $(\mathrm{rank}(E),\,\deg(E))$. We now describe a class of stability conditions on $\sD$. For any $\beta+i\omega\in\cH$, define a group homomorphism $Z_{\beta,\,\omega}\colon\cN(\sD)\rightarrow\bC$ by
$$
    Z_{\beta,\,\omega}(E)
    = -\deg(E)
    + (\beta+i\omega)\,\mathrm{rank}(E).
$$
For each $\phi\in(0,1]$, let $\sP(\phi)\subseteq\sD$ be the full subcategory of slope-semistable coherent sheaves $E$ satisfying $Z_{\beta,\,\omega}(E)\in\bR_{>0}\cdot e^{i\pi\phi}$. For all other $\phi\in\bR$, define $\sP(\phi)$ by the relation $\sP(\phi+1)=\sP(\phi)[1]$. Then $\sigma_{\beta,\,\omega}=(Z_{\beta,\,\omega},\sP)$ is a stability condition, called a \emph{geometric stability condition}.
\end{eg}

Let $\Stab(\sD)$ be the set of Bridgeland stability conditions on $\sD$. Bridgeland introduced a topology on $\Stab(\sD)$ and proved that the forgetful map
$$
    \Stab(\sD)
    \longrightarrow
    \Hom(\Gamma,\bC)
    : \sigma=(Z,\sP)
    \longmapsto
    Z
$$
is a local homeomorphism. This equips $\Stab(\sD)$ with the structure of a finite-dimensional complex manifold. The space $\Stab(\sD)$ carries a left action by the group of $\bC$-linear exact autoequivalences $\Aut(\sD)$ where an autoequivalence $\Phi$ acts as
$$
    \sigma = (Z,\sP)
    \longmapsto
    \Phi(\sigma)\colonequals(
        Z\circ[\Phi]^{-1},\sP'
    )
    \quad\text{where}\quad
    \sP'(\phi) = \Phi(\sP(\phi)).
$$
There is also a right action by $\widetilde{\GL}^+(2,\bR)$, namely, the universal cover of $\GL^+(2,\bR)$. An element of this group corresponds to a pair $(T,f)$ where
\begin{itemize}
    \item $T\in\GL^+(2,\bR)$, and
    \item $f\colon\bR\rightarrow\bR$ is an increasing function that satisfies $f(\phi+1) = f(\phi)+1$,
\end{itemize}
such that the actions of $f$ and $T$ on the quotients $\bR/2\bZ$ and $(\bR^2\setminus\{0\})/\bR_{>0}$, respectively, are compatible via the isomorphisms
$$\xymatrix@R=0pt{
    \bR/2\bZ\ar[r]^-\sim
    & S^1\ar[r]^-\sim
    & (\bR^2\setminus\{0\})/\bR_{>0} \\
    \phi\ar@{|->}[r] & e^{i\pi\phi}. &
}$$
Each $g=(T,f)$ acts on $\Stab(\sD)$ by
$$
    \sigma = (Z,\sP)
    \longmapsto
    \sigma\cdot g
    \colonequals
    (T^{-1}\circ Z,\sP'')
    \quad\text{where}\quad
    \sP''(\phi)\colonequals\sP(f(\phi)).
$$
The left action by $\Aut(\sD)$ and the right action by $\widetilde{\GL}^+(2,\bR)$ commute with each other according to \cite{Bri07}*{Lemma~8.2}.

If we view $\bR^2$ as the complex plane $\bC$, then multiplication by complex numbers realizes $\bC^*$ as a subgroup of $\GL^+(2,\bR)$, which lifts to a subgroup $\bC\subseteq\widetilde{\GL}^+(2,\bR)$. Each $\lambda\in\bC$ acts on the space $\Stab(\sD)$ by
\begin{equation}
\label{eqn:C-action}
    \sigma = (Z,\sP)
    \longmapsto
    \sigma\cdot\lambda
    \colonequals
    (Z\cdot e^{-i\pi\lambda}, \sP'')
    \quad\text{where}\quad
    \sP''(\phi) = \sP(\phi+\mathrm{Re}(\lambda)).
\end{equation}
Note that $\bC$ is not a normal subgroup of $\widetilde{\GL}^+(2,\bR)$ due to the following elementary fact.

\begin{lemma}
\label{lemma:GL2conjugation}
If $\lambda\in\bC\subseteq\widetilde{\GL}^+(2,\bR)$ satisfies $\mathrm{Re}(\lambda)\notin\bZ$, then for every $g\in\widetilde{\GL}^+(2,\bR)$,
$$
    g^{-1}\lambda g\in\bC
    \;\Longrightarrow\;
    g\in\bC.
$$
\end{lemma}

\begin{proof}
Let $\overline{\lambda}$ and $\overline{g}$ be the images of $\lambda$ and $g$ under the projection
$$
    \widetilde{\GL}^+(2,\bR)
    \longrightarrow\GL^+(2,\bR).
$$
After multiplying each by a positive scalar, we may assume that $\det(\overline{\lambda})=\det(\overline{g})=1$. Then the conditions $\lambda\in\bC$ and $g^{-1}\lambda g\in\bC$ imply that $\overline{\lambda}\in\SO(2,\bR)$ and $\overline{g}^{-1}\overline{\lambda}\overline{g}\in\SO(2,\bR)$, respectively. Moreover, the assumption $\mathrm{Re}(\lambda)\notin\bZ$ implies that $\overline{\lambda}\neq\pm I$.

Set $Q = \overline{g}\,\overline{g}^T$. Since both $\overline{\lambda}$ and $\overline{g}^{-1}\overline{\lambda}\overline{g}$ are orthogonal, a direct calculation yields
$$
    \overline{\lambda}Q = Q\overline{\lambda}.
$$
Because $\overline{\lambda}$ is a rotation by an angle not congruent to $0$ or $\pi$, it preserves no real line. This implies that $Q$ is scalar since it is symmetric and positive-definite. Since $\det(Q)=1$, we conclude that $Q=I$. Thus $\overline{g}\in\SO(2,\bR)$, and hence $g\in\bC$.
\end{proof}

The following lemma shows that an affirmative answer to the realization problem modulo even shifts gives an affirmative answer to the one without taking the quotient by shifts.

\begin{lemma}
\label{lemma:q2impliesq1}
Let $\sD$ be a triangulated category with a stability condition $\sigma$ and let $G$ be a finite subgroup of $\Aut(\sD)$. If $G$ preserves the set $\sigma\cdot\bC$, then it fixes $\sigma$. Therefore, if a finite subgroup of $\Aut(\sD)$ fixes a point of $\Stab(\sD)/\bC$, then it fixes a point of $\Stab(\sD)$.
\end{lemma}
\begin{proof}
For every $\Phi\in G$, there exists $\lambda\in\bC$ such that $\Phi(\sigma)=\sigma\cdot\lambda$, which implies that $\Phi^n(\sigma)=\sigma\cdot(n\lambda)$ for every integer $n$. Since $\Phi$ is of finite order, one obtains $n\lambda=0$ for some positive $n$. Thus $\lambda=0$ and $\Phi(\sigma)=\sigma$.
\end{proof}

\subsection{Gepner-type autoequivalences on curves}
\label{subsec:curve}

Let $X$ be a curve. Under the identification
$$
    \cN(X)\cong\bZ\oplus\bZ, \qquad
    [E]\longmapsto(\mathrm{rank}(E),\deg(E)),
$$
every autoequivalence of $\Db(X)$ acts on $\cN(X)$ via an element of $\SL(2,\bZ)$, and the shift functor $[1]$ acts as $-\mathrm{id}$. The reason behind this phenomenon depends on genus. If $g(X)=1$, the Euler form on $\cN(X)$ is the standard skew-symmetric form on $\bZ^2$. Since every autoequivalence preserves this form, its numerical action lies in $\mathrm{Sp}(2,\bZ)=\SL(2,\bZ)$. If $g(X)\neq1$, every autoequivalence is standard by Bondal--Orlov \cite{BondalOrlov}. Pullbacks by automorphisms act trivially on rank and degree, tensor products with line bundles act by
$$
    (r,d)\longmapsto\bigl(r,d+r\deg(L)\bigr),
$$
and shifts act by multiplication by $-1$. Hence their numerical actions all lie in $\SL(2,\bZ)$.

Consider the composition
$$
    \Aut(\Db(X))
    \longrightarrow\SL(2,\bZ)
    \longrightarrow\PSL(2,\bZ)
$$
and denote its kernel by $\widetilde{\cI}(\Db(X))$. Then
$$
    \widetilde{\cI}(\Db(X))
    \cong(\Pic^0(X)\rtimes\Aut(X))\times\bZ[1].
$$
Note that a stability condition $\sigma_{\beta,\,\omega}\in\Stab(X)$ given as in Example~\ref{eg:curvestab} is preserved under the action of $\Pic^0(X)\rtimes\Aut(X)$. Thus every element of $\widetilde{\cI}(\Db(X))$ is of Gepner type.

\begin{prop}
\label{prop:nielsen_curve}
Let $X$ be a curve. Then the following statements hold:
\begin{itemize}
    \item Every finite subgroup of $\Aut(\Db(X))/\bZ[2]$ fixes a point of $\Stab(X)/\bC$.
    \item Every finite subgroup of $\Aut(\Db(X))$ fixes a point of $\Stab(X)$.
\end{itemize}
\end{prop}

\begin{proof}
Suppose first that $g(X)\neq1$. By Bondal--Orlov \cite{BondalOrlov},
$$
    \Aut(\Db(X))
    \cong
    (\Pic(X)\rtimes\Aut(X))\times\bZ[1].
$$
This induces a short exact sequence
$$\xymatrix{
    0 \ar[r]
    & \widetilde{\cI}(\Db(X))/\bZ[2] \ar[r]
    & \Aut(\Db(X))/\bZ[2] \ar[r]
    & \bZ \ar[r]
    & 0.
}$$
Since $\bZ$ has no nontrivial finite subgroup, every finite subgroup of $\Aut(\Db(X))/\bZ[2]$ is contained in $\widetilde{\cI}(\Db(X))/\bZ[2]$, and thus preserves $\overline{\sigma}_{\beta,\,\omega}\in\Stab(X)/\bC$ for every $\sigma_{\beta,\,\omega}$ constructed in Example~\ref{eg:curvestab}.

Suppose $g(X)=1$. By \cite{Bri07}*{Theorem~9.1}, there is an isomorphism $\Stab(X)\cong\bC\times\cH$. In addition, the quotient $\Stab(X)/\bC\cong\cH$  admits a section in $\Stab(X)$ given by stability conditions $\sigma_{\beta,\,\omega}$ constructed in Example~\ref{eg:curvestab}. It is straightforward to check that the action of $\Aut(\Db(X))/\bZ[2]$ on $\Stab(X)/\bC\cong\cH$ is compatible with the standard $\SL(2,\bZ)$-action on $\cH$ through the homomorphism $\Aut(\Db(X))/\bZ[2]\longrightarrow\SL(2,\bZ)$. 
The result now follows from the elementary fact that every finite subgroup of $\SL(2,\bZ)$ fixes a point of $\cH$; see, for example, \cite{Kat92}*{Theorem~2.4.1}. This proves the first statement.

The second statement follows from the first by Lemma~\ref{lemma:q2impliesq1}.
\end{proof}

\begin{rmk}
\label{rmk:curve_gepner-not-finite}
The converse of Proposition~\ref{prop:nielsen_curve} is not true, that is, a Gepner-type autoequivalence on a curve $X$ may not have finite order in $\Aut(\Db(X))/\bZ[2]$.
For instance, an autoequivalence induced by an infinite-order automorphism of $X$ is of Gepner type, and is of infinite order in $\Aut(\Db(X))/\bZ[2]$. In contrast, we will show in later sections that for K3~surfaces of Picard number one, an autoequivalence is of Gepner type if and only if it is of finite order in $\Aut(\Db(X))/\bZ[2]$.
\end{rmk}

We first classify the autoequivalences that are of Gepner type with respect to geometric stability conditions. Recall that a stability condition $\sigma$ is \emph{geometric} if all skyscraper sheaves are $\sigma$-stable of the same phase. Let $U(X)\subseteq\Stab(X)$ be the set of geometric stability conditions. By \cite{Mac07}*{Theorem~2.7} and \cite{BMW15}*{Section~3.2}, there is an isomorphism
$$\xymatrix{
    \cH\times\bC \ar[r]^-\sim
    & U(X): (\beta+i\omega,\lambda) \ar@{|->}[r]
    & \sigma_{\beta,\,\omega}\cdot\lambda.
}$$
The action of $\Aut(\Db(X))$ preserves $U(X)$. If $g(X)\geq1$, this follows from the fact that every stability condition on $\Db(X)$ is geometric \cite{Mac07}*{Theorem~2.7}. If $X\cong\bP^1$, every autoequivalence is standard, and it is straightforward to verify that standard autoequivalences preserve $U(X)$. Under the isomorphism $U(X)/\bC\cong\cH$, the action of $\Aut(\Db(X))$ agrees with the action of $\SL(2,\bZ)$ on $\cH$. Therefore, $\Phi\in\Aut(\Db(X))$ is of Gepner type with respect to a geometric stability condition if and only if the action of $[\Phi]\in\SL(2,\bZ)$ on $\cH$ has a fixed point, which is equivalent to $[\Phi]$ being of finite order.

Before proceeding further, we introduce some necessary notation by looking at a few examples illustrating how the map
$
    \Aut(\Db(X))\rightarrow\SL(2,\bZ)
$
works. First, for curves of any genus, the map sends $T\colonequals -\otimes\cO_X(1)$ to
$$
    [T] = \begin{pmatrix}
        1 & 0 \\ 1 & 1
    \end{pmatrix}.
$$
For curves of genus one, let $P\in\Coh(X\times X)$ be the Poincar\'e line bundle. Then the Fourier--Mukai transform $S\colonequals\Phi_P$ along $P$ is mapped to
$$
    [S] = \begin{pmatrix}
        0 & 1 \\ -1 & 0
    \end{pmatrix}.
$$
Every non-identity finite-order element of $\PSL(2,\bZ)$ is conjugate to one of $[S],[ST],[(ST)^2]$; see \cite{DS05}*{Proposition~2.3.3 and the paragraph preceding it}.

\begin{lemma}
\label{lemma:curve_geo-stab}
For a curve $X$ of genus $g$, an autoequivalence $\Phi\in\Aut(\Db(X))$ is of Gepner type with respect to a geometric stability condition if and only if one of the following holds:
\begin{itemize}
    \item When $g\neq1$, we have $\Phi\in\widetilde{\cI}(\Db(X))$.
    \item When $g=1$, there exists $\Psi\in\widetilde{\cI}(\Db(X))$ such that $\Phi\Psi$ is conjugate to $S, ST, (ST)^2$, or the identity.
\end{itemize}
\end{lemma}

\begin{proof}
Note that $\Phi$ is of Gepner type with respect to a geometric stability condition if and only if the action of $[\Phi]\in\SL(2,\bZ)$ on $\cH$ has finite order.

When $g\neq1$, the image of $\Aut(\Db(X))\rightarrow\SL(2,\bZ)$ is generated by $[T]$ and $[[1]] = -\mathrm{id}$. In this case, $[\Phi]$ is of finite order if and only if $[\Phi]=\pm\mathrm{id}$, where the latter condition is equivalent to $\Phi\in\widetilde{\cI}(\Db(X))$. This proves the statement when $g\neq 1$.

When $g=1$, the map $\Aut(\Db(X))\rightarrow\SL(2,\bZ)$ is surjective, so $[\Phi]$ is of finite order if and only if it is conjugate to either $\pm[S]$, $\pm[ST]$, $\pm[(ST)^2]$, or $\pm\mathrm{id}$, which is equivalent to the existence of $\Psi\in\widetilde{\cI}(\Db(X))$ such that $\Phi\Psi$ is conjugate to $S, ST, (ST)^2$, or the identity.
\end{proof}

\begin{prop}
\label{prop:curve_gepner}
For a curve $X$ of genus $g$, an autoequivalence $\Phi\in\Aut(\Db(X))$ is of Gepner type if and only if one of the following holds:
\begin{itemize}
    \item When $g\neq1$, we have $\Phi\in\widetilde{\cI}(\Db(X))$.
    \item When $g=1$, there exists $\Psi\in\widetilde{\cI}(\Db(X))$ such that $\Phi\Psi$ is conjugate to $S, ST, (ST)^2$, or the identity.
\end{itemize}
\end{prop}

\begin{proof}
By \cite{Bri07}*{Theorem~0.1} and \cite{Mac07}*{Theorem~2.7}, any stability condition on $\Db(X)$ is geometric if $g\geq1$. This completes the proof in this case by Lemma~\ref{lemma:curve_geo-stab}.

Assume that $X\cong\bP^1$, where non-geometric stability conditions exist. By Lemma~\ref{lemma:curve_geo-stab}, every autoequivalence in $\widetilde{\cI}(\Db(\bP^1))$ is of Gepner type. It therefore remains to show that every $\Phi\notin\widetilde{\cI}(\Db(\bP^1))$ is not of Gepner type, and it suffices to consider non-geometric stability conditions. Note that every such $\Phi$ is of the form
$$
    \Phi = f^*(-\otimes\cO(k))[n] \qquad (k\neq0)
$$
for some $f\in\Aut(\bP^1)$ and $n\in\bZ$. By \cite{Okada} and \cite{BMW15}*{Section~3.2}, if $\sigma\in\Stab(\bP^1)$ is non-geometric, then there exists $\ell\in\bZ$ such that $\cO(\ell)$, $\cO(\ell+1)$, and their shifts are the only $\sigma$-stable objects. Since $k\neq0$, the autoequivalence $\Phi$ does not preserve the set of $\sigma$-stable objects for non-geometric $\sigma$, while the free $\bC$-action on $\sigma$ does preserve the set of $\sigma$-stable objects. Thus $\Phi$ is not of Gepner type with respect to non-geometric stability conditions. This concludes the proof.
\end{proof}

\subsection{Standard autoequivalences on surfaces}
\label{subsec:stdautoeq_surf}

Let us review the tilting construction of stability conditions on surfaces. First, recall that a \emph{torsion pair} on an abelian category $\sA$ consists of a pair of full additive subcategories $(\sT,\sF)$ such that:
\begin{enumerate}[label=(\alph*)]
    \item for any $T\in\sT$ and $F\in\sF$, we have $\Hom(T,F)=0$;
    \item for any $E\in\sA$, there exist (unique) $T\in\sT$ and $F\in\sF$ together with a short exact sequence $0\longrightarrow T\longrightarrow E\longrightarrow F\longrightarrow0$.
\end{enumerate}

\begin{lemma}[\cite{HRS96}]
\label{lem:tilting}
    Let $\sA$ be the heart of a bounded t-structure on a triangulated category $\sD$ and $(\sT,\sF)$ be a torsion pair on $\sA$. Then the category
    $$
        \sA^\sharp = \left\{
            E\in\sD
            \mid
            H^0_\sA(E)\in\sT,
            \ H^{-1}_\sA(E)\in\sF,
            \ H^i_\sA(E)=0
            \;\text{ for }\;
            i\neq0,-1
        \right\}
$$
is the heart of a bounded t-structure on $\sD$. Here $H^\bullet_\sA$ denotes the cohomology object with respect to the t-structure of $\sA$.
\end{lemma}

For a surface $X$, the numerical Grothendieck group of $\Db(X)$ is given by
$$
\cN(X)\cong H^0(X,\bZ)\oplus\frac{\NS(X)}{\NS(X)_{\text{torsion}}}\oplus H^4(X,\bZ).
$$
Let $\beta,\omega\in\NS(X)\otimes\bR$ be a pair of $\bR$-divisors such that $\omega$ is ample. Then the slope function
$$
    \mu_\omega(E)
    = \frac{\omega\cdot c_1(E)}{\mathrm{rank}(E)}
$$
defines a (slope) stability on the abelian category $\sA=\Coh(X)$. Each $E\in\Coh(X)$ has a unique Harder--Narasimhan filtration
$$
    0\subseteq E_\mathrm{tor}=E_0
    \subseteq E_1
    \subseteq
    \cdots
    \subseteq E_n=E
$$
where $E_\mathrm{tor}$ is the torsion part of $E$, and each $E_i/E_{i-1}$ is a torsion-free $\mu_\omega$-semistable sheaf of slope $\mu_i$ with $\mu_\omega^{\max}(E)=\mu_1>\cdots>\mu_n=\mu_\omega^{\min}(E)$. Then
$$
    \sT_{\beta,\,\omega}
    = \{\text{torsion sheaves}\}
        \cup
    \{E\mid\mu_\omega^{\min}(E)>\beta\cdot\omega\}
    \quad\text{and}\quad
    \sF_{\beta,\,\omega}
    = \{E\mid\mu_\omega^{\max}(E)\leq\beta\cdot\omega\}
$$
form a torsion pair on $\sA$. One obtains a new heart $\sA_{\beta,\,\omega}^\sharp\subseteq\Db(X)$ by applying Lemma~\ref{lem:tilting} to this torsion pair.

\begin{thm}[\cite{ArBe}*{Corollary~2.1}]
\label{thm:surfacestabilitychern}
Let $X$ be a surface, and consider the central charge
$$
    Z'_{\beta,\,\omega}(E)
    = -\int_X e^{-(\beta+i\omega)}\ch(E).
$$
Then the pair $\sigma'_{\beta,\,\omega}=(Z'_{\beta,\,\omega},\sA_{\beta,\,\omega}^\sharp)$ gives a stability condition on $\Db(X)$.
\end{thm}

For a variety $X$, the elements of the subgroup
$$
    \Aut_\std(\Db(X))
    \colonequals (\Pic(X)\rtimes\Aut(X))\times\bZ[1]
    \subseteq\Aut(\Db(X))
$$
are called \emph{standard autoequivalences}. For surfaces, the action of such autoequivalences on the stability conditions introduced in Theorem~\ref{thm:surfacestabilitychern} is particularly explicit.

\begin{lemma}
\label{lemma:standard-action-geometric-stability}
Consider a surface $X$, a standard autoequivalence
$$
    \Phi_{L,f}\colonequals(-\otimes L)\circ f^*
    \in \Pic(X)\rtimes\Aut(X),
$$
and pick $\beta,\omega\in\NS(X)\otimes\bR$ with $\omega$ ample. Then
$
    \Phi_{L,f}(\sigma'_{\beta,\,\omega})
    = \sigma'_{f^*\beta+c_1(L),\;f^*\omega}.
$
\end{lemma}

\begin{proof}
For brevity, set $\beta'=f^*\beta+c_1(L)$ and $\omega'=f^*\omega$. By definition, the central charge of an object $E$ with respect to $\Phi_{L,f}(\sigma'_{\beta,\,\omega})$ is given by
\begin{align*}
    Z'_{\beta,\,\omega}(\Phi^{-1}_{L,f}(E))
    &= -\int_X e^{-(\beta+i\omega)}
        \ch\left((f^{-1})^*(E\otimes L^{-1})\right) \\
    & = -\int_X e^{-(\beta+i\omega)}
        (f^{-1})^*\left(\ch(E)e^{-c_1(L)}\right) \\
    & = -\int_X e^{-\left(f^*\beta+c_1(L)+if^*\omega\right)}
        \ch(E)
    = Z'_{f^*\beta+c_1(L),\;f^*\omega}(E)
    = Z'_{\beta',\,\omega'}(E),
\end{align*}
which coincides with the central charge of $E$ with respect to $\sigma'_{\beta',\,\omega'}$.

It remains to identify the corresponding hearts. For a torsion-free sheaf $E$ of rank $r$, we have $c_1(\Phi_{L,f}(E))=f^*c_1(E)+rc_1(L)$. Hence,
$$
    \mu_{\omega'}(\Phi_{L,f}(E))
    = \frac{
        f^*\omega\cdot(f^*c_1(E) + rc_1(L))
    }{r}
    = \mu_\omega(E) + f^*\omega\cdot c_1(L).
$$
Equivalently,
$$
    \mu_{\omega'}(\Phi_{L,f}(E)) - \beta'\cdot\omega'
    = \mu_\omega(E) - \beta\cdot\omega.
$$
In particular, $\Phi_{L,f}$ is an exact autoequivalence of $\Coh(X)$ that shifts all slopes by the same constant, so it takes a $\mu_\omega$-Harder--Narasimhan filtration to a $\mu_{\omega'}$-Harder--Narasimhan filtration. Moreover,
$$
    \Phi_{L,f}(\sT_{\beta,\,\omega})
    = \sT_{\beta',\,\omega'}, \qquad
    \Phi_{L,f}(\sF_{\beta,\,\omega})
    = \sF_{\beta',\,\omega'},
$$
and hence
$
    \Phi_{L,f}(\sA^\#_{\beta,\,\omega})
    = \sA^\#_{\beta',\,\omega'}.
$
This proves the lemma.
\end{proof}

The following proposition shows that the realization problems have affirmative answers for standard autoequivalences on a surface.

\begin{prop}
\label{prop:nielsen_stdauto-surf}
Let $X$ be a surface. Then the following statements hold:
\begin{itemize}
    \item Every finite subgroup of $\Aut_\std(\Db(X))/\bZ[2]$ fixes a point of $\Stab(X)/\bC$.
    \item Every finite subgroup of $\Aut_\std(\Db(X))$ fixes a point of $\Stab(X)$.
\end{itemize}
\end{prop}
\begin{proof}
Since $[1]$ acts trivially on $\Stab(X)/\bC$, it suffices to prove that every finite subgroup
$$
    G\subseteq\Pic(X)\rtimes\Aut(X)
$$
fixes a stability condition of the form
$
    \sigma'_{\beta,\,\omega}
    = (
        Z'_{\beta,\,\omega},
        \sA_{\beta,\,\omega}^\sharp
    )
$
given in Theorem~\ref{thm:surfacestabilitychern}. For each $g\in G$, write $g=\Phi_{L_g,f_g}=(-\otimes L_g)\circ f_g^*$. Choose any $\omega_0\in\Amp(X)$ and set
$$
    \widetilde{\beta}\colonequals
    \frac{1}{|G|}\sum_{g\in G}c_1(L_g), \qquad
    \widetilde{\omega}\colonequals
    \frac{1}{|G|}\sum_{g\in G}f_g^*\omega_0.
$$
We claim that $G$ fixes the stability condition $\sigma'_{\widetilde{\beta},\,\widetilde{\omega}}$. By Lemma~\ref{lemma:standard-action-geometric-stability}, it suffices to show that for every $h=\Phi_{L_h,f_h}\in G$, one has
$$
    f_h^*\widetilde{\beta} + c_1(L_h)
    = \widetilde{\beta}, \qquad
    f_h^*\widetilde{\omega}
    = \widetilde{\omega}.
$$
The second equality follows immediately from the definition of $\widetilde{\omega}$. To prove the first, write the group law so that $hg$ corresponds to the composition $\Phi_{L_h,f_h}\circ\Phi_{L_g,f_g}$. Then
$$
    \Phi_{L_{hg},f_{hg}}(E)
    = \Phi_{L_h,f_h}(\Phi_{L_g,f_g}(E))
    = f_h^*(f_g^*E\otimes L_g)\otimes L_h
    = f_h^*f_g^*E\otimes f_h^* L_g\otimes L_h.
$$
Hence $L_{hg}=f_h^*L_g\otimes L_h$, and therefore
$$
    f_h^*\widetilde{\beta}+c_1(L_h)
    = \frac{1}{|G|}\sum_{g\in G} (c_1(f_h^*L_g)+c_1(L_h))
    = \frac{1}{|G|}\sum_{g\in G} c_1(L_{hg})
    = \widetilde{\beta}.
$$
Thus $\sigma'_{\widetilde{\beta},\,\widetilde{\omega}}$ is fixed by $G$.
\end{proof}

\begin{rmk}
\label{rmk:central_mukai}
Let $X$ be a K3~surface. Consider the Mukai vector
$$
    v = \ch\cdot\sqrt{\mathrm{td}_X}
    = (\ch_0,\ch_1,\ch_0+\ch_2)
$$
and the central charge
$$
    Z_{\beta,\,\omega}(E)
    = e^{\beta+i\omega}\cdot v(E).
$$
Then
$
    \sigma_{\beta,\,\omega} = (
        Z_{\beta,\,\omega},
        \sA^\sharp_{\beta,\,\omega}
    )
$
defines a stability condition on $\Db(X)$, provided that $\beta$ and $\omega$ are chosen so that $Z_{\beta,\,\omega}(E)\notin\bR_{\leq 0}$ for every spherical sheaf $E$ on $X$. In particular, this condition holds when $\omega^2>2$ \cite{Bri08}*{Lemma~6.2}, which can be achieved by taking $\omega$ to be a sufficiently large multiple of an invariant ample class. Following the same proof strategy, one can reproduce Proposition~\ref{prop:nielsen_stdauto-surf} for K3~surfaces using such a stability condition.
\end{rmk}

\section{Gepner-type autoequivalences on K3~surfaces}
\label{sec:k3_gepner}

Gepner-type autoequivalences on K3~surfaces admit a more explicit description in terms of their actions on period domains of stability conditions. This section develops the necessary preliminaries and discusses Gepner-type autoequivalences for K3~surfaces. In the final part of this section, we solve the realization problems for twisted abelian surfaces and generic twisted K3~surfaces.

\subsection{Basics on autoequivalences and stability conditions}
\label{subsec:prelim_stabK3}

Here we review basic facts about autoequivalences and stability conditions on K3~surfaces following \cites{Bri08,HMS09}. Along the way, we introduce notation that will be used throughout the paper.

Given a K3~surface $X$, one can extend the polarized weight-two Hodge structure on the middle cohomology $H^2(X,\bZ)$ to the total cohomology
$$
    \mukaiH(X,\bZ)\colonequals
    H^0(X,\bZ)(-1)\oplus H^2(X,\bZ)\oplus H^4(X,\bZ)(1)
$$
such that the $(1,1)$-part is isomorphic to the numerical Grothendieck group
$$
    \cN(X)\cong
    H^0(X,\bZ)(-1)\oplus\NS(X)\oplus H^4(X,\bZ)(1)
$$
and the polarization is obtained by extending the intersection product on $H^2(X,\bZ)$ orthogonally to the other summands as
$$
    (r_1, D_1, s_1)\cdot(r_2, D_2, s_2)
    = D_1\cdot D_2 - r_1s_2 - s_1r_2.
$$
This turns $\mukaiH(X,\bZ)$ into an even unimodular lattice of signature $(4,20)$ called the \emph{Mukai lattice}. The orientations of two given positive $4$-planes in $\mukaiH(X,\bR)\colonequals\mukaiH(X,\bZ)\otimes\bR$ either coincide or are opposite after orthogonal projection from an arbitrary negative $20$-plane. In particular, one can compare the orientations of a positive $4$-plane and its image under a Hodge isometry $\phi\in\Aut(\mukaiH(X,\bZ))$ this way. Let us denote by 
$$
    \Aut^+(\mukaiH(X,\bZ))\subseteq\Aut(\mukaiH(X,\bZ))
$$
the subgroup of orientation-preserving isometries.

As asserted by \cite{HMS09}*{Corollary~3}, every autoequivalence $\Phi\in\Aut(\Db(X))$ induces an element $\phi\in\Aut^+(\mukaiH(X,\bZ))$ and, conversely, every element of $\Aut^+(\mukaiH(X,\bZ))$ arises this way. This gives a short exact sequence
\begin{equation}
\label{eqn:torelli}
\xymatrix@R=0pt{
    0\ar[r] &
    \cI(\Db(X))\ar[r] &
    \Aut(\Db(X))\ar[r] &
    \Aut^+(\mukaiH(X,\bZ))\ar[r] &
    0.
}
\end{equation}
We call the kernel $\cI(\Db(X))$ the \emph{Torelli group}. The notation and name for the kernel are inspired by the study of mapping class groups $\mathrm{Mod}(\Sigma)$ of topological surfaces $\Sigma$, where the subgroup of $\mathrm{Mod}(\Sigma)$ acting trivially on $H_1(\Sigma,\bZ)$ is called the Torelli group and usually denoted by $\cI(\Sigma)$. Let us denote by
$$
    \Aut_s(\Db(X))\subseteq\Aut(\Db(X))
    \qquad\text{and}\qquad
    \Aut^+_s(\mukaiH(X,\bZ))\subseteq\Aut^+(\mukaiH(X,\bZ))
$$
the subgroups of symplectic elements, that is, the elements acting trivially on $H^{2,0}(X)$. Then the exact sequence \eqref{eqn:torelli} contains the exact subsequence
$$\xymatrix{
    0\ar[r] &
    \cI(\Db(X))\ar[r] &
    \Aut_s(\Db(X))\ar[r] &
    \Aut^+_s(\mukaiH(X,\bZ))\ar[r] &
    0.
}$$
Here, the whole group $\cI(\Db(X))$ is the kernel since its elements are symplectic.

Every stability condition $\sigma = (Z,\mathscr{P})$ on $X$ uniquely determines a $w\in\cN(X)\otimes\bC$ such that the central charge takes the form $Z(-) = (w, v(-))$, where $v(-) = \mathrm{ch}(-)\sqrt{\mathrm{td}(X)}$ is the Mukai vector. This identifies $Z$ as an element of $\cN(X)\otimes\bC$ and thus induces a map
\begin{equation}
\label{eqn:stab-to-groth}
\xymatrix@R=0pt{
    \Stab(X)\ar[r] &
    \cN(X)\otimes\bC : (Z,\mathscr{P})
    \ar@{|->}[r] & Z.
}
\end{equation}
There is a distinguished connected component
$
    \Stab^\dag(X)\subseteq\Stab(X)
$
which contains the set of stability conditions for which all skyscraper sheaves $\cO_x$ of points $x\in X$ are stable of the same phase. To describe the image of \eqref{eqn:stab-to-groth} when restricted to this component, first consider the open subset
$$
    \cP(X)\colonequals\left\{
    \gamma\in\cN(X)\otimes\bC
    \;\middle|\;
    \mathrm{Re}(\gamma),
    \mathrm{Im}(\gamma)
    \text{ span a positive plane in }
    \cN(X)\otimes\bR
    \right\}.
$$
The real and imaginary parts of each $\gamma\in\cP(X)$ determine a natural orientation on the positive plane spanned by them. This leads to a decomposition into connected components
$$
    \cP(X) = \cP^+(X)\sqcup\cP^-(X)
$$
which are complex conjugate to each other. Here we require $\cP^+(X)$ to be the component which contains $e^{ih}$ with $h\in\NS(X)$ an ample class.

By Bridgeland \cite{Bri08}*{Theorem~1.1}, if we remove every hyperplane section on $\cP^+(X)$ which is orthogonal to a $(-2)$-vector in $\cN(X)$ to get
$$
    \cP^+_0(X)\colonequals
    \cP^+(X)
    \;\setminus\;
    \bigcup_{\substack{
        \delta\in\cN(X),\;
        \delta^2=-2
    }}\delta^\perp,
$$
then the restriction of \eqref{eqn:stab-to-groth} to the component $\Stab^\dag(X)$ defines a topological covering
$$\xymatrix{
    \pi\colon\Stab^\dag(X)\ar[r] & \cP^+_0(X).
}$$
Moreover, the subgroup 
$
    \cI^\dag(\Db(X))\subseteq\cI(\Db(X))
$
preserving the component $\Stab^\dag(X)$ acts freely on $\Stab^\dag(X)$ and is the group of deck transformations of $\pi$. Due to this result, the space $\cP^+_0(X)$ may be viewed as a \emph{period domain} for stability conditions.

A second period domain, $\cQ^+_0(X)$, will also play an important role. To define it, consider
\begin{equation}
\label{eqn:classical_domain}
    \cQ(X)\colonequals
    \{
        [\gamma]\in\bP(\cN(X)\otimes\bC)
        \mid
        \gamma^2 = 0,\; \gamma\cdot\overline{\gamma}>0
    \}.
\end{equation}
For each class $[\gamma]\in\cQ(X)$, the oriented plane in $\cN(X)\otimes\bR$ spanned by $\mathrm{Re}(\gamma)$ and $\mathrm{Im}(\gamma)$ is independent of the choice of representative $\gamma$. This assignment identifies $\cQ(X)$ with the open subset of positive planes within the Grassmannian of oriented planes. In particular, $\cQ(X)$ has two connected components. We denote them by $\cQ^+(X)$ and $\cQ^-(X)$ so that they are respectively the bases under $\cP^+(X)$ and $\cP^-(X)$ in the $\GL^+(2,\bR)$-fibration
$$\xymatrix{
    \cP(X)\ar[r]
    & \cQ(X) : \gamma\ar@{|->}[r]
    & \text{the oriented plane spanned by }\mathrm{Re}(\gamma)\text{ and }\mathrm{Im}(\gamma).
}$$
The desired space is then obtained by removing the locus orthogonal to $(-2)$-vectors
$$
    \cQ^+_0(X)\colonequals
    \cQ^+(X)
    \;\setminus\;
    \bigcup_{\substack{
        \delta\in\cN(X),\;
        \delta^2=-2
    }}\delta^\perp.
$$

In the study of Gepner-type autoequivalences for K3~surfaces, we will consider only the stability conditions lying in the component $\Stab^\dag(X)$. For K3~surfaces, this component is expected to be simply connected and preserved by the action of autoequivalences \cite{Bri08}*{Conjecture~1.2}. This is known to hold when $X$ has Picard number one \cite{BB17}*{Theorem~1.3}. In view of this conjecture, the structure of the autoequivalence group should be determined by its action on $\Stab^\dag(X)$. Note that, if an autoequivalence is of Gepner type with respect to a stability condition on $\Stab^\dag(X)$, then it preserves $\Stab^\dag(X)$ as the $\bC$-action preserves connected components. Suppose that $\Phi$ is an autoequivalence of Gepner type with respect to $\sigma = (Z,\sP)\in\Stab^\dag(X)$ and let $\phi$ denote its action on the Mukai lattice. Then there exists $\lambda\in\bC$ such that
$$
    \Phi(\sigma) = \sigma\cdot\lambda
    \qquad\text{which implies}\qquad
    \phi(Z) = Z\cdot e^{-i\pi\lambda}.
$$
Let $P\subseteq\cN(X)\otimes\bR$ be the positive oriented plane spanned by $Z$. Then the second relation shows that $\phi$ fixes $P$ as a point of $\cQ^+_0(X)$, and its action on $\cP^+_0(X)$ is a change of frames on the plane $P$ by applying $e^{-i\pi\lambda}\in\GL^+(2,\bR)$. This observation lies at the core of our study of Gepner-type autoequivalences in the K3 case.

\subsection{Subgroups of Gepner-type autoequivalences}
\label{subsec:gepnergp}

Given a K3~surface $X$ and a stability condition $\sigma\in\Stab^\dag(X)$ with central charge $Z$, we consider the subgroup of autoequivalences on $\Db(X)$ fixing $\sigma$:
$$
    \Aut(\Db(X),\sigma)
    \colonequals\{
        \Phi\in\Aut(\Db(X))
        \mid
        \Phi(\sigma) = \sigma
    \}
$$
as well as the group of Hodge isometries on $\mukaiH(X,\bZ)$ fixing $Z$:
$$
    \Aut(\mukaiH(X,\bZ),Z)
    \colonequals\{
        \phi\in\Aut(\mukaiH(X,\bZ))
        \mid
        \phi(Z) = Z
    \}.
$$
Let us denote by $\Aut_s(\Db(X),\sigma)$ and $\Aut_s(\mukaiH(X,\bZ),Z)$ respectively the subgroups of symplectic elements in these two groups. By \cite{Huy16}*{Remark~1.2 \& Proposition~1.4}, these groups are finite and there is an isomorphism
\begin{equation}
\label{eqn:fixing-stab}
\xymatrix{
    \Aut(\Db(X),\sigma)\ar[r]^-\sim
    & \Aut(\mukaiH(X,\bZ),Z).
}
\end{equation}
This isomorphism also induces an isomorphism between their symplectic subgroups.

The above picture extends to Gepner-type autoequivalences by the same strategy. The autoequivalences of Gepner type with respect to $\sigma$ form a subgroup
$$
    \Aut(\Db(X),\,\sigma\cdot\bC)
    \colonequals\{
        \Phi\in\Aut(\Db(X))
        \mid
        \Phi(\sigma) = \sigma\cdot\lambda
        \;\text{ for some }\;
        \lambda\in\bC
    \}.
$$
Similarly, Hodge isometries which act on $Z$ by rescaling form a subgroup
$$
    \Aut(\mukaiH(X,\bZ),\,Z\cdot\bC^*)
    \colonequals\{
        \phi\in\Aut(\mukaiH(X,\bZ))
        \mid
        \phi(Z) = Z\cdot\kappa
        \;\text{ for some }\;
        \kappa\in\bC^*
    \}.
$$
Let us denote by $\Aut_s(\Db(X),\,\sigma\cdot\bC)$ and $\Aut_s(\mukaiH(X,\bZ),\,Z\cdot\bC^*)$ respectively the subgroups of symplectic elements in these two groups.

\begin{lemma}
\label{lemma:gepner_mukaiH_orient}
We have
$
    \Aut(\mukaiH(X,\bZ),\,Z\cdot\bC^*)
    \subseteq\Aut^+(\mukaiH(X,\bZ)).
$
\end{lemma}
\begin{proof}
Pick any $\phi\in\Aut(\mukaiH(X,\bZ),\,Z\cdot\bC^*)$ and let $P\subseteq\cN(X)\otimes\bR$ be the oriented positive plane determined by $Z$. The fact that $\phi(Z) = Z\cdot\kappa$ for some $\kappa\in\bC^*$ implies that it preserves the orientation of $P$. Because $\phi$ preserves the Hodge structure, it also preserves the orientation of the positive plane
$
    Q\colonequals
    (H^{2,0}\oplus H^{0,2})(X,\bR)
    \subseteq\mukaiH(X,\bR).
$
The two planes $P$ and $Q$ span a positive $4$-plane in $\mukaiH(X,\bR)$ whose orientation is invariant under the action of $\phi$, whence $\phi\in\Aut^+(\mukaiH(X,\bZ))$.
\end{proof}

\begin{lemma}
\label{lemma:gepner_finiteness}
The group $\Aut(\mukaiH(X,\bZ),\, Z\cdot\bC^*)$ is finite.
\end{lemma}

\begin{proof}
Let $P,Q\subseteq\mukaiH(X,\bR)$ be the orthogonal positive-definite two-planes defined in the proof of Lemma~\ref{lemma:gepner_mukaiH_orient}. Both are preserved by every $\phi\in\Aut(\mukaiH(X,\bZ),\, Z\cdot\bC^*)$. Since the lattice $\mukaiH(X,\bZ)$ has signature $(4,20)$, the orthogonal complement $R\colonequals(P\oplus Q)^\perp$ is negative-definite and is also preserved by $\phi$. Thus $\phi$ preserves the orthogonal decomposition
$$
    \mukaiH(X,\bR)\cong P\oplus Q\oplus R.
$$
It follows that
$$
    \Aut(\mukaiH(X,\bZ),\, Z\cdot\bC^*)
    \subseteq K \cap \Aut(\mukaiH(X,\bZ))
$$
where
$$
    K\colonequals \uO(P)\times\uO(Q)\times\uO(R) \subseteq \uO(\mukaiH(X,\bR)).
$$
Since the quadratic forms on $P$, $Q$, $R$ are definite, each of $\uO(P)$, $\uO(Q)$, $\uO(R)$ is compact, and hence so is $K$. Since $\Aut(\mukaiH(X,\bZ))$ is discrete in $\uO(\mukaiH(X,\bR))$, its intersection with the compact group $K$ is finite. Therefore, $\Aut(\mukaiH(X,\bZ),\, Z\cdot\bC^*)$ is finite.
\end{proof}

\begin{prop}
\label{prop:gepner_auteq-isom}
Let $X$ be a K3~surface and $\sigma\in\Stab^\dag(X)$ be a stability condition with central charge $Z$. Then the homomorphism
$$
    \eta\colon\Aut(\Db(X),\,\sigma\cdot\bC)
    \longrightarrow
    \Aut(\mukaiH(X,\bZ),\,Z\cdot\bC^*),
$$
which maps an autoequivalence to its action on the Mukai lattice, is surjective with kernel equal to $\bZ[2]$. In particular, $\eta$ induces an isomorphism between finite groups:
$$\xymatrix{
    \Aut(\Db(X),\,\sigma\cdot\bC)\,/\,\bZ[2]
    \ar[r]^-\sim &
    \Aut(\mukaiH(X,\bZ),\, Z\cdot\bC^*).
}$$
The same statement holds with $\Aut$ replaced by $\Aut_s$.
\end{prop}

\begin{proof}
We first show that $\eta$ is surjective. Let $\phi\in\Aut(\mukaiH(X,\bZ),\, Z\cdot\bC^*)$ and $\lambda\in\bC$ be any element which satisfies $\phi(Z) = Z\cdot e^{-i\pi\lambda}$. It is known that the subgroup
$$
    \Aut^\dag(\Db(X))
    \colonequals\{
        \Phi\in\Aut(\Db(X))
        \mid
        \Phi(\Stab^\dag(X)) = \Stab^\dag(X)
    \}
$$
is surjective onto $\Aut^+(\mukaiH(X,\bZ))$ \cite{Huy16}*{Section~1.3}. This fact, together with Lemma~\ref{lemma:gepner_mukaiH_orient}, implies that $\phi$ is induced by some $\Phi\in\Aut^\dag(\Db(X))$. Both $\Phi(\sigma)$ and $\sigma\cdot\lambda$ belong to $\Stab^\dag(X)$ and share the same central charge $Z\cdot e^{-i\pi\lambda}$, so there exists $\Psi\in\cI^\dag(\Db(X))$ such that $\Psi\Phi(\sigma) = \sigma\cdot\lambda$ \cite{Bri08}*{Theorem~13.3}. Now, $\Psi\Phi\in\Aut(\Db(X),\,\sigma\cdot\bC)$ and this element induces $\phi$. This shows that $\eta$ is surjective.

It is apparent that $\bZ[2]\subseteq\ker(\eta)$. To prove the converse, first pick any $\Phi\in\ker(\eta)$ and let $\lambda\in\bC$ be the element which satisfies $\Phi(\sigma) = \sigma\cdot\lambda$. The hypothesis implies that $e^{-i\pi\lambda} = 1$, so $\lambda = 2m$ for some $m\in\bZ$. It follows that $\Phi\circ[-2m]$ fixes $\sigma$, and it acts as the identity on $\mukaiH(X,\bZ)$. By \cite{Bri08}*{Theorem~1.1}, we have $\Phi\circ[-2m] = \mathrm{id}$, whence $\Phi = [2m]\in\bZ[2]$. This proves that $\ker(\eta)\subseteq\bZ[2]$ and thus $\ker(\eta)=\bZ[2]$.

We have proved that $\eta$ is surjective with kernel equal to $\bZ[2]$, so it induces the isomorphism in the statement. The finiteness of the groups follows from Lemma~\ref{lemma:gepner_finiteness}. The proof for the symplectic version is almost the same, so we leave it to the reader.
\end{proof}

Proposition~\ref{prop:gepner_auteq-isom} implies that an autoequivalence $\Phi$ of Gepner type with respect to a stability condition $\sigma\in\Stab^\dag(X)$ has finite order modulo even shifts. This imposes a strong restriction on possible scalars $\lambda\in\bC$ which can appear in the relation $\Phi(\sigma) = \sigma\cdot\lambda$.

\begin{cor}
\label{cor:finite-ord-scalar}
Let $X$ be a K3~surface, and let $\Phi\in\Aut(\Db(X))$ be of Gepner type with respect to $\sigma\in\Stab^\dag(X)$. If $\Phi$ has order $r$ modulo even shifts, then there exists an integer $k$ such that $\Phi(\sigma) = \sigma\cdot\frac{2k}{r}$.
\end{cor}

\begin{proof}
Let $Z\in\cP^+_0(X)$ be the central charge of $\sigma$ and $\phi\in\Aut(\mukaiH(X,\bZ),\,Z\cdot\bC^*)$ be the isometry induced by $\Phi$. By hypothesis, there exists $\lambda\in\bC$ such that $\Phi(\sigma) = \sigma\cdot\lambda$, which implies that $\phi(Z) = Z\cdot e^{-i\pi\lambda}$. Because $\Phi$ has order~$r$ modulo even shifts, the isometry $\phi$ has order~$r$. Hence $(e^{-i\pi\lambda})^r = 1$. Thus, there exists an integer $k$ such that $\lambda r = 2k$, or equivalently, $\lambda = \frac{2k}{r}$.
\end{proof}

We conclude this subsection by constructing a short exact sequence that will be used later. We begin with the homomorphism
\begin{equation}
\label{eqn:autoeq-scal}
\xymatrix{
    \Aut(\Db(X),\,\sigma\cdot\bC)\ar[r]
    & \bC
}
\end{equation}
which maps $\Phi$ to the scalar $\lambda\in\bC$ which satisfies $\Phi(\sigma) = \sigma\cdot\lambda$. At the lattice level, there is a similar homomorphism
\begin{equation}
\label{eqn:isometry-scal}
\xymatrix{
    \Aut(\mukaiH(X,\bZ),\,Z\cdot\bC^*)\ar[r]
    & \bC^*
}
\end{equation}
which maps $\phi$ to the scalar $\kappa\in\bC^*$ which satisfies $\phi(Z) = Z\cdot\kappa$. These two maps form a commutative diagram
$$\xymatrix{
    \Aut(\Db(X),\,\sigma\cdot\bC)
    \ar[d]^-{\eqref{eqn:autoeq-scal}}\ar@{->>}[r]^-\eta
    & \Aut(\mukaiH(X,\bZ),\,Z\cdot\bC^*)
    \ar[d]^-{\eqref{eqn:isometry-scal}}
    & \Phi\ar@{|->}[d]\ar@{|->}[rr]
    && \phi\ar@{|->}[d] \\
    \bC\ar@{->>}[r]^-{e^{-i\pi(-)}}
    & \bC^*
    & \lambda\ar@{|->}[rr]
    && \kappa=e^{-i\pi\lambda}
}$$
where the map $\eta$ is given in Proposition~\ref{prop:gepner_auteq-isom}. The group $\Aut(\mukaiH(X,\bZ),\,Z\cdot\bC^*)$ is finite by Lemma~\ref{lemma:gepner_finiteness}, so the image of \eqref{eqn:isometry-scal} is the subgroup $\mu_m\subseteq\bC^*$ of $m$-th roots of unity, where
$$
    m = m(X,Z)\colonequals
    \text{the cardinality of the image of \eqref{eqn:isometry-scal} }.
$$
Note that $m$ is an even integer since $-\mathrm{id}$ is mapped to $-1\in\mu_m$. If we consider only $\mu_m$ instead of $\bC^*$, then the above diagram becomes
$$\xymatrix{
    \Aut(\Db(X),\,\sigma\cdot\bC)
    \ar[d]^-{\eqref{eqn:autoeq-scal}}\ar@{->>}[r]^-\eta
    & \Aut(\mukaiH(X,\bZ),\,Z\cdot\bC^*)
    \ar@{->>}[d]^-{\eqref{eqn:isometry-scal}} \\
    \frac{2}{m}\bZ\ar@{->>}[r]^-{e^{-i\pi(-)}}
    & \mu_m.
}$$
By involving all the kernels of these maps, we obtain a commutative diagram with exact columns and rows
\begin{equation}
\label{diag:gepner}
\begin{aligned}
\xymatrix{
    &
    & 0 \ar[d]
    & 0 \ar[d]
    & \\
    &
    & \Aut(\Db(X),\,\sigma)\ar[d]
    \ar[r]_-\sim^-{\eqref{eqn:fixing-stab}}
    & \Aut(\mukaiH(X,\bZ),\,Z)\ar[d]
    & \\
    0 \ar[r]
    & \bZ[2]\ar[d]^(.45)*[@]{\sim}\ar[r]
    & \Aut(\Db(X),\,\sigma\cdot\bC)
    \ar[d]^-{\eqref{eqn:autoeq-scal}}\ar[r]^-\eta
    & \Aut(\mukaiH(X,\bZ),\,Z\cdot\bC^*)
    \ar[d]^-{\eqref{eqn:isometry-scal}}\ar[r]
    & 0 \\
    0 \ar[r]
    & 2\bZ\ar[r]
    & \frac{2}{m}\bZ\ar[d]\ar[r]^-{e^{-i\pi(-)}}
    & \mu_m\ar[d]\ar[r]
    & 0 \\
    &
    & 0
    & 0.
}
\end{aligned}
\end{equation}
Here the surjectivity of \eqref{eqn:autoeq-scal} is a consequence of the snake lemma.

The same construction can be applied to the symplectic subgroups to get the same diagram. The only difference is that, since $-\mathrm{id}$ is not symplectic, the group $\mu_m$ need not contain $-1$. Thus $m$ may be odd.

\begin{prop}
\label{prop:mod-vs-nomod}
Let $X$ be a K3~surface and fix $\sigma\in\Stab^\dag(X)$. Then there exists an even integer $m$ and a short exact sequence
$$\xymatrix@R=0pt{
    0 \ar[r]
    & \Aut(\Db(X),\sigma) \ar[r]
    & \Aut(\Db(X),\,\sigma\cdot\bC)/\bZ[2] \ar[r]
    & \mu_m \ar[r]
    & 0,
}$$
where $\mu_m\subseteq\bC^*$ is the subgroup of $m$-th roots of unity, and the homomorphism to $\mu_m$ sends the class of the shift functor $[1]$ to $-1\in\mu_m$. The same short exact sequence holds with $\Aut$ replaced by $\Aut_s$. In this case, $m$ may be odd, and the shift functor $[1]$ does not occur since it is not symplectic.
\end{prop}
\begin{proof}
The short exact sequence is obtained by taking the quotient of the middle column by the first column in diagram~\eqref{diag:gepner}.
\end{proof}

\subsection{Abelian surfaces and generic twisted K3~surfaces}
\label{subsec:genK3cat}

Let $X$ be a K3~surface or an abelian surface equipped with a Brauer class $\alpha\in\mathrm{Br}(X)$. Assume that $(X,\alpha)$ is \emph{generic}, meaning that the bounded derived category $\Db(X,\alpha)$ of $\alpha$-twisted coherent sheaves on $X$ contains no spherical objects. All twisted abelian surfaces are generic in this sense. For such a pair $(X,\alpha)$, one can define $\Stab^\dag(X,\alpha)$ and the domains $\cP^+(X,\alpha)$, $\cQ^+(X,\alpha)$ similarly to those in Section~\ref{subsec:prelim_stabK3}; see \cite{HMS08}*{Section~3} for details. Here, no hyperplane sections are removed from these domains since there are no spherical objects.

Due to the absence of spherical objects, the group of autoequivalences and the space of stability conditions become relatively manageable. Indeed, in this case, the forgetful map
$$
    \Stab^\dag(X,\alpha)
    \longrightarrow
    \cP^+(X,\alpha) : \sigma = (Z,\sP)
    \longmapsto Z
$$
is a universal covering map whose group of deck transformations is $\bZ[2]$; see \cite{Bri08}*{Theorem~15.2} for abelian surfaces and \cite{HMS08}*{Theorem~3.15} for generic twisted K3~surfaces. Consequently, the map
$$
    \Stab^\dag(X,\alpha)\longrightarrow\cQ^+(X,\alpha)
$$
is a $\widetilde{\GL}^+(2,\bR)$-fibration. This concrete description allows us to solve the realization problems in these cases.

\begin{prop}
\label{prop:nielsen_twisted-ab-k3}
Let $(X,\alpha)$ be a generic twisted K3~surface or an arbitrary twisted abelian surface. Then the following statements hold:
\begin{itemize}
    \item Every finite subgroup of $\Aut(\Db(X,\alpha))/\bZ[2]$ fixes a point of $\Stab(X,\alpha)/\bC$.
    \item Every finite subgroup of $\Aut(\Db(X,\alpha))$ fixes a point of $\Stab(X,\alpha)$.
\end{itemize}
\end{prop}
\begin{proof}
The numerical Grothendieck group $\cN(X,\alpha)$ has signature $(2,\rho)$ for some positive integer $\rho$. Let $\uO(2,\rho)$ be the orthogonal group of $\cN(X,\alpha)\otimes\bR$ and $\uO^+(2,\rho)\subseteq\uO(2,\rho)$ be the subgroup of elements preserving the orientation of a positive $2$-plane. 
Then there is a diffeomorphism
$$
    \cQ^+(X,\alpha)\,\cong\,
    \uO^+(2,\rho)\,/\,\SO(2)\times\uO(\rho),
$$
where $\SO(2)\times\uO(\rho)$ is a maximal compact subgroup of $\uO^+(2,\rho)$. Indeed, every point $[v]\in\cQ^+(X,\alpha)$ determines an oriented positive-definite two-plane $P_v\subseteq\cN(X,\alpha)\otimes\bR$, whose stabilizer in $\uO^+(2,\rho)$ is given by
$
    \SO(P_v)\times\uO(P^\perp_v)\cong\SO(2)\times\uO(\rho).
$

By \cite{HMS08}*{Theorem~2}, we have a homomorphism
$$
    \eta\colon\Aut(\Db(X,\alpha))/\bZ[2]
    \longrightarrow\uO^+(2,\rho).
$$
Let $G\subseteq\Aut(\Db(X,\alpha))/\bZ[2]$ be a finite subgroup. Then the image $\eta(G)\subseteq\uO^+(2,\rho)$, being finite, is contained in a maximal compact subgroup. All maximal compact subgroups of a Lie group with finitely many connected components are conjugate to each other \cite{HN12}*{Theorem~14.1.3}. Hence,
$$
    \eta(G)\subseteq g_0\,(\SO(2)\times\uO(\rho))\,g_0^{-1}
    \quad\text{for some}\quad
    g_0\in\uO^+(2,\rho).
$$
Thus $g_0\,(\SO(2)\times\uO(\rho))$ corresponds to a point $[v_0]\in\cQ^+(X,\alpha)$ fixed by the action of $G$.

Let $\sigma\in\Stab^\dag(X,\alpha)$ be a lift of $[v_0]\in\cQ^+(X,\alpha)$ and let $\Phi\in\Aut(\Db(X,\alpha))$ be a lift of any element of $G$. Then
$$
    \Phi(\sigma)=\sigma\cdot g
    \qquad\text{for some}\qquad
    g\in\widetilde{\GL}^+(2,\bR).
$$
Let $m$ be a positive integer such that $\Phi^m=[2k]$ for some $k\in\bZ$. Then
$$
    \sigma\cdot2k=\Phi^m(\sigma)=\sigma\cdot g^m.
$$
This implies that $g$ lies in $\bC\subseteq\widetilde{\GL}^+(2,\bR)$ and, more precisely, $g=\frac{2k}{m}$. Thus, we conclude that the point $\overline{\sigma}\in\Stab^\dag(X,\alpha)/\bC$ given by $\sigma$ is fixed by every element of $G$.

The second statement follows from the first by Lemma~\ref{lemma:q2impliesq1}.
\end{proof}

\section{Autoequivalences modulo even shifts}
\label{sec:auteqMod2}

In this section, we determine the structure of the group of autoequivalences modulo even shifts for K3~surfaces of Picard number one. This leads to a complete classification of its finite subgroups and formulas for the numbers of their conjugacy classes. We conclude the section with criteria for the existence of cubic fourfolds associated with such K3~surfaces.

\subsection{Actions of spherical twists in Fricke groups}
\label{subsec:sphericalFricke}

Let $X$ be a K3~surface of Picard number one with a polarization $h\in\NS(X)$ of degree $h^2 = 2n$. The numerical Grothendieck group of $X$ is a lattice of signature $(2,1)$ given by
$$
    \cN(X)\cong H^0(X,\bZ)\oplus \bZ h \oplus H^4(X,\bZ).
$$
Each element of $\cN(X)$ will be denoted by $(r,d,s)$ for some $r,d,s\in\bZ$. With this notation, the Mukai pairing between elements of $\cN(X)$ has the form
$$
    (r_1,d_1,s_1)\cdot (r_2,d_2,s_2)
    = 2n(d_1d_2) - r_1s_2 - r_2s_1.
$$
In this setting, there is a holomorphic map from the hyperbolic upper half-plane to the period domain $\cP^+(X)$, defined by
$$\xymatrix@R=0pt{
    \cH\colonequals\{z\in\bC\mid\mathrm{Im}(z)>0\}\ar[r]
    & \cP^+(X) \\
    z\ar@{|->}[r]
    & e^{zh} = (1, z, nz^2).
}$$
One can verify that this map is injective, and its image forms a section of the $\GL^+(2,\bR)$-fibration $\cP^+(X)\rightarrow\cQ^+(X)$. It therefore induces a biholomorphic map
\begin{equation}
\label{map:hyper-domain}
\xymatrix{
    \cH\ar[r]^-\sim & \cQ^+(X) : z\ar@{|->}[r] & [e^{zh}].
}
\end{equation}

Each autoequivalence induces an isometry of $\cN(X)$ preserving $\cQ^+(X)$, which, via the isomorphism above, induces an isometry of $\cH$ with respect to the standard hyperbolic metric.
This defines a homomorphism
\begin{equation}
\label{eqn:auteq-to-PSL}
\xymatrix{
    \Aut(\Db(X))\ar[r] & \PSL(2,\bR)
}
\end{equation}
Building on \cite{Dol96}*{Theorem~7.1, Remarks~7.2} and \cite{HMS09}*{Corollary~3}, Kawatani identifies the image of this homomorphism with the \emph{Fricke group} $\Gamma_0^+(n)\subseteq\PSL(2,\bR)$ \cite{Kaw14}*{Proposition~2.9}, which is defined as follows. For every integer $n>0$, the group $\PSL(2,\bR)$ contains the distinguished element
$$
    \frn\colonequals
    \begin{pmatrix}
        0 & -\frac{1}{\sqrt{n}} \\
        \sqrt{n} & 0
    \end{pmatrix}
    : \cH\longrightarrow\cH
    ,\quad z\longmapsto -\frac{1}{nz}
$$
which is called the \emph{Fricke involution}. Together with the Hecke congruence subgroup
$$
    \Gamma_0(n)\colonequals\left\{
        \begin{pmatrix}
            a & b \\
            c & d
        \end{pmatrix}\in\PSL(2,\bZ)
        \;\middle|\;
        c\equiv 0\pmod{n}
    \right\},
$$
it generates the Fricke group
$$
    \Gamma_0^+(n) \colonequals \left<
        \Gamma_0(n),\;
        \frn
    \right>\subseteq\PSL(2,\bR).
$$
Note that $\Gamma_0^+(1) = \Gamma_0(1) = \PSL(2,\bZ)$.

\begin{eg}
\label{eg:autoeq_on_hyper}
Here are some examples of actions of autoequivalences on $\cH$. 
With respect to the standard basis $\{(1,0,0),(0,h,0),(0,0,1)\}$ of $\cN(X)\cong H^0(X,\bZ)\oplus\bZ h\oplus H^4(X,\bZ)$, the actions of the tensoring functor $-\otimes\cO_X(1)$ and the spherical twist $T_{\cO_X}$ on $\cN(X)$ are represented, respectively, by
$$
\begin{pmatrix}
    1 & 0 & 0 \\
    1 & 1 & 0 \\
    n & 2n & 1
\end{pmatrix}
\qquad\text{and}\qquad
\begin{pmatrix}
    0 & 0 & -1 \\
    0 & 1 & 0 \\
    -1 & 0 & 0
\end{pmatrix}.
$$
Unless otherwise stated, all $3\times3$ matrices appearing below are written with respect to the same basis.

The matrix for the tensoring functor is computed directly. For the spherical twist, we use the fact that it acts as the reflection along the $(-2)$-vector $(1,0,1)\in\cN(X)$. To illustrate their actions on $\cH$, consider, for example, the action of the spherical twist
$$\xymatrix{
    e^{zh} = (1, z, nz^2)\ar@{|->}[r]
    & (-nz^2, z, -1)
    = -nz^2\left(1, -\frac{1}{nz}, \frac{1}{nz^2}\right)
    = -nz^2e^{-\frac{1}{nz}h}.
}$$
This shows that it induces the Fricke involution $z\mapsto -\frac{1}{nz}$. Similarly, one can verify that the tensoring functor induces the translation $z\mapsto z+1$.
\end{eg}

In the following, we characterize the involutions on $\cH$ induced by spherical twists.

\begin{lemma}
\label{lemma:ref-on-hyper}
Via \eqref{map:hyper-domain}, the reflection along a $(-2)$-vector $\delta = (r,d,s)\in\cN(X)$ induces the following involution
$$
    \begin{pmatrix}
        \sqrt{n}d & -\frac{s}{\sqrt{n}} \\
        \sqrt{n}r & -\sqrt{n}d
    \end{pmatrix}
    =
    \begin{pmatrix}
        s & d \\
        nd & r
    \end{pmatrix}
    \begin{pmatrix}
        0 & -\frac{1}{\sqrt{n}}\\
        \sqrt{n} & 0
    \end{pmatrix}
$$
which lies in the coset $\Gamma_0(n)\frn\subseteq\Gamma_0^+(n)$. Conversely, every involution in $\Gamma_0(n)\frn$ can be written in the form above, thus is induced by the reflection along a $(-2)$-vector in $\cN(X)$.
\end{lemma}
\begin{proof}
The image of $e^{zh}\in\cP^+(X)$ under the reflection equals
\begin{align*}
    & e^{zh} + (e^{zh}\cdot\delta)\delta \\
    & = (1,z,nz^2) + (2ndz - nrz^2 - s)(r,d,s) \\
    & = (1 + 2nrdz - nr^2z^2 - rs,\,
        z + 2nd^2z - nrdz^2 - ds,\,
        nz^2 + 2ndsz - nrsz^2 - s^2).
\end{align*}
We have $\delta^2 = 2nd^2 - 2rs = -2$, or equivalently, $nd^2 - rs = -1$. Substituting this relation into the expression above gives
\begin{align*}
    &\left(
        -n(rz-d)^2,\,
        -(rz-d)(ndz-s),\,
        -(ndz-s)^2
    \right) \\
    &= -n(rz-d)^2\left(
        1,\,
        \frac{ndz-s}{n(rz-d)},\,
        n\left(
            \frac{ndz-s}{n(rz-d)}
        \right)^2
    \right) \\
    &= -n(rz-d)^2\exp\left(\frac{ndz-s}{n(rz-d)}h\right) \\
    &= -n(rz-d)^2\exp\left(\frac{\sqrt{n}dz-\frac{s}{\sqrt{n}}}{\sqrt{n}rz-\sqrt{n}d}h\right).
\end{align*}
This shows that, on $\cQ^+(X)$, the reflection induces the desired transformation. A straightforward computation shows that the transformation decomposes as in the statement, where the first factor belongs to $\Gamma_0(n)$ as $rs - nd^2 = 1$, and the second factor is exactly the Fricke involution $\frn$. Therefore, the transformation belongs to $\Gamma_0(n)\frn$.

To prove the converse, first note that every element of $\Gamma_0(n)\frn$ can be written as
$$
    \begin{pmatrix}
        s & d \\
        n\ell & r
    \end{pmatrix}
    \begin{pmatrix}
        0 & -\frac{1}{\sqrt{n}}\\
        \sqrt{n} & 0
    \end{pmatrix}
    =
    \begin{pmatrix}
        \sqrt{n}d & -\frac{s}{\sqrt{n}} \\
        \sqrt{n}r & -\sqrt{n}\ell
    \end{pmatrix}
    \quad\text{for some}\quad
    \begin{pmatrix}
        s & d \\
        n\ell & r
    \end{pmatrix}\in\Gamma_0(n).
$$
Being an involution means that
$$
    \begin{pmatrix}
        \sqrt{n}d & -\frac{s}{\sqrt{n}} \\
        \sqrt{n}r & -\sqrt{n}\ell
    \end{pmatrix}^2
    =
    \begin{pmatrix}
        nd^2-rs & -s(d-\ell) \\
        nr(d-\ell) & n\ell^2-rs
    \end{pmatrix}
    = \pm
    \begin{pmatrix}
        1 & 0 \\
        0 & 1
    \end{pmatrix}.
$$
We claim that $d = \ell$. Assume, to the contrary, that $d\neq\ell$, or equivalently, that $d-\ell \neq 0$. The relation above then gives $r=s=0$, and consequently $nd^2 = n\ell^2 = \pm 1$. This is possible only when $n = 1$ and $(d,\ell) = \pm(1,-1)$. The original element of $\Gamma_0(n)\frn$ would then be the identity in $\PSL(2,\bZ)$, and hence would not be an involution. Therefore, $d = \ell$. It follows from the first part of the lemma that such an involution can be recovered by the reflection along the $(-2)$-vector $\delta = (r,d,s)\in\cN(X)$.
\end{proof}

\begin{prop}
\label{prop:sph-in-Fricke}
Let $X$ be a K3~surface of Picard number one and degree $2n$. Then the action of a spherical twist on $\cQ^+(X)\cong\cH$ corresponds to an involution in the coset
$$
    \Gamma_0(n)\frn = \begin{cases}
        \Gamma_0(1) = \PSL(2,\bZ) & \text{if}\quad n=1,\\
        \Gamma_0^+(n)\setminus\Gamma_0(n) & \text{if}\quad n\geq 2.
    \end{cases}
$$
Conversely, every involution in this coset is induced by a spherical twist.
\end{prop}
\begin{proof}
Given a spherical object $\cE$, the action of the spherical twist $T_{\cE}$ on $\cN(X)$ corresponds to the reflection along the $(-2)$-vector $v(\cE)\in\cN(X)$. By Lemma~\ref{lemma:ref-on-hyper}, such a reflection acts on $\cQ^+(X)\cong\cH$ as an involution in $\Gamma_0(n)\frn$. Conversely, every involution in $\Gamma_0(n)\frn$ is induced by the reflection along a $(-2)$-vector $\delta\in\cN(X)$ by the same lemma. Every $(-2)$-vector in $\cN(X)$ is the Mukai vector of a spherical object; see \cite{Kaw19}*{Remark~2.10} or \cite{Huy16_K3Lect}*{Remark~10.3.3}. Therefore, $\delta = v(\cE)$ for some spherical object $\cE$, and the reflection along $\delta$ can be recovered by $T_{\cE}$.
\end{proof}

\subsection{Fundamental group of the period space}
\label{subsec:orbfun_periodsp}

Recall that $\cQ^+_0(X)$ is obtained from $\cQ^+(X)$ by removing the hyperplane sections orthogonal to $(-2)$-vectors in $\cN(X)$. When $X$ has Picard number one, these hyperplane sections are points, which we call \emph{(-2)-points}. By Lemma~\ref{lemma:ref-on-hyper}, the isomorphism $\cH\cong\cQ^+(X)$ in \eqref{map:hyper-domain} identifies the $(-2)$-points with the points of $\cH$ fixed by involutions in the coset $\Gamma_0(n)\frn$. Thus,
$$
    \cQ^+_0(X)\cong\cH\setminus\left\{
        \text{fixed points of involutions in }\Gamma_0(n)\frn
    \right\}.
$$
We call the orbifold $\Gamma_0^+(n)\git\cQ^+_0(X)$ the \emph{period space}, as it parametrizes stability conditions modulo $\GL^+(2,\bR)$-actions and autoequivalences. Our goal is to derive an explicit formula for its orbifold fundamental group. We use \cite{Car22} as our main reference for orbifolds and orbifold fundamental groups.

Consider the orbifold
$$
    Y_0(n)\colonequals\Gamma_0(n)\git\cH.
$$
A point of $Y_0(n)$ is called an \emph{elliptic point} if a lift of it to $\cH$ has a nontrivial stabilizer. The \emph{order} of an elliptic point is the order of this stabilizer. The underlying topological space of $Y_0(n)$ is a compact Riemann surface with finitely many points removed. These missing points are \emph{cusps}, namely the images of the points of $\bR\cup\{i\infty\}$ fixed by parabolic elements of $\Gamma_0(n)$. Filling in the cusps as ordinary points yields the orbifold $\overline{Y_0(n)}$. We define
\begin{itemize}
\setlength\itemsep{0pt}
\item[] $g = g(n)\colonequals$ the genus of the compact Riemann surface underlying $\overline{Y_0(n)}$,
\item[] $\nu_i=\nu_i(n)\colonequals$ the number of its elliptic points of order $i$,
\item[] $\nu_\infty=\nu_\infty(n)\colonequals$ the number of its cusps.
\end{itemize}
The genus $g$ can be computed using \cite{Shi71}*{Proposition~1.40}. It is known that $\nu_i\neq 0$ only for $i=2,3,\infty$. These invariants can be computed using \cite{Shi71}*{Proposition~1.43}, and the formulas for $\nu_2$ and $\nu_3$ given there can be rewritten as
\begin{equation}
\label{eqn:nu2}
\nu_2 = \begin{cases}
    0 & \text{if}
    \quad n \text{ is divisible by } 4 \text{ or a prime } p\equiv 3\pmod{4}, \\
    2^{a} & \text{if}
    \quad n = 2^{\ell}p_1^{k_1}\cdots p_a^{k_a},
    \quad \ell\in\{0,1\},
    \quad p_i\equiv 1\pmod{4}\text{ are primes},
\end{cases}
\end{equation}
\begin{equation}
\label{eqn:nu3}
\nu_3 = \begin{cases}
    0 & \text{if}
    \quad n \text{ is divisible by } 9 \text{ or a prime } p\equiv 2\pmod{3}, \\
    2^{a} & \text{if}
    \quad n = 3^{\ell}p_1^{k_1}\cdots p_a^{k_a},
    \quad \ell\in\{0,1\},
    \quad p_i\equiv 1\pmod{3}\text{ are primes}.
\end{cases}
\end{equation}

Similarly, consider
$$
    Y_0^+(n)\colonequals\Gamma_0^+(n)\git\cH
$$
and denote its compactification by $\overline{Y_0^+(n)}$. We define
\begin{itemize}
\setlength\itemsep{0pt}
\item[] $g^+ = g^+(n)\colonequals$ the genus of the compact Riemann surface underlying $\overline{Y_0^+(n)}$,
\item[] $\nu_i^+=\nu_i^+(n)\colonequals$ the number of its elliptic points of order $i$,
\item[] $\nu_\infty^+=\nu_\infty^+(n)\colonequals$ the number of its cusps.
\end{itemize}
As a first step, we derive formulas for these invariants in terms of those of $Y_0(n)$. When $n=1$, the curves $Y_0(1)$ and $Y_0^+(1)$ are isomorphic. Since $\nu_2 = \nu_3 = \nu_\infty = 1$ and $g = 0$, it follows that $\nu_2^+ = \nu_3^+ = \nu_\infty^+ = 1$ and $g^+ = 0$, while all other invariants vanish. For $n\geq 2$, the index-two inclusion $\Gamma_0(n)\subseteq\Gamma_0^+(n)$ induces a ramified double cover
\begin{equation}
\label{eqn:mod-curve_cover}
\xymatrix{
    \overline{Y_0(n)}\ar[r] & \overline{Y_0^+(n)}.
}
\end{equation}
Ramification points of this map are classified in Fricke's book \cite{Fri28}*{II, 4, \S3}. For a modern account, see \cite{DHNT20}*{Lemma~2.11} and the paragraph following it. When $n = 2,3,4$, these points are respectively
\begin{itemize}
\setlength\itemsep{0pt}
\setlength{\itemindent}{2.5em}
\item[$(n=2)$] an ordinary point and an elliptic point of order $2$,
\item[$(n=3)$] an ordinary point and an elliptic point of order $3$,
\item[$(n=4)$] an ordinary point and a cusp.
\end{itemize}
To describe the general case, let $h(D)$ be the class number of primitive integral quadratic forms of discriminant $D$. In general, this number can be computed by Dirichlet's class number formula (cf. \cite{Dav00}*{Chapter~6}). In our setting, only $D < 0$ needs to be considered, and there are a number of algorithms for computing the class number under this condition (cf. \cite{Coh93}*{Section~5.3}). Using the class number, we define
$$
    \xi = \xi(n)\colonequals\begin{cases}
        h(-4n) & \text{if}\quad n\not\equiv 3\pmod{4}, \\
        h(-n) + h(-4n) & \text{if}\quad n\equiv 3\pmod{4}.
    \end{cases}
$$
For each $n\geq 1$, Fricke verified that \eqref{eqn:mod-curve_cover} is ramified at $\xi$ ordinary points and that, when $n\geq 5$, all ramification points are ordinary points. This helps determine the distribution of elliptic points and cusps on $Y_0^+(n)$.

\begin{lemma}
\label{lemma:fricke-curve-invariants}
We have $g^+ = 0$ when $n=1,2,3,4$. The other nonzero invariants in these cases, together with the branch loci of \eqref{eqn:mod-curve_cover}, are as follows:
\begin{itemize}
\setlength\itemsep{0pt}
\setlength{\itemindent}{2.5em}
\item[$(n=1)$] $\nu_2^+ = \nu_3^+ = \nu_\infty^+ = 1$, no branch locus in this case.
\item[$(n=2)$] $\nu_2^+ = \nu_4^+ = \nu_\infty^+ = 1$, where the two elliptic points form the branch locus.
\item[$(n=3)$] $\nu_2^+ = \nu_6^+ = \nu_\infty^+ = 1$, where the two elliptic points form the branch locus.
\item[$(n=4)$] $\nu_2^+ = 1$, $\nu_\infty^+ = 2$, where the elliptic point and one cusp form the branch locus.
\end{itemize}

For $n\geq 5$, the only invariants that may be nonzero are $\nu_2^+$, $\nu_3^+$, $\nu_\infty^+$, $g^+$, and they are given by
$$
    \nu_2^+ = \frac{\nu_2}{2} + \xi,
    \qquad
    \nu_3^+ = \frac{\nu_3}{2},
    \qquad
    \nu_\infty^+ = \frac{\nu_\infty}{2},
    \qquad
    2g^+ = g + 1 - \frac{\xi}{2}.
$$
In this case, the branch locus consists of $\xi$ elliptic points of order~$2$.
\end{lemma}
\begin{proof}
The case $n = 1$ follows from the fact that $Y_0^+(1)\cong Y_0(1)$. In particular, there is no branch locus in this case.

For $n\geq 2$, the deck transformation of the double cover \eqref{eqn:mod-curve_cover} sends ordinary points to ordinary points, elliptic points to elliptic points of the same order, and cusps to cusps. Away from ramification locus, map~\eqref{eqn:mod-curve_cover} sends a pair of points to a point of the same type. Along ramification locus, it sends an ordinary point to an elliptic point of order~$2$, an elliptic point of order $i$ to an elliptic point of order~$2i$, and a cusp to a cusp.

For $n = 2,3,4$, map~\eqref{eqn:mod-curve_cover} is ramified at exactly two points. Since $g = 0$, the Riemann--Hurwitz formula gives $g^+ = 0$. The remaining invariants and branch loci are as follows:

\begin{itemize}
\item The curve $Y_0(2)$ (resp. $Y_0(3)$) has one elliptic point of order~$2$ (resp. order~$3$) and two cusps. Map~\eqref{eqn:mod-curve_cover} is ramified at an ordinary point and at the elliptic point, and unramified at the two cusps. Consequently, $Y_0^+(2)$ (resp. $Y_0^+(3)$) has one elliptic point of order~$2$, one elliptic point of order~$4$ (resp. order~$6$), and one cusp. Thus $\nu_2^+ = 1$, $\nu_4^+ = 1$ (resp. $\nu_6^+ = 1$), $\nu_\infty^+ = 1$, and the two elliptic points form the branch locus.

\item The curve $Y_0(4)$ has three cusps. Map~\eqref{eqn:mod-curve_cover} is ramified at an ordinary point and one of the cusps, and identifies the remaining two cusps. Consequently, $Y_0^+(4)$ has one elliptic point of order~$2$ and two cusps. Hence $\nu_2^+ = 1$ and $\nu_\infty^+ = 2$. In this case, the branch locus consists of the elliptic point and one of the two cusps.
\end{itemize}

When $n\geq 5$, the ramification locus of \eqref{eqn:mod-curve_cover} consists of $\xi$ ordinary points. Thus, the map produces $\xi$ new elliptic points of order~$2$ and halves the number of elliptic points and cusps of each existing type. There is no elliptic point of order other than $2$ and $3$, since $Y_0(n)$ has no such point for any $n$. Together with the Riemann--Hurwitz formula, this gives the formulas in the statement. In this case, the $\xi$ new elliptic points form the branch locus.
\end{proof}

Recall that a $(-2)$-point of $\cH$ is a point fixed by an involution in $\Gamma_0(n)\frn$. The image of such a point in the quotient $Y_0^+(n) = \Gamma_0^+(n)\git\cH$ will also be called a $(-2)$-point. We give a precise description of their distribution on $Y_0^+(n)$.

\begin{lemma}
\label{lemma:(-2)-points}
The only $(-2)$-point of $Y_0^+(1)$ is the elliptic point of order~$2$. For $n\geq 2$, the set of $(-2)$-points on $Y_0^+(n)$ coincides with the set of branch points of \eqref{eqn:mod-curve_cover} whose preimages are ordinary points or elliptic points of odd order. Thus, it consists of
\begin{itemize}
\setlength\itemsep{0pt}
\item the unique elliptic point of order~$2$ when $n=2,4$,
\item the two elliptic points, one of order~$2$ and the other of order~$6$, when $n=3$,
\item $\xi$ elliptic points of order~$2$ when $n\geq5$.
\end{itemize}
\end{lemma}
\begin{proof}
For $n=1$, we have $\Gamma_0^+(1) = \Gamma_0(1)$, so the set of $(-2)$-points in $\cH$ coincides with the set of elliptic points of order~$2$, which forms a single point of $Y_0^+(1) = \Gamma_0^+(1)\git\cH$.

Suppose that $n\geq 2$ and let $p\in\cH$ be a $(-2)$-point. By definition, there exists an involution $w\in\Gamma_0^+(n)\setminus\Gamma_0(n)$ fixing $p$. Since $w$ represents the deck transformation of \eqref{eqn:mod-curve_cover}, the image of $p$ in $Y_0^+(n)$ is a branch point of this double cover. Assume, to the contrary, that its preimage in $Y_0(n)$ is an elliptic point of even order. Then its stabilizer $\Gamma_0(n)_p\subseteq\Gamma_0(n)$ is cyclic of even order, and so contains an involution $g$. Both $w$ and $g$ are elliptic elements of order~$2$ fixing $p$, which would imply $w = g\in\Gamma_0(n)$, contradiction. Therefore, the image of a $(-2)$-point in $Y_0^+(n)$ is a branch point whose preimage is an ordinary point or an elliptic point of odd order.

Conversely, let $p\in\cH$ be a point whose image in $Y_0(n)$ is an ordinary point or an elliptic point of odd order and lies in the ramification locus of \eqref{eqn:mod-curve_cover}. Then $\Gamma_0(n)_p\subseteq\Gamma_0(n)$ is cyclic of odd order $m$, where $m=1$ if the image is an ordinary point, whereas $\Gamma_0^+(n)_p\subseteq\Gamma_0^+(n)$ is cyclic of order~$2m$. In this situation, there exists an involution $w\in\Gamma_0^+(n)_p\setminus\Gamma_0(n)_p$, which implies that $p$ is a $(-2)$-point. Therefore, a branch point whose preimage is an ordinary point or an elliptic point of odd order is a $(-2)$-point.

Observe that, for a branch point $p\in Y_0^+(n)$, if its preimage in $Y_0(n)$ is an ordinary point or an elliptic point of odd order, then $p$ is an elliptic point of order $2m$ for some odd $m$, and vice versa. Together with Lemma~\ref{lemma:fricke-curve-invariants} and what we proved above, this gives the precise description of the set of $(-2)$-points in each case.
\end{proof}

We are now ready to compute the fundamental group of $\Gamma_0^+(n)\git\cQ^+_0(X)$. Note that this space contains a cusp given by $i\infty$ for every $n$. A counterclockwise loop around this cusp corresponds to the element $z\mapsto z+1$ in $\Gamma_0^+(n)$ induced by the functor $-\otimes\cO_X(1)$. In the following, we will call a cusp \emph{not} given by $i\infty$ a \emph{real cusp}.

\begin{lemma}
\label{lemma:orbfund}
Let $X$ be a K3~surface of Picard number one and degree $2n$. Then the fundamental group of $\Gamma_0^+(n)\git\cQ^+_0(X)$ can be expressed as
$$
    \pi_1^{\rm orb}(\Gamma_0^+(n)\git\cQ^+_0(X))
    \cong\begin{cases}
        \mathring{\bZ}*\bZ_3
        & \text{if}\quad n=1 \\
        \mathring{\bZ}*\bZ_4
        & \text{if}\quad n=2 \\
        \mathring{\bZ}*\mathring{\bZ}
        & \text{if}\quad n=3 \\
        \mathring{\bZ}*\check{\bZ}
        & \text{if}\quad n=4 \\
        \bZ_2^{*\frac{\nu_2}{2}}
        * \bZ_3^{*\frac{\nu_3}{2}}
        * \mathring{\bZ}\strut^{*\xi}
        * \check{\bZ}\strut^{*\left(
                \frac{\nu_\infty}{2}-1
            \right)}
        * \bZ\strut^{*\left(
                g + 1 - \frac{\xi}{2}
            \right)}
        & \text{if}\quad n\geq 5
    \end{cases}
$$
where a copy of $\bZ$ is decorated as $\mathring{\bZ}$ (resp. $\check{\bZ}$) if and only if it is generated by a loop around a $(-2)$-point (resp. a real cusp).
\end{lemma}
\begin{proof}
The topological space underlying $\Gamma_0^+(n)\git\cQ^+_0(X)$ is a compact Riemann surface with finitely many points removed, called \emph{punctures}. For $n=1,2,3,4$, Lemmas~\ref{lemma:fricke-curve-invariants} and \ref{lemma:(-2)-points} give the following description of its elliptic points and punctures:
\begin{itemize}
\item When $n=1$, there is one elliptic point of order~$3$, and when $n=2$, there is one elliptic point of order~$4$. In both cases, the underlying topological space is a sphere with two punctures. One corresponds to a $(-2)$-point, and the other is the cusp given by $i\infty$.
\item When $n=3,4$, there are no elliptic points. In both cases, the underlying topological space is a sphere with three punctures. For $n=3$, two punctures correspond to $(-2)$-points, and the third is the cusp given by $i\infty$. For $n=4$, one puncture corresponds to a $(-2)$-point, and the other two are cusps, one of which is given by $i\infty$.
\end{itemize}
In each of these cases, the formula for the fundamental group follows from the Seifert--Van Kampen theorem for orbifolds \cite{Car22}*{Theorem~2.2.3}.

To derive the formula for $n\geq 5$, we choose an open connected subset
$$
    D\subseteq \Gamma_0^+(n)\git\cQ^+_0(X),
$$
containing all elliptic points and punctures, whose underlying topological space becomes simply connected after compactification by filling in the punctures. The complement of $D$ is covered by an open subset
$$
    S\subseteq \Gamma_0^+(n)\git\cQ^+_0(X),
$$
which is an ordinary topological surface homotopy equivalent to a surface of genus $g^+$ with one puncture, such that $D\cap S$ is homotopy equivalent to a circle.

\begin{figure}[ht]
\centering
\begin{tikzpicture}[
    line cap=round,
    line join=round,
    every node/.style={font=\small}
]

\draw[line width=1pt, fill=black!10]
    (0,0) ellipse (5 and 2);

\draw[line width=1pt, fill=white]
    (-1.05,1.25)
    .. controls (-0.7,1.45) and (-0.3,1.45) .. (0.05,1.25)
    .. controls (-0.25,1) and (-0.75,1)
    .. cycle;
\draw[line width=1pt]
    (0.2,1.35)
    .. controls (-0.25,0.95) and (-0.75,0.95)
    .. (-1.2,1.35);

\draw[line width=1pt, fill=white]
    (1.5,-0.5)
    .. controls (1.85,-0.3) and (2.25,-0.3) .. (2.6,-0.5)
    .. controls (2.3,-0.75) and (1.8,-0.75)
    .. cycle;
\draw[line width=1pt]
    (2.75,-0.4)
    .. controls (2.3,-0.8) and (1.8,-0.8)
    .. (1.35,-0.4);

\draw[line width=1.75pt, rounded corners=20pt, fill=black!10]
    (-0.3,0.7)
    -- (2.2,1.6)
    -- (4.5,0.3)
    -- (3.5,-1.3)
    -- (2.8,0.5)
    -- (-3.1,-1.5)
    -- (-3.8,1.3)
    -- cycle;

\draw[line width=0.5pt, dotted, rounded corners=18pt, fill=white]
    (-0.3,0.5)
    -- (2.2,1.4)
    -- (4.3,0.3)
    -- (3.5,-1.1)
    -- (2.9,0.7)
    -- (-3.0,-0.4)
    -- (-3.6,1.1)
    -- cycle;

\node at (-3,0.7) {$\times$};
\node at (-2.5,0) {$\times$};
\node at (-2,0.5) {$\times$};
\node at (3.8,0.2) {$\times$};

\draw[line width=0.75pt] (2,1) circle (0.075);
\draw[line width=0.75pt] (3.55,-0.5) circle (0.075);

\node[anchor=south west] (ell) at (-8,0.5)
    {elliptic points};
\draw[->,line width=0.6pt]
    (ell.east) -- (-3.15,0.75);

\node[anchor=south west] (pun) at (4.5,1.5)
    {punctures};
\draw[->,line width=0.6pt]
    (pun.west) -- (2.15,1.1);

\node[font=\large] at (-1,0) {$D$};
\node[font=\large] at (0.2,-1.2) {$S$};

\node[font=\normalsize] (cap) at (-2.5,-0.7) {$D\cap S$};

\end{tikzpicture}
\end{figure}

The Seifert--Van Kampen theorem implies that
\begin{equation}
\label{eqn:amalg_prod}
    \pi_1^{\rm orb}(\Gamma_0^+(n)\git\cQ^+_0(X))
    \;\cong\;
    \pi_1^{\rm orb}(D)
    *_{\pi_1(D\cap S)}
    \pi_1(S).
\end{equation}
Note that $\pi_1(S)\cong\bZ\strut^{*2g^+}$. By Lemmas~\ref{lemma:fricke-curve-invariants} and \ref{lemma:(-2)-points}, the disk $D$ contains $\frac{\nu_2}{2}$ elliptic points of order~$2$ and $\frac{\nu_3}{2}$ elliptic points of order~$3$. It also contains $\xi+\frac{\nu_\infty}{2}$ punctures, of which $\xi$ correspond to $(-2)$-points and the remaining ones to cusps. The orbifold fundamental group of a sufficiently small neighborhood of an elliptic point of order~$i$ (resp. a puncture) is $\bZ_i$ (resp. $\bZ$). Applying the Seifert--Van Kampen theorem repeatedly, we obtain
\begin{equation}
\label{eqn:fundD_freeProd}
    \pi_1^{\rm orb}(D)
    \cong
    \bZ_2^{*\frac{\nu_2}{2}}
    * \bZ_3^{*\frac{\nu_3}{2}}
    * \mathring{\bZ}\strut^{*\xi}
    * \check{\bZ}\strut^{*\left(
            \frac{\nu_\infty}{2}-1
        \right)}
    * \bZ,
\end{equation}
where the last copy of $\bZ$ is generated by a loop around the cusp given by $i\infty$.

Let $\{\gamma_1,\dots,\gamma_N\}$, where $N = \xi+\frac{1}{2}(\nu_2+\nu_3+\nu_\infty)$, be a free-product basis for \eqref{eqn:fundD_freeProd}, with each $\gamma_i$ represented by a counterclockwise loop around an elliptic point or a puncture. After re-indexing, we may assume that $\gamma_N$ corresponds to the cusp represented by $i\infty$. In this setting, $\pi_1(D\cap S)\cong\bZ$ is generated by
$$
    \gamma_D\colonequals\prod_{i=1}^N\gamma_i.
$$
Replacing $\gamma_N$ by $\gamma_D$ gives a generating set $\{\gamma_1,\dots,\gamma_{N-1},\gamma_D\}$ for $\pi_1^{\rm orb}(D)$. Since the relations among these generators remain the same, this set is again a free-product basis. Express $\pi_1^{\rm orb}(D)$ in terms of this basis with the factors ordered as in \eqref{eqn:fundD_freeProd}. Then, in the amalgamated product \eqref{eqn:amalg_prod}, the $\bZ$-factor generated by $\gamma_D$ is identified with the image of $\pi_1(D\cap S)$ in $\pi_1(S)$. Thus,
\begin{align*}
    \pi_1^{\rm orb}(\Gamma_0^+(n)\git\cQ^+_0(X))
    & \cong\left(
        \bZ_2^{*\frac{\nu_2}{2}}
        * \bZ_3^{*\frac{\nu_3}{2}}
        * \mathring{\bZ}\strut^{*\xi}
        * \check{\bZ}\strut^{*\left(
            \frac{\nu_\infty}{2}-1
        \right)}
        * \bZ
    \right)
    *_{\bZ}
    \bZ\strut^{*2g^+} \\
    & \cong\bZ_2^{*\frac{\nu_2}{2}}
    * \bZ_3^{*\frac{\nu_3}{2}}
    * \mathring{\bZ}\strut^{*\xi}
    * \check{\bZ}\strut^{*\left(
            \frac{\nu_\infty}{2}-1
        \right)}
    * \bZ\strut^{*2g^+}.
\end{align*}
The desired expression now follows from the fact that $2g^+ = g + 1 - \frac{\xi}{2}$.
\end{proof}

\begin{rmk}
\label{rmk:FrickeStr}
The map
$
    \cQ^+(X)\longrightarrow\Gamma_0^+(n)\git\cQ^+(X)
$
is a universal covering map with deck transformation group $\Gamma_0^+(n)$. Thus $\Gamma_0^+(n)$ is isomorphic to $\pi_1^{\rm orb}(\Gamma_0^+(n)\git\cQ^+(X))$ (cf. \cite{Car22}*{Proposition~2.3.5~(i)}). Following the same argument as in the proof of Lemma~\ref{lemma:orbfund}, one can deduce that
$$
    \Gamma_0^+(n)
    \cong\pi_1^{\rm orb}(\Gamma_0^+(n)\git\cQ^+(X))
    \cong\begin{cases}
        \bZ_2\ast\bZ_3 & \text{if}\quad n=1 \\
        \bZ_2\ast\bZ_4 & \text{if}\quad n=2 \\
        \bZ_2\ast\bZ_6 & \text{if}\quad n=3 \\
        \bZ_2\ast\bZ & \text{if}\quad n=4 \\
        \bZ_2^{*\left(
                \frac{\nu_2}{2}+\xi
            \right)}
        * \bZ_3^{*\frac{\nu_3}{2}}
        * \bZ\strut^{*\left(
                g + \frac{\nu_\infty - \xi}{2}
            \right)}
        & \text{if}\quad n\geq 5.
    \end{cases}
$$
The inclusion $\cQ^+_0(X)\subseteq\cQ^+(X)$ induces a surjective homomorphism
$$\xymatrix{
    \pi_1^{\rm orb}(\Gamma_0^+(n)\git\cQ^+_0(X)) \ar@{->>}[r]
    & \pi_1^{\rm orb}(\Gamma_0^+(n)\git\cQ^+(X))
}$$
which transforms a $\mathring{\bZ}$ factor into a $\bZ_2$ factor and leaves all the other factors invariant. Its kernel is the free product of all the subgroups $2\mathring{\bZ}\subseteq\mathring{\bZ}$ and their conjugates, which is canonically isomorphic to $\pi_1(\cQ^+_0(X))$.
\end{rmk}

\subsection{Classifying finite subgroups modulo even shifts}
\label{subsec:finiteSubgpMod2}

For a K3~surface $X$ of Picard number one and degree~$2n$, the space $\Stab^\dag(X)$ is simply connected and invariant under the actions of autoequivalences \cite{BB17}*{Theorem~1.3}, which gives a universal covering map
\begin{equation}
\label{eqn:univ-cover}
\xymatrix{
    \Stab^\dag(X)/\widetilde{\GL}^+(2,\bR)\ar[r]
    & \Gamma_0^+(n)\git\cQ^+_0(X).
}
\end{equation}
When $n\geq 2$, we obtain an isomorphism
$$
    \Aut_s(\Db(X))/\bZ[2]
    \;\cong\;
    \pi_1^{\rm orb}(\Gamma_0^+(n)\git\cQ^+_0(X))
$$
where the right-hand side is the group of deck transformations of the above covering map \cite{BB17}*{Remark~7.2}. Lemma~\ref{lemma:orbfund} then implies that
\begin{equation}
\label{eqn:sympautoeq_n-not-1}
\begin{aligned}
    \Aut_s(\Db(X))/\bZ[2]
    \cong\begin{cases}
        \bZ*\bZ_4
        & \text{if}\quad n=2 \\
        \bZ*\bZ
        & \text{if}\quad n=3,4 \\
        \bZ_2^{*\frac{\nu_2}{2}}
        * \bZ_3^{*\frac{\nu_3}{2}}
        * \bZ\strut^{*\left(
                g + \frac{\nu_\infty + \xi}{2}
            \right)}
        & \text{if}\quad n\geq 5.
    \end{cases}
\end{aligned}
\end{equation}

The case $n=1$ requires separate treatment. Here, the covering involution of the double cover $X\to\bP^2$ induces an anti-symplectic element
$
    \iota\in\Aut(\Db(X)),
$
which lives in the center \cite{BK22}*{Example~8.4~(i)} and acts trivially on $\Stab^\dag(X)$ \cite{Huy12}*{Lemma~A.3}. Thus the composition $\iota[1]$ is symplectic with $(\iota[1])^2 = [2]$. Combining these observations, we obtain an isomorphism
$$
    \Aut_s(\Db(X))/\bZ(\iota[1])
    \;\cong\;
    \pi_1^{\rm orb}(\Gamma_0^+(1)\git\cQ^+_0(X)).
$$
It then follows from Lemma~\ref{lemma:orbfund} that
\begin{equation}
\label{eqn:sympautoeq_n=1}
    \Aut_s(\Db(X))/\bZ(\iota[1])
    \cong
    \bZ*\bZ_3.
\end{equation}
This formula and \eqref{eqn:sympautoeq_n-not-1} enable us to classify finite subgroups of symplectic autoequivalences modulo $\bZ[2]$ up to conjugation.

\begin{prop}
\label{prop:finite-subgp}
Let $X$ be a K3~surface of Picard number one and degree $2n$. Then every maximal finite subgroup $G_s\subseteq\Aut_s(\Db(X))/\bZ[2]$ is isomorphic to
$$
    G_s\cong\begin{cases}
        \bZ_6 & \text{if}\quad n=1 \\
        \bZ_4 & \text{if}\quad n=2 \\
        0 & \text{if}\quad n=3,4 \\
        0,\; \bZ_2, \;\text{or }\; \bZ_3 & \text{if}\quad n\geq 5.
    \end{cases}
$$
Moreover,
\begin{itemize}
    \item when $n=1,2$, there is a unique such subgroup up to conjugation.
    \item when $n\geq 5$, there are exactly $\frac{\nu_2}{2}$ such subgroups isomorphic to $\bZ_2$ and exactly $\frac{\nu_3}{2}$ such subgroups isomorphic to $\bZ_3$, up to conjugation.
\end{itemize}
Finally, every nontrivial subgroup of such a $G_s$ fixes a unique point of $\cQ^+_0(X)$, unless it is equal to $\left<\iota[1]\right>$ when $n=1$.
\end{prop}

\begin{proof}
Suppose that $n\geq 2$. As a consequence of Kurosh's theorem (cf. \cite{Ser03}*{Chapter~I, Section~5.5}), the subgroup $G_s$, if nontrivial, is conjugate to one of the finite cyclic factors in the free product \eqref{eqn:sympautoeq_n-not-1}. This yields the classification of isomorphism types and the numbers of conjugacy classes for $n\geq 2$.

Now assume that $n = 1$. In this case, we have a short exact sequence
$$\xymatrix{
    0 \ar[r]
    & \bZ_2(\iota[1]) \ar[r]
    & \Aut_s(\Db(X))/\bZ[2] \ar[r]^-f
    & \Aut_s(\Db(X))/\bZ(\iota[1]) \ar[r]
    & 0.
}$$
Let $C_3\cong\bZ_3$ be the finite cyclic factor in the free product \eqref{eqn:sympautoeq_n=1}. Then its preimage under $f$ fits into the short exact sequence
$$\xymatrix{
    0 \ar[r]
    & \bZ_2(\iota[1]) \ar[r]
    & C_6\colonequals f^{-1}(C_3) \ar[r]^-f
    & C_3 \ar[r]
    & 0.
}$$
The group $C_6$ has order~$6$, so it is isomorphic to either $\bZ_6$ or the symmetric group $S_3$. It follows that $C_6\cong\bZ_6$ because $S_3$ has a trivial center. According to Kurosh's theorem, every nontrivial finite subgroup of $\Aut_s(\Db(X))/\bZ(\iota[1])$ is conjugate to $C_3$. Therefore, there exists $\alpha\in\Aut_s(\Db(X))/\bZ[2]$ such that
$$
    f(\alpha)\cdot f(G_s)\cdot f(\alpha)^{-1}
    \subseteq C_3
    \qquad\text{whence}\qquad
    \alpha\cdot G_s\cdot\alpha^{-1}
    \subseteq C_6.
$$
Because $G_s$ is maximal, the latter inclusion is an equality. This proves that every maximal finite subgroup of $\Aut_s(\Db(X))/\bZ[2]$ is isomorphic to $\bZ_6$ and is unique up to conjugation when $n=1$.

Finally, a subgroup as in the last statement has a unique fixed point in $\cQ^+_0(X)$, since its image in $\Gamma_0^+(n)$ is generated by an elliptic element, which fixes a unique point of $\cQ^+_0(X)$.
\end{proof}

\begin{eg}
\label{eg:CanonacoKarp}
Canonaco--Karp \cite{CK08} proved that, if $X$ is a smooth hypersurface in the weighted projective space $\bP(w_0,\dots,w_m)$, then the composition of autoequivalences
$$
    \Theta\colonequals(-\otimes\cO_X(1))\circ T_{\cO_X}
$$
satisfies the relation $\Theta^w = [2]$ where $w\colonequals\sum_{i=0}^mw_i$. If $X$ is a K3~surface of Picard number one and degree $2$ or $4$, this relation gives:
\begin{itemize}
    \item If $X$ has degree~$2$, then it can be realized as a hypersurface in the weighted projective space $\bP(3,1,1,1)$, so we get $\Theta^6 = [2]$ in this case.
    \item If $X$ has degree~$4$, that is, a quartic hypersurface in $\bP^3$, then we obtain $\Theta^4 = [2]$.
\end{itemize}
In both cases, every maximal finite subgroup of $\Aut_s(\Db(X))/\bZ[2]$ is generated by $\Theta$ up to conjugation by Proposition~\ref{prop:finite-subgp}. In the degree~$2$ case, one can derive the equation
$$
    \Theta^3\equiv\iota[1]\pmod{\bZ[2]}
$$
using the same lemma and the fact that the autoequivalences on both sides are symplectic and involutive modulo $\bZ[2]$.
\end{eg}

\begin{lemma}
\label{lemma:symp-vs-nonsymp}
Let $X$ be a K3~surface with odd Picard number. Then
$$
    \Aut(\Db(X))/\bZ[2]
    \;\cong\;
    (\Aut_s(\Db(X))/\bZ[2])\times\bZ_2[1].
$$
In particular, every subgroup $G\subseteq\Aut(\Db(X))/\bZ[2]$ that contains $\bZ_2[1]$ decomposes as
$$
    G\cong G_s\times\bZ_2[1],
$$
where $G_s\subseteq G$ is the subgroup of symplectic elements.
\end{lemma}

\begin{proof}
Let $T(X)$ be the transcendental lattice of $X$. Our hypothesis implies that the only Hodge isometries on $T(X)$ are $\pm 1$ \cite{Huy16_K3Lect}*{Corollary~3.3.5}. Thus, the action of autoequivalences on $T(X)$ induces the short exact sequence
$$\xymatrix{
    0 \ar[r]
    & \Aut_s(\Db(X))/\bZ[2] \ar[r]
    & \Aut(\Db(X))/\bZ[2] \ar[r]^-\tau
    & \bZ_2 \ar[r]
    & 0.
}$$
The surjection $\tau$ has a section induced by $[1]$, so the middle term splits as a semidirect product
$$
    \Aut(\Db(X))/\bZ[2]
    \;\cong\;
    (\Aut_s(\Db(X))/\bZ[2])\rtimes\bZ_2[1]
$$
where $[1]$ acts on $\Aut_s(\Db(X))/\bZ[2]$ by conjugation. Because $[1]$ lies in the center, this semidirect product is in fact a direct product.
\end{proof}

\begin{proof}[Proof of Theorem~\ref{mainthm:finite-subgp-mod-2}]
Since $G$ is maximal, it contains $\bZ_2[1]$. Hence $G\cong G_s\times\bZ_2[1]$ by Lemma~\ref{lemma:symp-vs-nonsymp}. The rest of the statement follows from Proposition~\ref{prop:finite-subgp} as well as the formulas for $\nu_2$ and $\nu_3$ in \eqref{eqn:nu2} and \eqref{eqn:nu3}.
\end{proof}

\subsection{Existence criteria for associated cubic fourfolds}
\label{subsec:asso_cubic}

We begin by recalling the necessary background on cubic fourfolds and their relation to K3~surfaces.

Let $Y\subseteq\bP^5$ be a smooth cubic hypersurface. For very general $Y$, the lattice
$$
    H^{2,2}(Y,\bZ)\colonequals
    H^{2,2}(Y,\bC)\cap H^4(Y,\bZ)
$$
is spanned by the square $h^2$ of the hyperplane class. Following Hassett \cite{Has00}, we say $Y$ is \emph{special} if there exists a rank~$2$ primitive sublattice $L\subseteq H^4(Y,\bZ)$ such that
$$
    h^2\in L\subseteq H^{2,2}(Y,\bZ).
$$
The lattice $L$ is called a \emph{labelling}. Special cubic fourfolds with a labelling of discriminant $d$ form an irreducible divisor $\cC_d$ in the moduli space of cubic fourfolds. By \cite{Has00}*{Theorem~4.3.1}, this divisor is nonempty if and only if
\begin{equation}
\label{eqn:disc-nonempty}
    d\geq 8
    \quad\text{and}\quad
    d\equiv 0,2 \pmod{6}.
\end{equation}
For very general $Y\in\cC_d$ with labelling $L$, it holds that $L = H^{2,2}(Y,\bZ)$.

A K3~surface $X$ and a cubic fourfold $Y$ are \emph{associated} if there exist a polarization $f$ on $X$ and a labelling $L$ on $Y$ such that there exists a Hodge isometry
\begin{equation}
\label{eqn:k3-cube-tran}
\xymatrix{
    f^{\perp}_{H^2(X,\bZ)(-1)}
    \ar[r]^-\sim &
    L^{\perp}_{H^4(Y,\bZ)}.
}
\end{equation}
Note that, if $X$ has Picard number one, then this gives an isometry between the transcendental parts $T(X)$ and $T(Y)$ in their middle cohomologies up to a twist by $-1$. By \cite{Has00}*{Theorem~5.1.3}, a cubic fourfold of discriminant $d$ has an associated K3~surface if and only if
\begin{equation}
\label{eqn:disc-K3}
    d\text{ is not divisible by }
    4, 9,
    \text{ or any odd prime }
    p\equiv 2\pmod{3}.
\end{equation}
The condition for a K3~surface and a cubic fourfold to be associated is categorical. More precisely, for every cubic fourfold $Y$, there exists a semiorthogonal decomposition
$$
    \Db(Y)\cong\left<
        \cA_Y, \cO_Y, \cO_Y(1), \cO_Y(2)
    \right>
$$
where the subcategory $\cA_Y$ is called the \emph{K3~category} of $Y$. Then $Y$ is Hodge-theoretically associated with a K3~surface $X$ if and only if $\cA_Y\cong\Db(X)$ \cites{AT14,BLMNPS21}.

\begin{prop}
\label{prop:cubic}
Let $X$ be a K3~surface of Picard number one and degree~$2n\geq 4$. Then the following conditions are equivalent:
\begin{enumerate}[label=\textup{(\arabic*)}]
    \item\label{cubic:k3-cat}
    There exists a cubic fourfold $Y\subseteq\bP^5$ associated with $X$.
    \item\label{cubic:autoeq-ord-3}
    There exists $\Phi\in\Aut(\Db(X))$ satisfying $\Phi^3 = [2]$.
    \item\label{cubic:hdg-ord-3}
    $n\geq 7$ and there exists $\phi\in\Aut^+(\mukaiH(X,\bZ))$ of order~$3$.
    \item\label{cubic:fricke-ord-3}
    $n\geq 7$ and there exists a cyclic subgroup in $\Gamma_0^+(n)$ of order~$3$, that is, $\nu_3\neq 0$.
\end{enumerate}
\end{prop}

\begin{proof}
The implication \ref{cubic:k3-cat} $\Rightarrow$ \ref{cubic:autoeq-ord-3} is a consequence of \cite{Kuz04}*{Lemma~4.2}. Suppose that there exists $\Phi\in\Aut(\Db(X))$ such that $\Phi^3 = [2]$.
\begin{itemize}
    \item If $\Phi$ is symplectic, then it defines an element
    $
        \overline{\Phi}\in\Aut_s(\Db(X))/\bZ[2]
    $
    of order~$3$, from which one can check that $n\geq 7$ using \eqref{eqn:sympautoeq_n-not-1}. 

    \item If $\Phi$ is anti-symplectic, then
    $
        \overline{\Phi}[1]\in\Aut_s(\Db(X))/\bZ[2]
    $
    has order~$6$. However, there is no such element by \eqref{eqn:sympautoeq_n-not-1}.
\end{itemize}
This shows that $n\geq 7$. 
Let $\phi\in\Aut^+(\mukaiH(X,\bZ))$ be the Hodge isometry induced by $\Phi$. Since $\Phi^3=[2]$, this isometry has order dividing~$3$. We claim that $\phi\neq\mathrm{id}$. Indeed, if $\phi=\mathrm{id}$, then $\Phi\in\cI(\Db(X))$, and hence $\overline{\Phi}\in\cI(\Db(X))/\bZ[2]$ has order~$3$. This is impossible since $\cI(\Db(X))/\bZ[2]$ is a free group. Hence $\phi$ has order~$3$. This proves \ref{cubic:autoeq-ord-3} $\Rightarrow$ \ref{cubic:hdg-ord-3}. 

Now, assume that $n\geq 7$ and that there exists $\phi\in\Aut^+(\mukaiH(X,\bZ))$ of order~$3$. Since the homomorphism $\Aut^+(\mukaiH(X,\bZ))\to\Gamma_0^+(n)$ has kernel $\{\pm\mathrm{id}\}$, the isometry $\phi$ generates a cyclic subgroup of order~$3$ in $\Gamma_0^+(n)$, whence $\nu_3\neq 0$. This proves \ref{cubic:hdg-ord-3} $\Rightarrow$ \ref{cubic:fricke-ord-3}.

Finally, we show that \ref{cubic:fricke-ord-3} $\Rightarrow$ \ref{cubic:k3-cat}. Suppose that $\nu_3\neq 0$ and define $d\colonequals 2n$. Using the formula in \cite{Shi71}*{Proposition~1.43}, one can verify directly that \eqref{eqn:disc-K3} holds. We claim that $d\equiv 0,2\pmod{6}$, or equivalently, $n\equiv 0,1 \pmod{3}$. Indeed, if $n$ contains a prime factor $p\equiv 2\pmod{3}$, then $\nu_3 = 0$ by \cite{Shi71}*{Proposition~1.43}, contradiction. Hence every prime factor of $n$ is congruent to $0,1 \pmod{3}$, thus $n\equiv 0,1 \pmod{3}$. Together with the condition $n\geq 7$, we conclude that \eqref{eqn:disc-nonempty} holds as well. Conditions~\eqref{eqn:disc-K3} and \eqref{eqn:disc-nonempty} imply that, for all labelled cubic fourfolds $(Y,L)$ such that $\mathrm{disc}(L) = d$ and $L = H^{2,2}(Y,\bZ)$, there exists an isomorphism \eqref{eqn:k3-cube-tran} at the \emph{lattice} level \cite{Has16}*{Proposition~21}. The surjectivity of the period map for cubic fourfolds \cite{Laz10}*{Theorem~1.1} then implies that there exists one such $Y$ with \eqref{eqn:k3-cube-tran} preserving Hodge structure.
\end{proof}

\begin{rmk}
\label{rmk:cubic_deg-2}
Proposition~\ref{prop:cubic} can be extended to K3~surfaces of Picard number one and degree~$2$. Indeed, such K3~surfaces are associated with cubic fourfolds with an ordinary double point in the sense of \cite{Has00}*{\S4.2}. In this case, the order~$3$ elements in \ref{cubic:autoeq-ord-3}, \ref{cubic:hdg-ord-3}, \ref{cubic:fricke-ord-3} can be taken to be those induced by $\Theta^2$, where $\Theta$ is as in Example~\ref{eg:CanonacoKarp}.
\end{rmk}

\section{The realization problems for generic K3~surfaces}
\label{sec:catNielsen}

In this section, we solve the realization problems for K3~surfaces of Picard number one. As a finer result, we prove that maximal nontrivial subgroups of autoequivalences that are finite up to even shifts are isotropy groups of reduced stability conditions. Moreover, the central charges of these stability conditions represent elliptic points for the action of the corresponding Fricke group. We then classify finite subgroups of the autoequivalence group without passing to the quotient by even shifts.

\subsection{Affirmative answers to the realization problems}
\label{subsec:solving-catNielsen}

Following \cite{BB17}, we say that a stability condition $\sigma\in\Stab^\dag(X)$ is \emph{reduced} if its central charge $Z$ satisfies $(Z,Z)=0$ when regarded as an element of $\cN(X)\otimes\bC$. Reduced stability conditions form a submanifold
$$
    \Stab^\dag_\red(X)\subseteq\Stab^\dag(X).
$$
This submanifold is invariant under the $\bC$-action, since this action rescales the central charge and therefore preserves the condition $(Z,Z)=0$. It is also invariant under the action of autoequivalences, since their cohomological actions preserve the Mukai pairing. By \cite{BB17}*{Lemma~2.1}, there is an isomorphism
$$
    \Stab^\dag_\red(X)/\bC
    \;\cong\;
    \Stab^\dag(X)/\widetilde{\GL}^+(2,\bR).
$$
The following result shows that a Gepner-type relation with a non-integral parameter for the $\bC$-action occurs only at a reduced stability condition.

\begin{prop}
\label{prop:Gepnerisreduced}
Let $X$ be a K3~surface and fix $\sigma\in\Stab^\dag(X)$. If there exists an autoequivalence $\Phi\in\Aut(\Db(X))$ such that $\Phi(\sigma)=\sigma\cdot\lambda$ for some $\lambda\in\bC\backslash\bZ$, then $\sigma$ is reduced.
\end{prop}

\begin{proof}
Let $Z_\sigma$ and $Z_{\sigma\cdot\lambda}$ denote the central charges of $\sigma$ and $\sigma\cdot\lambda$, respectively. Then they satisfy $Z_{\sigma\cdot\lambda} = Z_\sigma\cdot e^{-i\pi\lambda}$. Since autoequivalences preserve the pairing on central charges, the relation $\Phi(\sigma) = \sigma\cdot\lambda$ implies
$$
    (Z_\sigma, Z_\sigma)
    = (Z_{\sigma\cdot\lambda}, Z_{\sigma\cdot\lambda})
    = (Z_\sigma, Z_\sigma)\cdot e^{-2i\pi\lambda}.
$$
The assumption $\lambda\notin\bZ$ implies $e^{-2i\pi\lambda}\neq 1$. Hence $(Z_\sigma, Z_\sigma) = 0$, so $\sigma$ is reduced.
\end{proof}

We now prove a stronger form of the realization properties for K3~surfaces of Picard number one, in which the fixed points are represented by reduced stability conditions.

\begin{thm}
\label{thm:nielsen_k3-pic1}
For a K3~surface $X$ of Picard number one, the following statements hold:
\begin{enumerate}[label=\textup{(\arabic*)}]
\item\label{nielsen_k3-pic1:mod2}
Every finite subgroup $G\subseteq\Aut(\Db(X))/\bZ[2]$ fixes a point of $\Stab^\dag_\red(X)/\bC$. Moreover, this fixed point is unique unless
\begin{itemize}
    \item $X$ has degree at least $4$ and $G$ is a subgroup of $\bZ_2[1]$.
    \item $X$ has degree~$2$ and $G$ is a subgroup of $\bZ_2[1]\times\left<\iota\right>$, where $\iota$ is the anti-symplectic involution induced by the covering involution of $X\to\bP^2$.
\end{itemize}
\item\label{nielsen_k3-pic1:no-mod}
Every finite subgroup $G\subseteq\Aut(\Db(X))$ fixes a point of $\Stab^\dag_\red(X)$.
\end{enumerate}
\end{thm}

\begin{proof}
By \cite{BB17}*{Theorem~1.3}, the natural map
$$
    \Stab^\dag_\red(X)/\bC
    \;\cong\;
    \Stab^\dag(X)/\widetilde{\GL}^+(2,\bR)\longrightarrow \cQ^+_0(X)
$$
is a universal covering map. Since $\cQ^+_0(X)\cong\cH\backslash\Delta$ for a discrete subset $\Delta\subseteq\cH$, its universal cover is biholomorphic to $\cH$. Hence,
$$
    \Stab^\dag_\red(X)/\bC\cong\cH.
$$
Under this identification, autoequivalences act by biholomorphic automorphisms of $\cH$, and hence through elements of $\PSL(2,\bR)$, the group of orientation-preserving isometries of the hyperbolic metric.

Every non-identity finite-order element of $\PSL(2,\bR)$ is elliptic. Consequently, every finite subgroup of $\PSL(2,\bR)$ has a fixed point in $\cH$, and this fixed point is unique whenever the subgroup acts nontrivially. See, for example, \cite{Kat92}*{Theorem~2.4.1}. In particular, every finite subgroup
$
    G\subseteq\Aut(\Db(X))/\bZ[2]
$
fixes a point of $\Stab^\dag_\red(X)/\bC$.

We now prove the uniqueness statement. Suppose first that $X$ has degree at least $4$ and that $G$ is not a subgroup of $\bZ_2[1]$. Set
$$
    G'\colonequals\left<G,\, \bZ_2[1]\right>
    \subseteq\Aut(\Db(X))/\bZ[2].
$$
Then $\bZ_2[1]$ is a proper subgroup of $G'$. Together with Lemma~\ref{lemma:symp-vs-nonsymp}, this implies that
$$
    G'\cong G'_s\times\bZ_2[1],
$$
where the symplectic subgroup $G'_s\subseteq G'$ is nontrivial. Proposition~\ref{prop:finite-subgp} asserts that the action of $G'_s$ on $\cQ^+_0(X)$, and hence on $\Stab^\dag_\red(X)/\bC$, is nontrivial. Therefore, $G'_s$ has a unique fixed point. Since $[1]$ acts trivially on $\Stab^\dag_\red(X)/\bC$, the subgroups $G$, $G'$, and $G'_s$ have the same fixed locus. It follows that the fixed point of $G$ is unique.

When $X$ has degree~$2$, the same argument applies by replacing $\bZ_2[1]$ with $\bZ_2[1]\times\left<\iota\right>$ in the definition of $G'$. In this case, the hypothesis ensures that $G'_s\neq\left<\iota[1]\right>$. This proves \ref{nielsen_k3-pic1:mod2}.

Statement~\ref{nielsen_k3-pic1:no-mod} follows from Lemma~\ref{lemma:q2impliesq1}.
\end{proof}

\begin{proof}[Proof of Theorem~\ref{mainthm:nielsen_K3}]
Theorem~\ref{thm:nielsen_k3-pic1} and Proposition~\ref{prop:gepner_auteq-isom} together imply the statement.
\end{proof}

\subsection{Realizing finite subgroups as isotropy groups}
\label{subsec:distGepner}

\begin{cor}
\label{cor:finite-is-fix}
Let $X$ be a K3~surface of Picard number one and $G\subseteq\Aut(\Db(X))$ be a maximal finite subgroup. Then there exists $\sigma\in\Stab^\dag_\red(X)$ such that
$$
    G = \Aut(\Db(X),\sigma).
$$
The same statement holds with $\Aut$ replaced by $\Aut_s$.
\end{cor}

\begin{proof}
By Theorem~\ref{thm:nielsen_k3-pic1}~\ref{nielsen_k3-pic1:no-mod}, there exists $\sigma\in\Stab^\dag_\red(X)$ such that
$$
    G\subseteq\Aut(\Db(X),\sigma).
$$
This inclusion is an equality because $G$ is a maximal finite subgroup and $\Aut(\Db(X),\sigma)$ is finite by \cite{Huy16}*{Remark~1.2}. The proof of the symplectic version is the same.
\end{proof}

\begin{cor}
\label{cor:finite-is-gepner}
Let $X$ be a K3~surface of Picard number one and $G\subseteq\Aut(\Db(X))$ be a subgroup containing $[2]$ such that the image $\overline{G}\subseteq\Aut(\Db(X))/\bZ[2]$ is a maximal finite subgroup. Then there exists $\sigma\in\Stab^\dag_\red(X)$ such that
$$
    G = \Aut(\Db(X),\,\sigma\cdot\bC).
$$
The same statement holds with $\Aut$ replaced by $\Aut_s$.
\end{cor}

\begin{proof}
By Theorem~\ref{thm:nielsen_k3-pic1}~\ref{nielsen_k3-pic1:mod2}, there exists $\sigma\in\Stab^\dag_\red(X)$ such that
\begin{equation}
\label{eqn:inclusion-no-mod-2}
    G\subseteq\Aut(\Db(X),\,\sigma\cdot\bC).
\end{equation}
Passing to the quotient by even shifts, this yields the inclusion
\begin{equation}
\label{eqn:inclusion-mod-2}
    \overline{G} = G/\bZ[2]
    \subseteq
    \Aut(\Db(X),\,\sigma\cdot\bC)/\bZ[2].
\end{equation}
By Proposition~\ref{prop:gepner_auteq-isom}, the group on the right-hand side is finite. Since $\overline{G}$ is maximal, the inclusion \eqref{eqn:inclusion-mod-2} is therefore an equality. For every $\Phi\in\Aut(\Db(X),\,\sigma\cdot\bC)$, this equality implies that $\Phi[m]\in G$ for some even $m$. Since $G$ contains $[2]$, it follows that $\Phi\in G$. Hence \eqref{eqn:inclusion-no-mod-2} is also an equality. The proof of the symplectic version is the same.
\end{proof}

\begin{thm}
\label{thm:dist-stabred}
Let $X$ be a K3~surface of Picard number one and degree $2n$. Then every nontrivial maximal finite subgroup
$$
    G\subseteq\Aut_s(\Db(X))/\bZ[2]
$$ 
fixes a unique point $\overline{\sigma}_G\in\Stab^\dag_\red(X)/\bC$, and the assignment
$
    G\mapsto\overline{\sigma}_G
$
induces a bijection between the following two sets:
\begin{enumerate}[label=\textup{(\alph*)}]
\item\label{item:one-to-one-a}
Nontrivial maximal finite subgroups of $\Aut_s(\Db(X))/\bZ[2]$.
\item\label{item:one-to-one-b}
Points of $\Stab^\dag_\red(X)/\bC$ whose images under the covering map
$$
    \Stab^\dag_\red(X)/\bC \longrightarrow \cQ^+_0(X)
$$
are elliptic points for the action of $\Gamma_0^+(n)$ on $\cQ^+_0(X)$.
\end{enumerate}
The same statement holds with $\Aut_s$ replaced by $\Aut$, provided that $G$ contains $\bZ_2[1]$ as a proper subgroup.
\end{thm}
\begin{proof}
Since $\bZ_2[1]$ is not contained in $\Aut_s(\Db(X))/\bZ[2]$ and $G$ is nontrivial, $G$ is not contained in $\bZ_2[1]$ when regarded as a subgroup of $\Aut(\Db(X))/\bZ[2]$. When $n\geq 2$, Theorem~\ref{thm:nielsen_k3-pic1}~\ref{nielsen_k3-pic1:mod2} then implies that $G$ fixes a unique point of $\Stab^\dag_\red(X)/\bC$.

When $n = 1$, the subgroup
$$
    \bZ_2[1]\times\left<\iota\right>
    \subseteq\Aut(\Db(X))/\bZ[2]
$$
where $\iota$ arises from the covering involution of $X\to\bP^2$, intersects $\Aut_s(\Db(X))/\bZ[2]$ in 
$
   \bZ_2(\iota[1]),
$
which is not a maximal subgroup by Proposition~\ref{prop:finite-subgp}. Since $G$ is maximal, it is not contained in $\left<\iota[1]\right>$, and thus is not contained in $\bZ_2[1]\times\left<\iota\right>$ as a subgroup of $\Aut(\Db(X))/\bZ[2]$. Hence, Theorem~\ref{thm:nielsen_k3-pic1}~\ref{nielsen_k3-pic1:mod2} implies that $G$ fixes a unique point of $\Stab^\dag_\red(X)/\bC$.

The existence and uniqueness of these fixed points define a map from Set~\ref{item:one-to-one-a} to Set~\ref{item:one-to-one-b}, which we now show to be bijective:
\begin{itemize}
\item \textbf{Injectivity:} Let $G$ and $G'$ be maximal finite subgroups of $\Aut_s(\Db(X))/\bZ[2]$ fixing the same point $\sigma\cdot\bC\in\Stab^\dag_\red(X)/\bC$. Then both are contained in the group
$$
    \Aut_s(\Db(X),\,\sigma\cdot\bC)/\bZ[2],
$$
which is finite by Proposition~\ref{prop:gepner_auteq-isom}. Their maximality therefore implies that $G=G'$.

\item \textbf{Surjectivity:} Pick any $\overline{\sigma}\in\Stab^\dag_\red(X)/\bC$ lying over an elliptic point $p\in\cQ^+_0(X)$.

First assume that $n\geq 2$. Via the isomorphism
$$
    \Aut_s(\Db(X))/\bZ[2]
    \;\cong\;
    \pi_1^{\rm orb}(\Gamma_0^+(n)\git\cQ^+_0(X)),
$$
the stabilizer of $p$ in $\Gamma_0^+(n)$ corresponds to a nontrivial maximal finite subgroup
$$
    G'\subseteq\Aut_s(\Db(X))/\bZ[2].
$$
This subgroup fixes a unique point
$
    \overline{\sigma}'\in\Stab^\dag_\red(X)/\bC.
$
Both $\overline{\sigma}$ and $\overline{\sigma}'$ lie over $p$, so there exists $\Psi\in\cI(\Db(X))$ such that $\overline{\sigma} = \Psi(\overline{\sigma}')$. Hence
$$
    G\colonequals\Psi G'\Psi^{-1}
    \subseteq
    \Aut_s(\Db(X))/\bZ[2]
$$
is a nontrivial maximal finite subgroup fixing $\overline{\sigma}$.

When $n = 1$, the above isomorphism is replaced by
$$
    \Aut_s(\Db(X))/\bZ(\iota[1])
    \cong
    \pi_1^{\rm orb}(\Gamma_0^+(1)\git\cQ^+_0(X)).
$$
In this case, the stabilizer of $p$ in $\Gamma_0^+(1)$ corresponds to a finite subgroup of the left-hand side, whose preimage is a nontrivial maximal finite subgroup
$$
    G'\subseteq\Aut_s(\Db(X))/\bZ[2]
$$
under the central extension
$$\xymatrix{
    0 \ar[r]
    & \bZ_2(\iota[1]) \ar[r]
    & \Aut_s(\Db(X))/\bZ[2] \ar[r]
    & \Aut_s(\Db(X))/\bZ(\iota[1]) \ar[r]
    & 0.
}$$
The rest of the argument is identical.
\end{itemize}

For the statement with $\Aut_s$ replaced by $\Aut$, let
$
    G\subseteq\Aut(\Db(X))/\bZ[2]
$
be a maximal finite subgroup containing $\bZ_2[1]$ as a proper subgroup. By Lemma~\ref{lemma:symp-vs-nonsymp}, we have
$$
    G\cong G_s\times\bZ_2[1],
$$
where 
$
    G_s\subseteq\Aut_s(\Db(X))/\bZ[2]
$
is a nontrivial maximal finite subgroup. Since $[1]$ acts trivially on $\Stab^\dag_\red(X)/\bC$, the subgroups $G$ and $G_s$ have the same fixed locus. The statement thus follows from the symplectic case.
\end{proof}

Theorem~\ref{thm:dist-stabred} identifies the points of $\Stab^\dag_\red(X)/\bC$ represented by reduced Gepner-type stability conditions. We next determine which stability conditions in the $\widetilde{\GL}^+(2,\bR)$-orbit of each such $\sigma$ remain of Gepner type.

\begin{lemma}
\label{lemma:only-invol}
Let $X$ be a K3~surface of Picard number one and fix $\sigma\in\Stab^\dag(X)$. Then every $\Phi\in\Aut(\Db(X),\sigma)$ satisfies $\Phi^2= \mathrm{id}$.
\end{lemma}

\begin{proof}
Let $Z\in\cP^+_0(X)$ be the central charge of $\sigma$. Then the group $\Aut(\Db(X),\sigma)$ is isomorphic to $\Aut(\mukaiH(X,\bZ),Z)$ by \cite{Huy16}*{Proposition~1.4}. Therefore, it suffices to prove that every $\phi\in\Aut(\mukaiH(X,\bZ),Z)$ satisfies $\phi^2=\mathrm{id}$. Let $P\subseteq\cN(X)\otimes\bR$ be the positive plane spanned by $\mathrm{Re}(Z)$ and $\mathrm{Im}(Z)$. Then $\phi$ fixes $P$ pointwise, and its actions on $\cN(X)$ and $T(X)$ are as follows:
\begin{itemize}
    \item $\phi$ acts on $\cN(X)$ either trivially or as the reflection across $P$, since it fixes $P$ pointwise and acts by $\pm1$ on the one-dimensional orthogonal complement $P^\perp_{\cN(X)\otimes\bR}$.
    \item $\phi$ acts on $T(X)$ as $\pm\mathrm{id}$ by \cite{Huy16_K3Lect}*{Corollary~3.3.5}.
\end{itemize}
It follows that $\phi^2$ acts trivially on both $\cN(X)$ and $T(X)$, which implies that $\phi^2 = \mathrm{id}$.
\end{proof}

\begin{prop}
\label{prop:dist_orb}
Let $X$ be a K3~surface of Picard number one and $\Phi\in\Aut(\Db(X))$ be an autoequivalence fixing a unique point of $\Stab^\dag_\red(X)/\bC$ represented by $\sigma\in\Stab^\dag_\red(X)$. Let $r$ denote the order of $\Phi$ modulo even shifts.
\begin{itemize}
\item If $r = 2$, then all points of $\sigma\cdot\widetilde{\GL}^+(2,\bR)$ are of Gepner type with respect to $\Phi$. 
\item If $r\geq 3$, then the points of $\sigma\cdot\widetilde{\GL}^+(2,\bR)$ that are of Gepner type with respect to $\Phi$ are precisely those in $\sigma\cdot\bC$.
\end{itemize}
\end{prop}

\begin{proof}
By Corollary~\ref{cor:finite-ord-scalar}, there exists an integer $k$ such that
$$
    \Phi(\sigma) = \sigma\cdot\frac{2k}{r}.
$$
Composing $\Phi$ with a shift functor does not change the stability conditions of Gepner type with respect to $\Phi$, so we may assume that $0\leq 2k < r$. If $r = 2$, then $k = 0$, so the relation above reduces to $\Phi(\sigma) = \sigma$. Hence $\Phi(\sigma\cdot g) = \sigma\cdot g$ for every $g\in\widetilde{\GL}^+(2,\bR)$. This proves the first statement.

Assume that $r\geq3$. If $k=0$, then $\Phi$ fixes $\sigma$. By Lemma~\ref{lemma:only-invol} and the uniqueness hypothesis, $\Phi$ must be an involution, which would imply $r=2$, contradiction. Hence $k>0$. Since $0 < \frac{2k}{r} < 1$, the centralizer of $\frac{2k}{r}$ in $\widetilde{\GL}^+(2,\bR)$ is $\bC$ by Lemma~\ref{lemma:GL2conjugation}. This completes the proof.
\end{proof}

\subsection{Classification of finite subgroups}
\label{subsec:involAuteq}

In what follows, we classify finite subgroups of autoequivalences without passing to the quotient by even shifts. By Corollary~\ref{cor:finite-is-fix}, every maximal such subgroup has the form $\Aut(\Db(X),\sigma)$ for some stability condition $\sigma$. Thus it suffices to classify subgroups of this form. The following lemma reduces the problem further to the classification of involutions.

\begin{lemma}
\label{lemma:only-Z2}
Let $X$ be a K3~surface of Picard number one and fix $\sigma\in\Stab^\dag(X)$. Then the group $\Aut(\Db(X),\sigma)$ is either trivial or isomorphic to $\bZ_2$.
\end{lemma}

\begin{proof}
By Lemma~\ref{lemma:symp-vs-nonsymp}, the group of autoequivalences of Gepner type with respect to $\sigma$ modulo even shifts splits as a direct product:
\begin{equation}
\label{eqn:symp-vs-nonsymp}
    \Aut(\Db(X),\,\sigma\cdot\bC)/\bZ[2]
    \;\cong\;
    \bigl(
        \Aut_s(\Db(X),\,\sigma\cdot\bC)/\bZ[2]
    \bigr)\times\bZ_2[1].
\end{equation}
The first factor is finite by Proposition~\ref{prop:gepner_auteq-isom}, and there is an inclusion
$$
    \Aut_s(\Db(X),\,\sigma\cdot\bC)/\bZ[2]
    \;\subseteq\;
    \Aut_s(\Db(X))/\bZ[2].
$$
Since every finite subgroup of the group on the right-hand side is cyclic by Proposition~\ref{prop:finite-subgp}, the group on the left-hand side is cyclic.

By Proposition~\ref{prop:mod-vs-nomod}, there is a short exact sequence
$$
\xymatrix@R=0pt{
    0 \ar[r]
    & \Aut(\Db(X),\sigma) \ar[r]
    & \Aut(\Db(X),\,\sigma\cdot\bC)/\bZ[2] \ar[r]
    & \mu_m \ar[r]
    & 0. \\
    && [1] \ar@{|->}[r] & -1 &
}
$$
This sequence identifies $\Aut(\Db(X),\sigma)$ with a subgroup of the direct product in \eqref{eqn:symp-vs-nonsymp} not containing $[1]$. By the discussion above, this direct product is isomorphic to $\bZ_r\times\bZ_2[1]$ for some positive integer $r$. By Lemma~\ref{lemma:only-invol}, every non-identity element of $\Aut(\Db(X),\sigma)$ is an involution. Since it does not contain $[1]$, it is either trivial or isomorphic to $\bZ_2$.
\end{proof}

We next relate involutions in $\Aut(\Db(X))$ to those in $\Aut_s(\Db(X))/\bZ[2]$.

\begin{lemma}
\label{lemma:bij_invol}
Let $X$ be a K3~surface of odd Picard number. Then there is a natural bijection
$$
    \{
        \text{involutions in }\Aut(\Db(X))
    \}
    \longrightarrow
    \{
        \text{involutions in }\Aut_s(\Db(X))/\bZ[2]
    \}
$$
which preserves conjugacy classes. Explicitly, the map from left to right is given by
$$
\Phi\longmapsto\begin{cases}
        \overline{\Phi}
        & \text{if }\Phi\text{ is symplectic}, \\
        \overline{\Phi}[1]
        & \text{if }\Phi\text{ is anti-symplectic}.
    \end{cases}
$$
\end{lemma}

\begin{proof}
Since $X$ has odd Picard number, the only Hodge isometries of its transcendental lattice are $\pm\mathrm{id}$ \cite{Huy16_K3Lect}*{Corollary~3.3.5}. Thus, every autoequivalence is either symplectic or anti-symplectic, and composing with $[1]$ interchanges these two types. This defines the map in the statement, which is clearly compatible with conjugation.

Next, we construct its inverse. Let $\overline{\Phi}\in\Aut_s(\Db(X))/\bZ[2]$ be an involution, and choose a representative $\Psi\in\Aut(\Db(X))$. Then $\Psi^2 = [2m]$ for some integer $m$. Set
$$
    \Phi\colonequals\Psi[-m].
$$
Then $\Phi$ is an involution in $\Aut(\Db(X))$. By construction, $\Phi$ represents either $\overline{\Phi}$ or $\overline{\Phi}[1]$, depending on the parity of $m$. Moreover, it is the unique finite-order representative among all representatives of $\overline{\Phi}$ and $\overline{\Phi}[1]$. Indeed, any other representative has the form $\Phi[r]$ with $r\neq0$, which has infinite order because the shift functor does.

This construction thus defines a map from the set of involutions in $\Aut_s(\Db(X))/\bZ[2]$ to the set of involutions in $\Aut(\Db(X))$. It is straightforward to verify that the two maps are inverse to each other.
\end{proof}

\begin{thm}
\label{thm:classification-finite}
Let $X$ be a K3~surface of Picard number one and degree~$2n$. If $G$ is a maximal finite subgroup of $\Aut(\Db(X))$, then there exists $\sigma\in\Stab^\dag(X)$ such that
$$
    G = \Aut(\Db(X),\sigma)
    \cong\begin{cases}
        \bZ_2 & \text{if}\quad n=1,2 \\
        0 & \text{if}\quad n=3,4 \\
        0 \;\text{or }\; \bZ_2 & \text{if}\quad n\geq 5.
    \end{cases}
$$
\begin{itemize}
    \item When $n=1,2$, there is a unique such subgroup up to conjugation.
    \item When $n\geq 5$, there are exactly $\frac{\nu_2}{2}$ such subgroups isomorphic to $\bZ_2$ up to conjugation, where $\nu_2$ is given by \eqref{eqn:nu2}.
\end{itemize}
\end{thm}

\begin{proof}
The identification $G = \Aut(\Db(X),\sigma)$ is given by Corollary~\ref{cor:finite-is-fix}. By Lemma~\ref{lemma:only-Z2}, the group $\Aut(\Db(X),\sigma)$ is either trivial or isomorphic to $\bZ_2$, which turns the problem into the classification of involutions in $\Aut(\Db(X))$. Lemma~\ref{lemma:bij_invol} further turns it into the classification of involutions in $\Aut_s(\Db(X))/\bZ[2]$. The list and the numbers of conjugacy classes then follow from Proposition~\ref{prop:finite-subgp}.
\end{proof}

The involutions representing the unique conjugacy classes for $n=1,2$ are induced, respectively, by the covering involution of $X\to\bP^2$ and $\Theta^2[-1]$ where $\Theta = (-\otimes\cO_X(1))\circ T_{\cO_X}$ (see Example~\ref{eg:CanonacoKarp}). Note that both involutions are anti-symplectic. We will show that, in fact, all involutive autoequivalences are anti-symplectic for all $n$.

Assume that $n\geq 2$ and consider the isometry group
$$
    \widehat{\uO}(\cN(X))\colonequals\{
        \phi\in\uO(\cN(X))
        \mid
        \phi\text{ acts on }\cN(X)^*/\cN(X)\text{ as }\pm\mathrm{id}
    \}.
$$
Here $\cN(X)^*$ denotes the dual lattice of $\cN(X)$ with respect to the Mukai pairing, and $\cN(X)^*/\cN(X)$ is its discriminant group.
The group $\widehat{\uO}(\cN(X))$ admits two natural homomorphisms to $\bZ_2\cong\{\pm 1\}$ given by
$$
    \det\colon\widehat{\uO}(\cN(X))\longrightarrow\{\pm 1\}
    \qquad\text{and}\qquad
    \mathrm{disc}\colon\widehat{\uO}(\cN(X))\longrightarrow\{\pm 1\}
$$
where $\det$ is the determinant function and $\mathrm{disc}$ records the sign of the action of an isometry on the discriminant group. Note that elements in (resp. outside) the kernel of $\mathrm{disc}$ become symplectic (resp. anti-symplectic) upon extension to $\mukaiH(X,\bZ)$. We have
$$
    (\det\cdot\,\mathrm{disc})(\mathrm{id})
    = (\det\cdot\,\mathrm{disc})(-\mathrm{id}) = 1.
$$
Thus the multiplication $\det\cdot\,\mathrm{disc}$ descends to the quotient  $\widehat{\uO}(\cN(X))/\{\pm\mathrm{id}\}$ as
$$
    \overline{\det\cdot\,\mathrm{disc}}
    \colon\widehat{\uO}(\cN(X))/\{\pm\mathrm{id}\}
    \longrightarrow\{\pm 1\}.
$$
The isomorphism $\cH\cong\cQ^+(X)$ in \eqref{map:hyper-domain} identifies $\Gamma_0^+(n)$ as a subgroup of $\widehat{\uO}(\cN(X))/\{\pm\mathrm{id}\}$, so we can consider the restriction
$$
    \mathbf{d}\colonequals
    \overline{\det\cdot\,\mathrm{disc}}\,|_{\Gamma_0^+(n)}
    \colon\Gamma_0^+(n)\longrightarrow\{\pm 1\}.
$$

\begin{lemma}
\label{lemma:det-disc}
For every $n\geq 2$, we have $\Gamma_0(n) = \ker\mathbf{d}$.
\end{lemma}

\begin{proof}
Pick any
$
    g = \begin{pmatrix}
        \alpha & \beta \\
        \gamma & \delta
    \end{pmatrix}
    \in\Gamma_0^+(n).
$
Then the isometries of $\cN(X)$ inducing $g$ are represented by $\pm M_g$ where
$$
    M_g = \begin{pmatrix}
        \delta^2
        & 2\gamma\delta
        & \frac{1}{n}\gamma^2 \\[3pt]
        \beta\delta
        & \alpha\delta+\beta\gamma
        & \frac{1}{n}\alpha\gamma \\[3pt]
        n\beta^2
        & 2n\alpha\beta
        & \alpha^2
    \end{pmatrix}.
$$
See \cite{Dol96}*{Section~7} or \cite{Kaw14}*{Section~2.4}. A direct computation shows that $M_g$ is an integral matrix and
$
    \det(M_g) = (\alpha\delta - \beta\gamma)^3 = 1.
$
To prove the lemma, it suffices to prove that $g\in\Gamma_0(n)$ if and only if $\mathrm{disc}(M_g) = 1$.

The discriminant group $\cN(X)^*/\cN(X)$ is a cyclic group of order $2n$ generated by $(0,\frac{1}{2n},0)$. One can verify that $M_g$ acts on this generator by multiplying
$$
    \alpha\delta + \beta\gamma
    = 2\alpha\delta - 1
    = 2\beta\gamma + 1.
$$
\begin{itemize}
    \item If $g\in\Gamma_0(n)$, then $\gamma\equiv 0\pmod{n}$, thus $2\beta\gamma + 1\equiv 1\pmod{2n}$. In this case, the action of $M_g$ on the discriminant group is the identity, so $\mathrm{disc}(M_g) = 1$.
    \item If $g\notin\Gamma_0(n)$, then we can write
    $$
        \begin{pmatrix}
            \alpha & \beta \\
            \gamma & \delta
        \end{pmatrix}
        = \begin{pmatrix}
            a & b \\
            c & d
        \end{pmatrix}
        \begin{pmatrix}
            0 & -\frac{1}{\sqrt{n}} \\
            \sqrt{n} & 0
        \end{pmatrix}
        \quad\text{where}\quad
        \begin{pmatrix}
            a & b \\
            c & d
        \end{pmatrix}\in\Gamma_0(n).
    $$
    From here we get $\alpha\delta = -bc \equiv 0\pmod{n}$, thus $2\alpha\delta - 1\equiv -1\pmod{2n}$. In this case, the action of $M_g$ on the discriminant group is $-\mathrm{id}$, so $\mathrm{disc}(M_g) = -1$.
\end{itemize}
This shows that $g\in\Gamma_0(n)$ if and only if $\mathrm{disc}(M_g) = 1$, which completes the proof.
\end{proof}

\begin{prop}
\label{prop:invol-anti-symp}
For a K3~surface $X$ of Picard number one, every involutive autoequivalence is anti-symplectic.
\end{prop}

\begin{proof}
Suppose that $X$ has degree $2n$. If $n=1$, then every involutive autoequivalence is conjugate to the one induced by the covering involution $X\to\bP^2$, which is anti-symplectic, so the statement holds in this case.

Assume that $n\geq 2$ and let $\Phi\in\Aut(\Db(X))$ be an involution. By Theorem~\ref{thm:nielsen_k3-pic1}, there exists $\sigma\in\Stab^\dag_\red(X)$ fixed by $\Phi$. Hence the induced isometry $\phi\in\widehat{\uO}(\cN(X))$ fixes pointwise the positive plane determined by $\sigma$. Since $\phi$ is an involution, it acts as $-1$ on the orthogonal complement of this plane, and therefore $\det(\phi) = -1$.

Let $\overline{\phi}\in\Gamma_0^+(n)$ denote the induced action on $\cQ^+(X)\cong\cH$. The point $[v]\in\cQ^+_0(X)$ determined by $\sigma$ is fixed by $\overline{\phi}$. We claim that
$
    \overline\phi\notin\Gamma_0(n)\frn.
$
Indeed, if $\overline\phi\in\Gamma_0(n)\frn$, then Lemma~\ref{lemma:ref-on-hyper} shows that $\overline{\phi}$ is induced by reflection in a $(-2)$-vector $\delta\in\cN(X)$. Such an involution has a unique fixed point in $\cQ^+(X)$, which is a removed $(-2)$-point and hence does not belong to $\cQ^+_0(X)$. This contradicts the fact that $\overline{\phi}$ fixes $[v]\in\cQ^+_0(X)$.

Hence $\overline\phi\in\Gamma_0(n)$. Since $\det(\phi) = -1$, Lemma~\ref{lemma:det-disc} implies that $\mathrm{disc}(\phi) = -1$. Thus $\phi$, and hence $\Phi$, is anti-symplectic.
\end{proof}

\begin{proof}[Proof of Theorem~\ref{mainthm:finite-subgp}]
The statement follows from Proposition~\ref{prop:invol-anti-symp}, Theorem~\ref{thm:classification-finite}, and the formula for $\nu_2$ in \eqref{eqn:nu2}.
\end{proof}

\section{A polychotomy of autoequivalences}
\label{sec:polychotomy}

For K3~surfaces of Picard number one, we classify autoequivalences according to their actions on the period domain of stability conditions and relate this classification to categorical entropy. Some lemmas on derived categories of twisted sheaves needed along the way are deferred to Section~\ref{subsec:lemmas-twisted}.

\subsection{Polychotomy via period-domain actions}

Let $X$ denote a K3~surface of Picard number one and degree~$2n$. As discussed in Section~\ref{subsec:sphericalFricke}, the action of autoequivalences on the upper half-plane $\cQ^+(X)\cong\cH$ induces a surjective homomorphism
\begin{equation}
\label{eqn:Auteq-to-Fricke}
\xymatrix{
    F\colon\Aut(\Db(X)) \ar@{->>}[r]
    & \Gamma_0^+(n).
}
\end{equation}
Since a numerical isometry $\phi$ and its negative $-\phi$ induce the same action on $\cQ^+(X)$, the homomorphism $F$ does not distinguish between them. When $n=1$, the surface $X$ is a double cover of $\bP^2$, and its covering involution $\iota$ acts trivially on $\cN(X)$, hence also trivially on $\cQ^+(X)$. It follows that
$$
    \widetilde{\cI}(\Db(X))
    \colonequals
    \ker F = \begin{cases}
        \left<\cI(\Db(X)),[1]\right>
        & \text{if}\quad n\geq 2, \\
        \left<\cI(\Db(X)), [1], \iota\right>
        & \text{if}\quad n = 1.
    \end{cases}
$$
Here, $\cI(\Db(X))$ is the Torelli group defined in Section~\ref{subsec:prelim_stabK3} and, when $n=1$, we use the same notation $\iota$ for the autoequivalence of $\Db(X)$ induced by the covering involution.

Recall that the action on $\cN(X)$ induced by
$
\begin{pmatrix}
    \alpha & \beta \\
    \gamma & \delta
\end{pmatrix}
\in\Gamma_0^+(n)
$
is represented by the matrices $\pm M$ where
$$
    M = \begin{pmatrix}
        \delta^2
        & 2\gamma\delta
        & \frac{1}{n}\gamma^2 \\[3pt]
        \beta\delta
        & \alpha\delta+\beta\gamma
        & \frac{1}{n}\alpha\gamma \\[3pt]
        n\beta^2
        & 2n\alpha\beta
        & \alpha^2
    \end{pmatrix}.
$$
See \cite{Dol96}*{Section~7} or \cite{Kaw14}*{Section~2.4}.

\begin{lemma}
\label{lemma:Dolgachev}
Suppose that $M\neq\mathrm{id}$. Then the eigenvalues of $M$ are
$$
    \lambda=1, \qquad
    \mu^\pm = \frac{
        -2 + (\alpha + \delta)^2
        \pm i(\alpha + \delta)\sqrt{
            4 - (\alpha + \delta)^2
        }
    }{2}.
$$
The $1$-eigenspace of $M$ is one-dimensional and generated by
$$
    w = \begin{pmatrix}
        2\gamma \\
        \alpha-\delta \\
        -2n\beta 
    \end{pmatrix}.
$$
In the case $|\alpha + \delta| = 2$, the only eigenvalue of $M$ is $1$, and $M$ has a single Jordan block of size three. Moreover, viewing $w$ as a vector in $\cN(X)\otimes\bR$, its self-intersection is
$$
    w^2 = 2n\left(
        (\alpha + \delta)^2 - 4
    \right).
$$
\end{lemma}

\begin{proof}
The eigenvalues of $M$ and the identity $Mw= w$ can be computed directly. Assume that $|\alpha+\delta| = 2$. Then a direct computation gives $\mu^\pm=1$.
\begin{itemize}
\item Assume that $\beta\neq 0$. If $\alpha + \delta = 2$, then
\begin{equation}
\label{eqn:JordanDecomp}
M = \begin{pmatrix}
    \frac{(1-\alpha)^2}{\beta^2n} & \frac{1-\alpha}{\beta^2n} & \frac{\alpha}{2\beta^2n} \\[3pt]
    \frac{1-\alpha}{\beta n} & \frac{1}{2\beta n} & -\frac{1}{4\beta n} \\[3pt]
    1 & 0 & 0
\end{pmatrix}
\begin{pmatrix}
    1&1&0 \\[3pt]
    0&1&1 \\[3pt]
    0&0&1
\end{pmatrix}
\begin{pmatrix}
    \frac{(1-\alpha)^2}{\beta^2n} & \frac{1-\alpha}{\beta^2n} & \frac{\alpha}{2\beta^2n} \\[3pt]
    \frac{1-\alpha}{\beta n} & \frac{1}{2\beta n} & -\frac{1}{4\beta n} \\[3pt]
    1 & 0 & 0
\end{pmatrix}^{-1}.
\end{equation}
If $\alpha + \delta = -2$, then
$$
M = \begin{pmatrix}
    \frac{(1+\alpha)^2}{\beta^2n} & \frac{1+\alpha}{\beta^2n} & -\frac{\alpha}{2\beta^2n} \\[3pt]
    -\frac{1+\alpha}{\beta n} & -\frac{1}{2\beta n} & \frac{1}{4\beta n} \\[3pt]
    1 & 0 & 0
\end{pmatrix}
\begin{pmatrix}
    1&1&0 \\[3pt]
    0&1&1 \\[3pt]
    0&0&1
\end{pmatrix}
\begin{pmatrix}
    \frac{(1+\alpha)^2}{\beta^2n} & \frac{1+\alpha}{\beta^2n} & -\frac{\alpha}{2\beta^2n} \\[3pt]
    -\frac{1+\alpha}{\beta n} & -\frac{1}{2\beta n} & \frac{1}{4\beta n} \\[3pt]
    1 & 0 & 0
\end{pmatrix}^{-1}.
$$

\item If $\beta = 0$, then $\alpha = \delta = \pm1$. In this case,  $\gamma\neq 0$ since $M\neq\mathrm{id}$, and we have
$$
M = \begin{pmatrix}
    1 & \pm2\gamma & \frac{1}{n}\gamma^2 \\[3pt]
    0 & 1 & \pm\frac{1}{n}\gamma \\[3pt]
    0 & 0 & 1
\end{pmatrix}
= \begin{pmatrix}
    1 & 0 & 0 \\[3pt]
    0 & \pm\frac{1}{2\gamma} & \mp\frac{1}{4\gamma} \\[3pt]
    0 & 0 & \frac{n}{2\gamma^2}
\end{pmatrix}
\begin{pmatrix}
    1&1&0 \\[3pt]
    0&1&1 \\[3pt]
    0&0&1
\end{pmatrix}
\begin{pmatrix}
    1 & 0 & 0 \\[3pt]
    0 & \pm\frac{1}{2\gamma} & \mp\frac{1}{4\gamma} \\[3pt]
    0 & 0 & \frac{n}{2\gamma^2}
\end{pmatrix}^{-1}.
$$
\end{itemize}
Hence $M$ has a single Jordan block of size three. Finally, one can verify that $w^2 = 2n\left((\alpha+\delta)^2-4\right)$ using the relation $\alpha\delta-\beta\gamma=1$.
\end{proof}

\begin{propdef}
\label{prop:polychotomy}
Let $X$ be a K3~surface of Picard number one, $\Phi$ be an autoequivalence not contained in $\widetilde{\cI}(\Db(X))$, and $[\Phi]$ be the isometry of $\cN(X)$ induced by $\Phi$. Then exactly one of the following holds:
\begin{enumerate}[label=\textup{(\alph*)}]
\item\label{item:a} There exists $\Psi\in\widetilde{\cI}(\Db(X))$ such that $(\Phi\Psi)^m=[k]$ for some $m>0$ and $k\in\bZ$.
\item\label{item:b} There exists $\Psi\in\widetilde{\cI}(\Db(X))$ and a spherical object $S\in\Db(X)$ such that $\Phi\Psi=T_S$.
\item\label{item:c} There exists a nonzero vector $w\in\cN(X)$ such that $w^2=0$ and $[\Phi]w =\pm w$.
\item\label{item:d} The isometry $[\Phi]$ has spectral radius $\rho([\Phi])>1$.
\end{enumerate}
We say that $\Phi$ is of finite-order type, $(-2)$-reducible type, $0$-reducible type, or pseudo-Anosov type if it satisfies \ref{item:a}, \ref{item:b}, \ref{item:c}, or \ref{item:d}, respectively.
\end{propdef}

\begin{proof}
Recall that a nontrivial $g\in\PSL(2,\bR)$ is called \emph{elliptic}, \emph{parabolic}, or \emph{hyperbolic} according to whether $|\tr(g)|<2$, $|\tr(g)|=2$, or $|\tr(g)|>2$, respectively. We classify $\Phi$ according to the type of
$$
    F(\Phi) = \begin{pmatrix}
        \alpha & \beta \\
        \gamma & \delta
    \end{pmatrix}\in\Gamma_0^+(n).
$$
Assume that $F(\Phi)$ is elliptic. Then its action on $\cQ^+(X)$ has a fixed point.

\begin{itemize}
\item The fixed point belongs to $\cQ^+_0(X)$ if and only if there exists $\Psi\in\widetilde{\cI}(\Db(X))$ such that $\Phi\Psi$ fixes a point of $\Stab^\dag_\red(X)/\bC$ by \cite{Bri08}*{Theorem~1.1}. This is equivalent to Condition~\ref{item:a} by Theorems~\ref{mainthm:nielsen_K3} and \ref{thm:nielsen_k3-pic1}.

\item The fixed point belongs to $\cQ^+(X)\backslash\cQ^+_0(X)$ if and only if $F(\Phi)$ is given by the reflection along a $(-2)$-vector, which is equivalent to Condition~\ref{item:b} by Proposition~\ref{prop:sph-in-Fricke}.
\end{itemize}

\noindent It remains to consider the cases where $F(\Phi)$ is parabolic or hyperbolic. By Lemma~\ref{lemma:Dolgachev}, we have:

\begin{itemize}
\item The element $F(\Phi)$ is parabolic, that is, $|\alpha+\delta|=2$, if and only if there is a nonzero vector $w\in\cN(X)\otimes\bR$ such that $[\Phi]w =\pm w$ and $w^2=0$. Moreover, the vector
$$
    w = \begin{pmatrix}
        2\gamma \\ \alpha-\delta \\ -2n\beta 
    \end{pmatrix}
$$
is integral if $F(\Phi)\in\Gamma_0(n)$ and, if $F(\Phi)\in\Gamma_0^+(n)\setminus\Gamma_0(n)$, becomes integral after multiplication by $\sqrt{n}$. Thus, this is equivalent to Condition~\ref{item:c}.

\item The element $F(\Phi)$ is hyperbolic, that is, $|\alpha+\delta|>2$, if and only if Condition~\ref{item:d} holds.
\end{itemize}
This completes the proof.
\end{proof}

\begin{rmk}
\label{rmk:nielsen-thurston}
Proposition~\ref{prop:polychotomy} resembles the \emph{Nielsen--Thurston classification} of mapping classes on Riemann surfaces, where mapping classes are classified into three types: finite order, reducible, or pseudo-Anosov. A similar polychotomy for elliptic curves was established in \cite{KKO22}*{Theorem~5.20}. In our setting, reducible autoequivalences are further divided into \emph{$(-2)$-reducible} and \emph{$0$-reducible} types, which correspond respectively to the \emph{$(-2)$-points} and the \emph{cusps} of the modular curve $\Gamma_0^+(n)\git\cQ^+(X)$ discussed in Section~\ref{subsec:orbfun_periodsp}.
\end{rmk}

Unlike $(-2)$-reducible autoequivalences, which are spherical twists composed with elements of $\widetilde{\cI}(\Db(X))$, $0$-reducible autoequivalences have received less attention. To study them, for each nonzero vector $w\in\cN(X)$ with $w^2=0$, we define
$$
    \Aut(\Db(X),\pm w)\colonequals
    \left\{
        \Phi\in\Aut(\Db(X))
    \;\middle|\;
        [\Phi]w = \pm w
    \right\}.
$$

\begin{prop}
\label{prop:autoeq_fixing-w}
Consider a K3~surface $X$ of Picard number one and degree~$2n$, a nonzero vector $w\in\cN(X)$ with $w^2=0$, and the map $F$ defined in \eqref{eqn:Auteq-to-Fricke}. Denote by
$$
    F_w\colon\Aut(\Db(X),\pm w)\longrightarrow\Gamma_0^+(n)
$$
the restriction of $F$ to $\Aut(\Db(X),\pm w)$. Then $F_w$ fits into a short exact sequence
$$\xymatrix{
    0 \ar[r]
    & \widetilde{\cI}(\Db(X)) \ar[r]
    & \Aut(\Db(X),\pm w) \ar[r]^-{F_w}
    & \bZ \ar[r]
    & 0.
}$$
\end{prop}

\begin{proof}
Every element in $\ker F_w$ belongs to $\widetilde{\cI}(\Db(X))$. Thus
$$
    \ker F_w\subseteq\widetilde{\cI}(\Db(X)).
$$
Conversely, every element of $\widetilde{\cI}(\Db(X))$ sends $w$ to either $w$ or $-w$, and so belongs to the subgroup $\Aut(\Db(X),\pm w)$. Therefore, such an element lies in the kernel of $F_w$, and thus
$$
    \ker F_w\supseteq\widetilde{\cI}(\Db(X)).
$$
It follows that
$
    \ker F_w = \widetilde{\cI}(\Db(X)).
$

We claim that $\operatorname{im}F_w\subseteq\Gamma_0^+(n)$ is the stabilizer of a cusp. Write elements of $\operatorname{im}F_w$ as
$$\begin{pmatrix}
    \alpha & \beta \\
    \gamma & \delta
\end{pmatrix}.$$
Since these elements are parabolic, we have $|\alpha+\delta|=2$. As we work in $\PSL(2,\bR)$, we may assume without loss of generality that $\alpha+\delta=2$.
\begin{itemize}
\item Assume that $\beta\neq 0$. Writing $w = (r, d, s)$, the decomposition~\eqref{eqn:JordanDecomp} implies that
$$
    \begin{pmatrix}
        \frac{(1-\alpha)^2}{\beta^2n} \\[3pt]
        \frac{1-\alpha}{\beta n} \\[3pt]
        1
    \end{pmatrix}
    \text{ and }
    \begin{pmatrix}
        r \\ d \\ s
    \end{pmatrix}
    \text{ are parallel.}
    \quad\Longrightarrow\quad
    \frac{1-\alpha}{\beta} = \frac{nd}{s}.
$$
Hence,
$$
    \operatorname{im}F_w = \left\{
        \begin{pmatrix}
            \alpha & \beta \\
            \gamma & 2-\alpha
        \end{pmatrix}\in\Gamma_0^+(n)
    \ \middle\vert\
        \frac{1-\alpha}{\beta} = \frac{nd}{s}
    \right\},
$$
which is exactly the stabilizer of the cusp $\frac{s}{nd}$.

\item If $\beta = 0$, then $\alpha = \delta = 1$. In this case,
$$
    \operatorname{im}F_w = \left\{
        \begin{pmatrix}
            1 & 0 \\
            \gamma & 1
        \end{pmatrix}\in\Gamma_0^+(n)
    \right\}
$$
which is the stabilizer of the cusp $0$.
\end{itemize}
Therefore, $\operatorname{im}F_w\subseteq\Gamma_0^+(n)$ is the stabilizer of a cusp, so it is isomorphic to $\bZ$.
\end{proof}

\subsection{Connection with categorical entropy}

Another way to view the polychotomy is through the dynamical behavior of autoequivalences, namely, the behavior of the iterates $\Phi^n$ as $n$ tends to infinity. Various invariants can be associated with autoequivalences from a dynamical perspective. Among them, the \emph{categorical entropy}~$h_\cat$ and the \emph{categorical polynomial entropy}~$h_\poly$, introduced in \cites{DHKK14,FFO21}, are most closely related to the polychotomy. For instance, \cite{FFO21}*{Theorem~1.4} proves that, if $\Phi$ is an autoequivalence of the bounded derived category of coherent sheaves on an elliptic curve, then it falls into one of the following types:
\begin{itemize}
    \item[] (finite order) $|\tr([\Phi])|<2$ if and only if $h_\cat(\Phi)=h_\poly(\Phi)=0$.
    \item[] (reducible) $|\tr([\Phi])|=2$ if and only if $h_\cat(\Phi)=0$ and $h_\poly(\Phi)=1$.
    \item[] (pseudo-Anosov) $|\tr([\Phi])|>2$ if and only if $h_\cat(\Phi)>0$.
\end{itemize}

\begin{prop}
\label{prop:0-reducible-entropy}
Let $X$ be a K3~surface of Picard number one, and let $w\in\cN(X)$ be a primitive nonzero vector that satisfies $w^2=0$. Then, for every $\Phi\in\Aut(\Db(X),\pm w)$, there exists $\Psi\in\widetilde{\cI}(\Db(X))$ with the following properties:
\begin{enumerate}[label=\textup{(\alph*)}]
    \item\label{item:reducible-a} $h_\cat(\Phi\Psi)=0$.
    \item\label{item:reducible-b} There exists a semirigid object $E\in\Db(X)$ such that $v(E)=w$ and $(\Phi\Psi)(E)=E$.
\end{enumerate}
Recall that an object $E\in\Db(X)$ is called semirigid if $\dim\Ext^1(E,E)=2$.
\end{prop}
\begin{proof}
Replacing $(\Phi,\Psi)$ by $(\Phi[1],\Psi[-1])$ if necessary, we may assume that $[\Phi]w=w$. Choose a stability condition $\sigma\in\Stab^\dag(X)$ that is generic with respect to $w$. Since $w$ is primitive and $w^2 = 0$, \cite{BM14_JAMS}*{Lemma~7.2} asserts that $Y\colonequals M_\sigma(w)$ is a K3~surface. Moreover, there exists a Brauer class $\alpha\in\mathrm{Br}(Y)$ and an equivalence
$$
    \Xi\colon \Db(X)\longrightarrow\Db(Y,\alpha),
$$
which takes the class $w$ to the point class $(0,0,1)$. Consider
$$
    \Xi\circ\Phi\circ\Xi^{-1}
    \in\Aut(\Db(Y,\alpha)).
$$
Since $\Phi$ preserves the class $w$, the cohomological action
$$
    (\Xi\circ\Phi\circ\Xi^{-1})^H
    \in\uO(\mukaiH(Y,B,\bZ)),
$$
where $B$ is a B-field lift of $\alpha$, fixes $(0,0,1)$. Moreover, it is orientation-preserving by \cite{Rei19}*{Theorem~A}. Thus, by Lemma~\ref{lemma:twisted-fix-point-class}, there exist $L\in\Pic(Y)$ and $f\in\Aut(Y)$ with $f^*\alpha = \alpha$, together with a geometric lift $f^*_\alpha\in\Aut(\Db(Y,\alpha))$, such that
$$
    (\Xi\circ\Phi\circ\Xi^{-1})^H
    = ((-\otimes L)\circ f^*_\alpha)^H.
$$
This implies that
$$
    \Phi^H
    = (\Xi^{-1}\circ(-\otimes L)\circ f^*_\alpha\circ\Xi)^H
$$
as actions on $\mukaiH(X,\bZ)$. Hence, there exists $\Psi\in\cI(\Db(X))$ such that
$$
\Phi\Psi=\Xi^{-1}\circ((-\otimes L)\circ f^*_\alpha)\circ\Xi.
$$

We first prove \ref{item:reducible-a}. Note that $Y$ has Picard number one. Indeed, there are isometries
$$
    \cN(X)_\bQ
    \cong\cN(Y,\alpha)_\bQ
    \cong\cN(Y)_\bQ
$$
where the first is induced by the equivalence $\Db(X)\cong\Db(Y,\alpha)$, whereas the second is given by $e^{-B}$. Since $\cN(X)$ has rank three, $\cN(Y)$ also has rank three, which implies that $Y$ has Picard number one. Consequently, $f$ is either the identity or an involution \cite{Huy16_K3Lect}*{Corollary~15.2.12}. Since $f_\alpha^*$ is a geometric lift of $f^*$, we have $(f_\alpha^*)^2\cong(-\otimes M)$ for some line bundle $M$. Then
\begin{align*}
    ((-\otimes L)\circ f^*_\alpha)^2
    &= (-\otimes L)\circ(
        f^*_\alpha\circ(-\otimes L)\circ f^*_\alpha
    ) \\
    &= (-\otimes L)\circ\left(
        f^*_\alpha\circ(-\otimes L)\circ(f^*_\alpha)^{-1}
    \right)\circ(f^*_\alpha)^{2} \\
    &= -\otimes (L\otimes f^*L\otimes M).
\end{align*}
Lemma~\ref{lemma:entropy-twisted-line-bundle} then gives
$$
    h_\cat(\Phi\Psi)
    = h_\cat((-\otimes L)\circ f^*_\alpha)
    =\frac{1}{2}h_\cat(-\otimes (L\otimes f^*L\otimes M))
    = 0.
$$

We now prove \ref{item:reducible-b}. Again by \cite{Huy16_K3Lect}*{Corollary~15.2.12}, either $f=\mathrm{id}_Y$, or $Y$ has degree two and $f$ is the covering involution of the double cover $Y\rightarrow\bP^2$. In either case, the fixed locus of $f$ is nonempty. Choose a point $x\in Y$ fixed by $f$, and let $\kappa_\alpha(x)$ be the twisted skyscraper sheaf at $x$. Then
$$
    f^*_\alpha(\kappa_\alpha(x))
    \cong\kappa_\alpha(x),
    \qquad
    \kappa_\alpha(x)\otimes L
    \cong\kappa_\alpha(x).
$$
Set $E\colonequals \Xi^{-1}(\kappa_\alpha(x))$. Then $(\Phi\Psi)(E)=E$, $v(E)=w$, and $E$ is semirigid since $\kappa_\alpha(x)$ is a semirigid object of $\Db(Y,\alpha)$.
\end{proof}

\begin{eg}
\label{eg:auteq_fixing_pm-w}
According to Proposition~\ref{prop:autoeq_fixing-w}, any autoequivalence fixing $\pm w = (0,0,\pm1)$ has the form $\Psi\circ(-\otimes\cO(m))$ for some $\Psi\in\widetilde{\cI}(\Db(X))$ and $m\in\bZ$. By \cite{DHKK14}*{Lemma~2.13}, we have $h_\cat(-\otimes\cO(m))=0$. Note that $-\otimes\cO(1)$ fixes skyscraper sheaves, which are semirigid objects representing $w$.
\end{eg}

We propose the following question concerning the categorical (polynomial) entropy of $(-2)$-\emph{reducible} autoequivalences.

\begin{ques}
\label{ques:entropy}
Let $X$ be a K3~surface of Picard number one and $S\in\Db(X)$ be a spherical object. Is it true that $h_\poly(T_S\Psi)>0$ for every $\Psi\in\widetilde{\cI}(\Db(X))$ satisfying  $h_\cat(T_S\Psi)=0$?
\end{ques}

For $\Psi=\mathrm{id}$, we have $h_\cat(T_S)=0$ by \cite{OucEntropy}*{Theorem~1.4} and $h_\poly(T_S)=1$ by \cite{FanPtwist}*{Theorem~1.2}. Thus, Question~\ref{ques:entropy} has an affirmative answer in this case. An affirmative answer in general would yield the following trichotomy of autoequivalences in terms of entropy.

\begin{prop}
\label{conj:NTclassificationDyn}
Assume that Question~\ref{ques:entropy} has an affirmative answer. Then, for every autoequivalence $\Phi$ not lying in $\widetilde{\cI}(\Db(X))$, the following statements hold:
\begin{enumerate}[label=\textup{(\alph*)}]
    \item\label{item:Entropy(a)} $\Phi$ is of finite-order type if and only if there exists $\Psi\in\widetilde{\cI}(\Db(X))$ such that $h_\cat(\Phi\Psi) = 0$ and $h_\poly(\Phi\Psi) = 0$.
    \item $\Phi$ is reducible, either $(-2)$-reducible or $0$-reducible, if and only if
    \begin{itemize}
        \item there exists $\Psi\in\widetilde{\cI}(\Db(X))$ such that $h_\cat(\Phi\Psi)=0$, and
        \item for every $\Psi'\in\widetilde{\cI}(\Db(X))$ satisfying $h_\cat(\Phi\Psi')=0$, we have $h_\poly(\Phi\Psi')>0$.
    \end{itemize}
    \item\label{item:Entropy(c)} $\Phi$ is pseudo-Anosov if and only if $h_\cat(\Phi\Psi)>0$ for every $\Psi\in\widetilde{\cI}(\Db(X))$.
\end{enumerate}
Moreover, categorical polynomial entropy can be used to distinguish between $(-2)$-reducible and $0$-reducible autoequivalences as follows:
\begin{enumerate}[label=\textup{(\roman*)}]
    \item If $\Phi$ is $(-2)$-reducible, then there exists $\Psi\in\widetilde{\cI}(\Db(X))$ such that $h_\cat(\Phi\Psi)=0$ and $0<h_\poly(\Phi\Psi)\leq1$.
    \item If $\Phi$ is $0$-reducible, then every $\Psi\in\widetilde{\cI}(\Db(X))$ satisfying $h_\cat(\Phi\Psi)=0$ also satisfies $h_\poly(\Phi\Psi)\geq2$.
\end{enumerate}
\end{prop}

\begin{proof}
By Proposition~\ref{prop:polychotomy}, the three types on the left-hand sides of \ref{item:Entropy(a)}--\ref{item:Entropy(c)} are mutually exclusive and exhaustive. The conditions on the right-hand sides are also mutually exclusive. Thus, it suffices to prove the ``only if'' directions.
\begin{itemize}
    \item The ``only if'' direction of \ref{item:Entropy(a)} follows from the definitions of categorical entropy and categorical polynomial entropy \cites{DHKK14,FFO21}.
    \item The ``only if'' direction of \ref{item:Entropy(c)} follows from the Yomdin-type lower bound for categorical entropy $h_\cat(\Phi)\geq\log\rho([\Phi])$ \cite{KST20}*{Theorem~2.13}, \cite{FFO21}*{Proposition~4.3}.
    \item Suppose that $\Phi$ is $(-2)$-reducible, that is, there exists $\Psi\in\widetilde{\cI}(\Db(X))$ such that $\Phi\Psi=T_S$ for some spherical object $S$. Then $h_\cat(\Phi\Psi) = h_\cat(T_S)=0$ by \cite{OucEntropy}*{Theorem~1.4}. Moreover, assuming that Question~\ref{ques:entropy} has an affirmative answer, we have $h_\poly(\Phi\Psi')>0$ for every $\Psi'\in\widetilde\cI(\Db(X))$ satisfying  $h_\cat(\Phi\Psi')=0$.
    \item Suppose that $\Phi$ is $0$-reducible. Then, by Proposition~\ref{prop:0-reducible-entropy}~\ref{item:reducible-a}, there exists $\Psi\in\widetilde{\cI}(\Db(X))$ such that $h_\cat(\Phi\Psi)=0$. Moreover, any $\Psi\in\widetilde\cI(\Db(X))$ satisfying $h_\cat(\Phi\Psi)=0$ also satisfies $h_\poly(\Phi\Psi)\geq2$ by \cite{FFO21}*{Proposition~4.5} and Lemma~\ref{lemma:Dolgachev}.
\end{itemize}
The last part of the statement follows from the fact that $h_\poly(T_S)\leq1$ for any spherical twist $T_S$ \cite{FFO21}*{Proposition~6.13}.
\end{proof}

\subsection{Lemmas on twisted derived categories}
\label{subsec:lemmas-twisted}

Here we record two basic lemmas on twisted derived categories that are used in the proof of Proposition~\ref{prop:0-reducible-entropy}. The first lemma concerns twisted K3~surfaces. We follow the terminology and notation of \cite{HS05}. In the following, $X$ is a K3~surface, $\alpha\in\mathrm{Br}(X)$ is a Brauer class, and $B\in H^2(X,\bQ)$ is a B-field lift of $\alpha$. We denote by $\mukaiH(X,B,\bZ)$ the Mukai lattice endowed with its weight-two Hodge structure. Its $(2,0)$-part is spanned by $e^B\sigma_X$, where $\sigma_X$ is a nonzero holomorphic two-form on $X$.

\begin{lemma}
\label{lemma:twisted-fix-point-class}
Let $(X,\alpha)$ be a twisted K3~surface with $X$ of Picard number one, and choose a $B$-field lift $B\in H^2(X,\bQ)$ of $\alpha$. Suppose that
$
    \theta\in\uO^+(\mukaiH(X,B,\bZ))
$
is an orientation-preserving Hodge isometry fixing $p\colonequals(0,0,1)\in\mukaiH(X,B,\bZ)$. Then there exist $L\in\Pic(X)$ and $f\in\Aut(X)$ with $f^*\alpha=\alpha$, together with a geometric lift $f^*_\alpha\in\Aut(\Db(X,\alpha))$, such that
$$
    \theta=((-\otimes L)\circ f^*_\alpha)^H.
$$
\end{lemma}
\begin{proof}
Since $\theta$ fixes $p$, it preserves $p^\perp$ and hence induces an isometry
$$\xymatrix{
    p^\perp/\bZ p \ar[r]^-\sim & H^2(X,\bZ)
    : [(0,\ell,s)]\ar@{|->}[r] & \ell.
}$$
This map preserves Hodge structures since it sends $e^B\sigma_X = (0,\,\sigma_X,\,B\cdot\sigma_X)$ to $\sigma_X$. Via this identification, $\theta$ induces a Hodge isometry
$$
    \theta_2\colon
    H^2(X,\bZ)\longrightarrow H^2(X,\bZ).
$$
Using the assumption $\theta(p) = p$, a straightforward computation shows that
$$
    \theta(1,0,0) = e^D = \left(1, D, \frac{D^2}{2}\right)
$$
for some $D\in H^2(X,\bZ)$. For each $x\in H^2(X, \bZ)$, write
$
    \theta(0, x, 0) = (0, \theta_2(x), a(x)).
$
Pairing this with $\theta(1,0,0)$ gives
$
    a(x)=\theta_2(x)\cdot D.
$
Consequently,
\begin{equation}
\label{eqn:theta-lattice-form}
    \theta(r,x,s) = \left(
        r,\;
        \theta_2(x) + rD,\;
        s + (\theta_2(x)\cdot D)+\frac{rD^2}{2}
    \right).
\end{equation}
Equivalently,
$$
    \theta = e^D\circ\widetilde{\theta}_2,
$$
where $\widetilde{\theta}_2$ acts as $\theta_2$ on $H^2(X,\bZ)$ and as the identity on $H^0(X,\bZ)\oplus H^4(X,\bZ)$.

Next, we prove that $\theta_2$ preserves the ample cone. The isometry $e^D$ is orientation-preserving, since it can be continuously deformed to the identity via the one-parameter family $e^{tD}$. Hence $\widetilde{\theta}_2 = e^{-D}\circ\theta$ is orientation-preserving. Since $X$ has Picard number one, $\NS(X)$ is generated by an ample class $h$. Consider the positive-definite plane spanned by the orthogonal basis elements
$$
    (1, 0, -h^2/2),\qquad
    (0, h, 0).
$$
The isometry $\widetilde{\theta}_2$ fixes the first vector and sends the second to $(0, \theta_2(h), 0)$. Since $\NS(X)=\bZ h$, we have $\theta_2(h)=\pm h$. The orientation-preserving property of $\widetilde{\theta}_2$ forces $\theta_2(h)=h$, so $\theta_2$ preserves the ample cone. Consequently, there exists an automorphism $f\in\Aut(X)$ such that $\theta_2 = f^*$ by the global Torelli theorem \cite{Huy16_K3Lect}*{Theorem~7.5.3}.

We now show that the automorphism $f$ preserves the Brauer class $\alpha$. Choose $\zeta\in\bC^*$ that satisfies
$
    \theta_2(\sigma_X)
    = \zeta\sigma_X.
$
Using \eqref{eqn:theta-lattice-form}, a direct computation gives
$$
    \theta(e^B\sigma_X)
    = \left(
        0,\;
        \zeta\sigma_X,\;
        B\cdot\sigma_X + \zeta(D\cdot\sigma_X)
    \right).
$$
Since $\theta$ preserves the Hodge structure on $\mukaiH(X,B,\bZ)$, there exists $\zeta'\in\bC^*$ such that
$$
    \theta(e^B\sigma_X)
    = \zeta'e^B\sigma_X
    = \left(
        0,\;
        \zeta'\sigma_X,\;
        \zeta'(B\cdot\sigma_X)
    \right).
$$
Identifying the degree-two components gives $\zeta = \zeta'$. It follows from comparing the degree-four components that
$$
    \zeta^{-1}(B\cdot\sigma_X)
    - B\cdot\sigma_X
    + D\cdot\sigma_X
    = 0.
$$
Moreover,
$$
    \zeta^{-1}(B\cdot\sigma_X)
    = B\cdot\theta_2^{-1}(\sigma_X)
    = \theta_2(B)\cdot\sigma_X.
$$
Thus, the above equation becomes
$$
    (\theta_2(B) - B + D)\cdot\sigma_X = 0.
$$
Hence $\theta_2(B) - B + D\in\NS(X)_\bQ$, and therefore $f^*B - B\in H^2(X,\bZ)+\NS(X)_\bQ$. By the standard description \cite{Huy16_K3Lect}*{Section~16.4.1}
$$
    \mathrm{Br}(X)\cong
    \frac{H^2(X,\bQ)}{H^2(X,\bZ) + \NS(X)_\bQ},
$$
we conclude that $f^*\alpha = \alpha$. Therefore, $f$ lifts to an autoequivalence $f^*_\alpha\in\Aut(\Db(X,\alpha))$.

The cohomological action of $f^*_\alpha$ fixes $p=(0,0,1)$ and acts as $\theta_2$ on $p^{\perp}/\bZ p$. Thus,
$$
    u\colonequals
    \theta\circ ((f^*_\alpha)^H)^{-1}
$$
is a Hodge isometry of $\mukaiH(X,B,\bZ)$ fixing $p = (0,0,1)$ and acting trivially on $p^\perp/\bZ p$. A calculation similar to the one used to derive \eqref{eqn:theta-lattice-form} shows that
$$
    u = e^{D'}
$$
for some $D'\in H^2(X,\bZ)$. Since $u$ preserves the Hodge structure, $u(e^B\sigma_X) = e^{D'+B}\sigma_X$ is a scalar multiple of $e^B\sigma_X$. Comparing the degree-two components shows that this scalar is $1$, and the degree-four components then give
$$
    D'\cdot\sigma_X = 0.
$$
Hence $D'\in\NS(X)$. Thus, there exists $L\in\Pic(X)$ such that $c_1(L) = D'$. Since tensoring by $L$ acts as $e^{D'}$ on the Mukai lattice, we obtain
$$
    \theta = ((-\otimes L)\circ f^*_\alpha)^H.
$$
This completes the proof.
\end{proof}

The following lemma is the twisted analogue of \cite{DHKK14}*{Lemma~2.13}.

\begin{lemma}
\label{lemma:entropy-twisted-line-bundle}
Let $X$ be a smooth complex projective variety, $\alpha\in\mathrm{Br}(X)$ be a Brauer class, and $L\in\Pic(X)$ be a line bundle. Then the autoequivalence
$$
    -\otimes L\colon
    \Db(X,\alpha)\longrightarrow\Db(X,\alpha)
$$
satisfies $h_\cat(-\otimes L) = 0$.
\end{lemma}
\begin{proof}
Choose a split generator $G$ of $\Db(X,\alpha)$, whose existence follows, for example, from \cite{Toe12}*{Theorem~1.2}. By \cite{DHKK14}*{Theorem~2.6},
$$
    h_\cat(-\otimes L)
    = \lim_{n\to\infty}\frac{1}{n}\log\left(
        \sum_{k\in\bZ}\dim\Ext^k_\alpha(
            G,G\otimes L^n
        )
    \right).
$$
Set $P\colonequals\mathbf{R}\mathcal{H}om_\alpha(G,G)$. Note that the source and target are both $\alpha$-twisted, so $P$ is an untwisted perfect complex. Hence,
$$
    \mathbf{R}\Hom_\alpha(G,G\otimes L^n)
    \cong
    \mathbf{R}\Gamma(X,P\otimes L^n).
$$
See \cite{CaldararuThesis}*{Propositions~2.3.2 and 2.3.12}. Therefore,
$$
    \Ext^k_\alpha(G,G\otimes L^n)
    \cong
    \bH^k(X,P\otimes L^n).
$$

Let $d=\dim X$. For every coherent sheaf $\cF$ on $X$ and every $i$, we have
$$
    \dim H^i(X, \cF\otimes L^n)=O(n^d)
    \quad\text{as}\quad n\to\infty.
$$
See, for example, \cite{Lazarsfeld}*{Example~1.2.33}. Applying the spectral sequence
$$
    E_2^{p,q} = H^p(X,\mathcal{H}^q(P)\otimes L^n)
    \Longrightarrow
    \bH^{p+q}(X,P\otimes L^n)
$$
and using the fact that $P$ has only finitely many nonzero cohomology sheaves, we obtain
$$
    \sum_{k\in\bZ}\dim\Ext^k_\alpha(G,G\otimes L^n)
    = O(n^d).
$$
It follows that $h_\cat(-\otimes L)=0$.
\end{proof}



\bibliography{CategoricalNielsen}
\bibliographystyle{alpha}

\ContactInfo
\end{document}